\documentclass{amsart}

\usepackage{graphicx}
\usepackage[utf8]{inputenc}
\usepackage{nicefrac}
\usepackage{mathabx}
\usepackage[alphabetic]{amsrefs}
\usepackage{amsmath}
\usepackage{esint}

\usepackage{color}
\newtheorem{theorem}{Theorem}[section]
\newtheorem{lemma}[theorem]{Lemma}

\theoremstyle{definition}
\newtheorem{definition}[theorem]{Definition}

\newtheorem{prop}[theorem]{Proposition}
\newtheorem{cor}[theorem]{Corollary}

\allowdisplaybreaks
\theoremstyle{remark}
\newtheorem{remark}[theorem]{Remark}

\numberwithin{equation}{section}

\newcommand{\V}{\mathcal{V}}
\newcommand{\W}{\ensuremath{\mathcal{W}}}
\newcommand{\N}{\ensuremath{\mathbb{N}}}

\newcommand{\Sph}{\ensuremath{\mathbb{S}}}
\newcommand{\Hyp}{\ensuremath{{\mathbb{H}^2}}{}}
\newcommand{\R}{\ensuremath{\mathbb{R}}}
\newcommand{\E}{\ensuremath{\mathcal{E}}}

\newcommand*\diff{\mathop{}\!\mathrm{d}}

\newcommand{\Aring}{\ensuremath{\mathring{A}}{}}

\newcommand{\LH}{\mathcal{L}_{\mathbb{H}^2}}

\newcommand{\diam}{\mathrm{diam}}

\newcommand{\dx}{\;\mathrm{d}x}

\renewcommand{\max}{\mathrm{max}}

\makeatletter
\@namedef{subjclassname@2020}{%
  \textup{2020} Mathematics Subject Classification}
\makeatother

\begin{document}

\title[The Willmore Flow of Tori of Revolution]{The Willmore Flow of Tori of Revolution}

% Info Anna
 \author{Anna Dall'Acqua}
\address{Institut für Analysis, Universität Ulm, 89069 Ulm}
\email{anna.dallacqua@uni-ulm.de}
\author{Marius Müller}
\address{Institut für Analysis, Albert-Ludwigs-Universität Freiburg, 79104 Freiburg}
\email{marius.mueller@math.uni-freiburg.de}
\author{Reiner Schätzle}
\address{Fachbereich Mathematik, Eberhard-Karls-Universit\"at T\"ubingen, Auf der Morgenstelle 10, 72076 T\"ubingen, Germany}
\email{schaetz@everest.mathematik.uni-tuebingen.de}
\author{Adrian Spener}
\address{Institut für angewandte Analysis, Universität Ulm, 89069 Ulm}
\email{adrian.spener@alumni.uni-ulm.de}

\thanks{\textbf{Acknowledgment.}   The research of the first author was supported by the Deutsche Forschungsgemeinschaft (DFG, German Research Foundation)- project no. 404870139. The second author is supported by the LGFG Grant (Grant no. 1705 LGFG-E). The fourth author is supported by the Deutsche Forschungsgemeinschaft - project no. 355354916. The authors would like to thank Fabian Rupp for helpful discussions and the referee for helpful suggestions.}

%    General info


\begin{abstract}
We study long-time existence and asymptotic behavior for the $L^2$-gradient flow of the Willmore energy, under the condition that the initial datum is a torus of revolution. We show that if an initial datum has Willmore energy below $8\pi$ then the solution of the Willmore flow converges for $t \to \infty$ to the Clifford Torus, possibly rescaled and translated. The energy threshold of $8 \pi$ turns out to be optimal for such a convergence result. We give an application to the conformally constrained Willmore minimization problem.
\end{abstract}
\maketitle

%\today

%\tableofcontents
\vspace{-0.5cm}

\section{Introduction}

Let $f\colon \Sigma \to \R^3$ be a smooth immersion of a two-dimensional manifold without boundary. Its \emph{Willmore energy} is 
\begin{equation}
 \label{eq:DefWillEnergy}
 \W(f) = \frac{1}{4} \int_\Sigma {|\vec{H}|}^2 \diff \mu,
\end{equation}
{where $\vec{H}$ denotes the mean curvature vector and $\diff \mu$ the induced Riemannian measure. Its critical points are called \emph{Willmore immersions} and satisfy}
\begin{equation}
 \label{eq:WillmoreEquation}
\Delta \vec{H} + Q(\Aring) \vec{H} = 0, 
\end{equation}
where $\Delta$ denotes the Laplace-Beltrami operator, \Aring{} is the trace-free second fundamental form and $Q$ is quadratic in \Aring{} (see \eqref{eq:aring}). If $f(\Sigma)$ is orientable (or two-sided, which is equivalent in $\R^3$) 
% Notice that if $f(\Sigma)$ is orientable, 
then $\vec{H}=(\kappa_1+\kappa_2) \vec{N}$ with $\kappa_1,\kappa_2$ the principal curvatures of $f(\Sigma)$ and $\vec{N}$ a smooth normal vector-field.  The $L^2$-gradient flow of the Willmore functional with given initial datum $f_0$, a smooth immersion, is 
\begin{equation}
 \label{eq:WillFlowColl}
 \partial_t f = - (\Delta \vec{H} + Q(\Aring) \vec{H})
\end{equation}
with $f(t=0) = f_0$. This fourth order quasilinear geometric evolution equation has been extensively studied in \cites{KS1, KS2}, where a blow-up criterion is formulated. 
With the aid of this criterion the same authors proved in \cite{KS3} long-time existence and convergence for the \emph{flow of spherical immersions}  under the assumption that the initial immersions $f_0$, $f_0 : \mathbb{S}^2 \rightarrow \mathbb{R}^3$, satisfies $\mathcal{W}(f_0) < 8\pi$. The energy threshold of $8\pi$ is shown to be sharp in \cite{Blatt} for the convergence of spherical immersions. 

In the classical work \cite{Mayer01anumerical} the Willmore flow is studied numerically, not only for spheres but also for surfaces of different genus, such as tori. See also \cite{BarGarNu} for other numerical examples. In \cite[Sec. 8.1]{Mayer01anumerical} it is stated that the flow converges for all tori that the authors looked at, which was astounding as this behavior differs fundamentally from the surface diffusion flow, where the hole of all initial tori seems to close and the curvature blows up, cf. \cite{MayerSurfDif,BarGarNu}. Our goal is to understand analytically what happens to tori along the Willmore flow. In this article we only look at the special case of tori of revolution.% ({see Definition \ref{def:toriofrevolution}}).

\begin{definition}\label{def:toriofrevolution}
In the sequel we identify $\mathbb{S}^1 = \mathbb{R}/ \mathbb{Z}$ and set $\mathbb{H}^2 := \mathbb{R} \times (0,\infty)$. We call an immersion $f: \mathbb{S}^1 \times \mathbb{S}^1 \rightarrow \mathbb{R}^3 $ a \emph{torus of revolution} if there exists an immersed curve $\gamma \in C^\infty(\mathbb{S}^1, \mathbb{H}^2)$, {$\gamma=(\gamma^{(1)},\gamma^{(2)})$,} such that
\begin{equation}\label{eq:D1}
f(u,v) = \begin{pmatrix}
\gamma^{(1)}(u) \\ \gamma^{(2)}(u) \cos(2\pi v) \\ \gamma^{(2)}(u) \sin(2\pi v) 
\end{pmatrix}.
\end{equation}
We call $\gamma$ \emph{profile curve} and we will frequently denote $f$ as in \eqref{eq:D1} by $F_\gamma$.
\end{definition}

An essential element in our argument is that %It is especially remarkable that 
the property of being a torus of revolution is preserved along the Willmore flow. Hence the evolution by Willmore flow can also be regarded as a time evolution of the profile curves. In the arguments to come we will take advantage of an interplay between the revolution symmetry and the %aforementioned 
blow-up-criterion developed in \cite{KS1,KS2}. With this technique we have identified a geometric quantity whose boundedness ensures convergence. {This quantity is the \emph{hyperbolic length} of the profile curves given by
\begin{equation*}
    \mathcal{L}_{\mathbb{H}^2}(\gamma) := \int_{\mathbb{S}^1} \frac{|\gamma'(x)|}{\gamma^{(2)}(x)} \dx, \quad  \gamma \in C^\infty(\mathbb{S}^1, \mathbb{R} \times (0,\infty) ).
\end{equation*}}

Strikingly, the \emph{hyperbolic geometry} of the curve evolution is decisive for the convergence behavior. We recall that the hyperbolic plane $\mathbb{H}^2= \mathbb{R} \times (0,\infty)$ is endowed with the metric $g_{(x,y)} = y^{-2}  \; \mathrm{d}x \; \mathrm{d} y$.

Now we can state our main convergence criterion

\begin{theorem}\label{thm:boundhypmain}
Let $f: [0,T) \times \mathbb{S}^1 \times \mathbb{S}^1 \rightarrow \mathbb{R}^3$ be a maximal evolution by Willmore flow such that $f(0)$ is a torus of revolution. Then $f(t)$ is a torus of revolution for all $t \in [0,T)$. Suppose that $(\gamma(t))_{t \in [0,T)}$ is a collection of profile curves of  $f(t)$. If 
\begin{equation}
\liminf_{ t \rightarrow T} \mathcal{L}_{\mathbb{H}^2} ( \gamma(t)) < \infty, 
\end{equation}  
then $T = \infty$ and the Willmore flow converges {(up to reparametrizations)} in $C^k$ for all $k$ to a Willmore torus of revolution $f_\infty$. 
\end{theorem}

We remark that the concept of $C^k$-convergence that we impose is a \emph{geometric} one, see Appendix \ref{app:smoothconv} (Definition \ref{def:Clconv}) for details. {From now on, the term $C^k$-convergence is understood up to reparametrizations as in Definition \ref{def:Clconv}}.

That the \emph{hyperbolic geometry} of the profile curve plays a role is not surprising --  there is an interesting correspondence between the Willmore energy of tori of revolution and the hyperbolic elastic energy of curves, observed in \cite{LangerSingerWillmore}. With this correspondence one can for example show the Willmore conjecture for tori of revolution, cf. \cite{LangerSinger1}. Other applications of this relationship include \cite{MR2480063,Mandel2018}. To the authors' knowledge, this is the first time that this correspondence is used in a problem depending on time.

The main question now is {to identify which}
initial data generate evolutions with bounded hyperbolic length. {It turns out that the same energy threshold of $8 \pi $ needed for spherical immersions (see  \cite{KS3}) is needed in the case of tori of revolution.}

\begin{theorem}
 \label{thm:MainThmPrecisely}
 Let $f_0 \colon \mathbb{S}^1 \times \mathbb{S}^1 \to \R^3$ be a torus of revolution satisfying $\W(f_0) \leq 8 \pi$. Let $f\colon[0,T) \times (\mathbb{S}^1 \times \mathbb{S}^1) \to \R^3$ evolve by the Willmore flow with initial datum $f_0$. Then $T=  \infty$ and $f$ converges in $C^k$ for all $k \in \mathbb{N}$, to the Clifford torus, {possibly rescaled and translated in the direction (1,0,0).}
\end{theorem} 

Here the Clifford Torus is the surface of revolution given by 
\begin{equation}
 \label{eq:defClifford}
 (u,v) \mapsto \left(\tfrac{1}{\sqrt{2}} \sin (2 \pi u),  \left(1 + \tfrac{1}{\sqrt{2}} \cos (2\pi u)  \right) \cos (2\pi v), \left( 1 + \tfrac{1}{\sqrt{2}} \cos (2\pi u) \right) \sin (2\pi v)   \right).
\end{equation}
Notice that it is not important which parametrization we choose since $C^k$-convergence is a geometric concept. The Clifford torus 
arises from stereographic projection of the minimal surface $\frac{1}{\sqrt{2}} (\mathbb{S}^1 \times \mathbb{S}^1) \subset \mathbb{S}^3.$  From the solution \cite{MR3152944} of the famous Willmore conjecture we know that the Clifford torus is the global minimum of the Willmore energy among tori in $\R^3$ and the unique minimum modulo smooth conformal transformations (of $\R^3$) and reparametrizations. Our method relies on a \emph{gap theorem} for Willmore tori of revolution consequence of \cite{AdrianMarius}, see Proposition \ref{prop:willtori}. This relates to the findings in \cite{Mondino}.

The convergence result in Theorem \ref{thm:MainThmPrecisely} holds up to surprisingly little invariances. It is often expected that such convergence results can only be shown up to invariances of the Willmore energy - i.e. reparametrizations and conformal transformations. The fact that we do not have to apply conformal transformations along the flow to achieve convergence is explained by the use of a Lojasiewicz-Simon gradient inequality. This inequality is a purely analytical tool, so the invariances will not play a role.  For the limit immersion, we can rule out all conformal transformations that break the rotational symmetry and even more -- symmetry-preserving Möbius inversions can also be ruled out due to the fact that they are not invariances of the Willmore flow equation.
What remains is just scaling and translation in direction $(1,0,0)$. This is not surprising since both transformations preserve the symmetry we consider and also preserve solutions of the Willmore flow equation, possibly rescaling appropriately in time.

We also prove that the energy threshold of $8\pi$ is sharp by constructing explicit non-convergent evolutions  {with initial data $f_0$ satisfying $\mathcal{W}(f_0)>8\pi$.} There are multiple reasons why this number could be a universal threshold for any genus.   The most striking is the inequality of Li and Yau that shows that {immersions} of Willmore energy below $8\pi$ are embeddings, cf. \cite[Thm 6]{LiYau}. Another property is that the metric of tori of energy $\leq 8\pi- \delta, \delta > 0,$ is `uniformly controlled up to Möbius transformations and reparametrizations, see {\cite[Thm 1.1]{Schaetzle}} for details. As pointed out in \cite[p.282]{Simon}, \cite{KuwertLi}, there exist surfaces of arbitrary genus with energy below $8\pi$.

As already announced, we also show optimality of the energy bound of $8\pi$. 
\begin{theorem}\label{thm:optii}
For any $\varepsilon >0 $ there exists a torus of revolution {$f_0 : \mathbb{S}^1 \times \mathbb{S}^1 \rightarrow \mathbb{R}^3$} such that $\mathcal{W}(f_0) < 8\pi + \varepsilon$ and the maximal Willmore flow $(f(t))_{t \in [0,T)}$ develops a singularity (in finite or infinite time). {More precisely, one of the following phenomena occurs 
\begin{enumerate}
\item[(1)] (Concentration of curvature) The second fundamental form $(||A(t)||_{L^\infty(\Sigma)})_{t \in [0,T)}$ is unbounded. This singularity can occur in finite or infinite time.
 \item[(2)] (Diameter Blow-Up in infinite time) $T= \infty$ and $\lim_{t\rightarrow \infty} \mathrm{diam} (f(t))( \mathbb{S}^1 \times \mathbb{S}^1) = \infty$.
\end{enumerate}
In both cases the Willmore flow cannot converge in $C^2$. }
\end{theorem}
The singular behavior as described in Theorem \ref{thm:optii} will actually occur for each initial immersion $F_\gamma$, as in Definition \ref{def:toriofrevolution}, with $\gamma$ a curve of \emph{vanishing total curvature}, cf. \eqref{eq:totalcurv}. {This gives }
a class of singular examples for the Willmore flow. The total curvature also plays a significant role in earlier constructions of singular examples, see 
\cite{Blatt} for $\Sigma = \mathbb{S}^2$.

As a consequence of our main result we are able to show that each rectangular conformal class contains a torus of revolution of energy below $8\pi$. This result has far-reaching consequences for the minimization of the Willmore energy with fixed conformal class, studied for example in \cite{KS4}.
In this article the authors show that minimizers in a given conformal class exist under the condition that the class contains an element of Willmore energy below $8\pi$. By our result this condition is satisfied for every \emph{rectangular} conformal class.

{This paper is organized as follows. In Section 2 we fix the notation and collect some useful facts on elastic curves in the hyperbolic plane and on tori of revolution. Section 3 exploits the consequences of the initial datum being a torus of revolution for the symmetries properties of the evolution, for the possible singularities and the limit. It also  contains the proofs of the main results and of the optimality results. 
In the last section we give the application on existence of tori of revolution with energy below $8 \pi$ in each conformal class. Some useful results on smooth convergence (see Definition \ref{def:smocon} below) and the Willmore flow are collected in the appendix.}

%%%%%%%%%%%%%%%%%%%%%%%%%%%%%%%%%%%%%%%%%%%%%%%%%
%%%%%%%%%%%%%%%%%%%%%%%%%%%%%%%%%%%%%%%%%%%%%%%
\section{Geometric preliminaries}
%%%%%%%%%%%%%%%%%%%%%%%%%%%%%%%%%%%%%%%%%
%%%%%%%%%%%%%%%%%%%%%%%%%%%%%%%%%%%%%%%%%%%
\subsection{Notation}
%%%%%%%%%%%%%%%%%%%%%%%%%%%%%%%%%%%%%%%%%%%
%%%%%%%%%%%%%%%%%%%%%%%%%%%%%%%%%%%%%%%%%%%
We first recall some basic definitions from differential geometry. Let $\Sigma$ be a two-dimensional smooth manifold and $f\colon \Sigma \to \R^{{n}}$ be a smooth immersion. In this paper all manifolds are assumed to have no boundary. {If we talk about tori of revolution we need to impose the restriction that $n=3$, but we will also discuss some results on the Willmore flow that remain valid in any codimension, i.e. for all $n \geq 3$. Let $g$ be the induced Riemannian metric and $\nabla$ the Levi-Civita connection on $\Sigma$, and denote the set of smooth vector fields on $\Sigma$ by $\V(\Sigma)$. For $X \in \V(\Sigma)$ and $h \in \mathcal{C}^\infty(\Sigma, \R^{{n}})$ we define $D_X h \in \mathcal{C}^\infty(\Sigma, \R^{{n}})$ as follows 
$$D_X h := \sum_{i = 1}^{{n}} X(h_i) \vec{e}_i, \quad \textrm{whenever} \quad h = \sum_{i = 1}^{{n}} h_i \vec{e}_i \in C^\infty(M;\mathbb{R}^{{n}}), $$
and $\{ \vec{e}_1, \vec{e}_2 , \vec{e}_3,...,\vec{e}_n \}$ is the canonical basis of $\mathbb{R}^{{n}}$ (see also Appendix \ref{app:tensor}). The second fundamental form of $\Sigma$ is $A \colon \V(\Sigma) \times \V(\Sigma) \to \mathcal{C}^\infty(\Sigma, \R^{{n}})$, given by
\begin{equation}
 \label{eq:Def2FF}
 A(X,Y):= D_X(D_Yf)-D_{\nabla_XY}f .\end{equation}
We remark that for all $p \in \Sigma$ one has $A_p(X,Y) \in df_p(T_p\Sigma)^\perp$, we say it takes values in the \emph{normal bundle}. Moreover $A_p(X,Y)$ only depends on $X(p)$,$Y(p)$. Its trace-free part $\Aring$ is given by
\[
 \Aring(X,Y) := A(X,Y) - \frac{1}{2}g(X,Y) \vec{H},
\]
where the mean curvature vector $\vec{H}$ is the trace of the bilinear form \eqref{eq:Def2FF} and can be computed by 
\[\vec{H}(p) = A(e_1,e_1)+A(e_2,e_2),
\]
with $\{e_1,e_2\}$ being an orthonormal basis of $T_p\Sigma$. Similarly (see Appendix
\ref{sec:AppendixGeometry} for details) we have 
\[|{A}|^2 = \sum_{i,j=1}^2 \langle A(e_i, e_j), A(e_i,e_j)\rangle_{\R^{{n}}}.\]

With these definitions we may introduce the Willmore flow of a smooth immersion $f_0 \colon \Sigma \to \R^{{n}}$. We say that a smooth family of smooth immersions $f\colon [0,T)\times \Sigma \to \R^{{n}}$, where $T>0$, evolves by the \emph{Willmore flow} with initial datum $f_0$, if $f$ satisfies
\begin{equation} \label{eq:WillmoreFlow}
 \partial_t f = - (\Delta\vec{H} + Q(\Aring) \vec{H}) \qquad \text{in }(0,T)\times \Sigma
\end{equation}
with $f(t=0) = f_0$. Here, $\Delta$ denotes the \emph{normal Laplacian}, i.e. for an orthonormal basis $\{e_1,e_2\}$ is a basis of $T_p\Sigma$ with respect to $f(t,\cdot)^*g_{\R^{{n}}}$ one has
\begin{equation*}
    \Delta \vec{H} = \sum_{i = 1}^2 (\nabla^{\perp})^2 \vec{H}(e_i,e_i),  
\end{equation*}
where $\nabla^\perp_X Y = (D_X Y)^\perp$ (cf. \eqref{eq:normalconn}, \eqref{eq:tnsderiv} for details). With the same notation as above, the quadratic operator $Q$ is given by
\begin{equation}\label{eq:aring}
\big(Q(\Aring)\vec{H}\big)(t,p) = \sum_{i,j=1}^2\langle \Aring(e_i, e_j),\vec{H}\rangle_{\R^{{n}}}\Aring(e_i,e_j).
\end{equation}
 Since \eqref{eq:WillmoreFlow} is well-posed for smooth initial immersions $f_0$ (see \cite[Prop. 1.1]{KS1}) we will always assume that the evolution is maximal, i.e.~non-extendable in the class of smooth immersions.

To study the behavior of $f(t)$ as $t \to T$ we use the following notion of smooth convergence on compact sets from \cite[Thm 4.2]{KS2}, see also \cite{Breuning} and Appendix \ref{app:smoothconv}.

\begin{definition}\label{def:smocon}(Smooth convergence of immersions) 
Let $\Sigma$ and $\widehat{\Sigma}$ be smooth two-dimensional manifolds, $(f_j)_{j = 1}^\infty \colon \Sigma \to \mathbb{R}^{{n}}$ and $\widehat{f}\colon \widehat{\Sigma} \to \mathbb{R}^{{n}}$  be smooth immersions. Define 
\begin{equation}\label{eq:sigmaj}
\widehat{\Sigma}(m) := \{ p \in \widehat{\Sigma} : |\widehat{f}(p) | < m \}, \qquad m \in \N. 
\end{equation}
 We say that $f_j$ \emph{converges to  $\widehat{f}$ 
smoothly on compact subsets of $\mathbb{R}^{{n}}$} if for each $j \in \mathbb{N}$ there exists a diffeomorphism $\phi_j \colon \widehat{\Sigma}(j) \rightarrow U_j$ for some open $U_j \subset \Sigma$, and a normal vector field $u_j \in C^\infty( \widehat{\Sigma}(j), \mathbb{R}^n)$ satisfying 
 \begin{equation}\label{eq:graph}
 f_j \circ \phi_j  = \widehat{f} + u_j \quad \textrm{on} \; \; \widehat{\Sigma}(j) 
\end{equation}  
as well as
 $||(\widehat{\nabla}^\perp)^k u_j||_{L^\infty( (\widehat{\Sigma}(j)))} \rightarrow 0 $ as $j \rightarrow \infty$ for all $k \in \mathbb{N}_0$. Here $\widehat{\nabla}$ is the Levi-Civita connection on $(\widehat{\Sigma},g_{\widehat{f}})$ and $(\widehat{\nabla}^\perp)^k u_j$ is defined as in Appendix \ref{app:tensor}. Additionally, we require that for each $R> 0 $ there exists $j(R) \in \mathbb{N}$ such that $j \geq j(R)$ implies that $f_j^{-1}(B_R(0)) \subset U_j$.
\end{definition}

We exploit a fundamental correspondence between the Willmore energy of tori and the elastic energy of curves in the hyperbolic plane already used in several works since its observation in \cite{LangerSingerWillmore}.

%%%%%%%%%%%%%%%%%%%%%%%%%%%%%%%%%%%%%%%%%%%%%%%%%
\subsection{Curves in the hyperbolic plane}
\label{sec:HypCurves}
%%%%%%%%%%%%%%%%%%%%%%%%%%%%%%%%%%%%%%%%%%%%%%%%

We consider the hyperbolic half-plane $\Hyp = \{(x^{(1)},x^{(2)}) \in \R \times (0,\infty)\}$ endowed with the metric
\begin{equation*}
g_\Hyp(v,w) = \frac{1}{z^2} \langle v,w\rangle_{\R^2} \quad v, w \in T_z \Hyp
\end{equation*}
and denote $|{v}|_\Hyp = \sqrt{g_\Hyp (v,v)}$, $v \in T_z\Hyp$. 
For a smooth  immersed curve $\gamma=(\gamma^{(1)}, \gamma^{(2)})$ in $\Hyp$, $\gamma \in \mathcal{C}^\infty(\Sph^1, \Hyp)$, the length is {as in the introduction} given by
{
\begin{equation}\label{eq:lengthhyp}
\LH(\gamma) := \int_0^1 \frac{|\gamma'(x)|_{\mathbb{R}^2}}{\gamma^{(2)}(x)} \diff x = \int_0^1 \diff{s} ,
\end{equation}}
 where $\diff {s} = {|\partial_x \gamma|}_{\mathbb{H}^2} \diff x$ denotes the arc length parameter and the derivative with respect to $x$ is abbreviated with the prime. As usual, $\partial_s = \frac{\partial_x}{|\partial_x \gamma|_{\mathbb{H}^2}}$ denotes the arc length derivative. The curvature vector field of $\gamma$ is given by
\begin{equation}\label{eq:curvhyp}
\kappa[\gamma] = \nabla_s \partial_s \gamma = \begin{pmatrix} \partial_s^2 \gamma^{(1)} - \frac{2}{\gamma^{(2)}} \partial_s \gamma^{(1)} \partial_s \gamma^{(2)} \vspace{.2cm}\\ \partial_s^2 \gamma^{(2)} + \frac{1}{\gamma^{(2)}} ( ( \partial_s \gamma^{(1)} )^2 - (\partial_s \gamma^{(2)})^2 ) \end{pmatrix}  
\end{equation}
as an element of $T_z \mathbb{H}^2$ \cite[(12)]{DS17}. Here $\nabla_s = \frac{\nabla}{\diff s}$ denotes the covariant derivative along $\gamma$ with respect to the Levi-Civita connection on $\mathbb{H}^2$. We write $\kappa = \kappa[\gamma]$ if the curve is clear from the context. 
The \emph{elastic energy} $\E$ of $\gamma$ is then defined to be
\begin{equation*}
\E(\gamma) := \int_\gamma |\kappa|_{\mathbb{H}^2}^2 \diff {s}.
\end{equation*}
Its critical points are called \emph{free hyperbolic elastica} and satisfy 
\begin{equation*}
(\nabla_{s}^\perp)^2 \kappa + \frac{1}{2} |\kappa|_{\mathbb{H}^2}^2 \kappa - \kappa = 0,
\end{equation*}
where $\nabla_{s}^\perp \eta = \nabla_s \eta - \langle \nabla_s \eta, \partial_s \gamma \rangle_{\mathbb{H}^2} \partial_s \gamma$ is the covariant derivative on the normal bundle of $\gamma$. 

We collect some results connecting the length and the elastic energy of smooth closed curves in the hyperbolic plane. 

\begin{theorem}[{\cite[Thm 5.3]{AdrianMarius}}]\label{thm:reilly}
For each $\varepsilon > 0 $ there exists $c(\varepsilon) > 0 $ such that 
\begin{equation*}
\frac{\E(\gamma)}{\mathcal{L}_{\mathbb{H}^2}(\gamma)} \geq c(\varepsilon) 
\end{equation*}
for all immersed and closed curves $\gamma \in C^\infty(\mathbb{S}^1,\mathbb{H}^2)$  such that $\mathcal{E}(\gamma) \leq 16- \varepsilon$.
\end{theorem}
Note that the energy threshold of $16$ is sharp for this result, cf. \cite{AdrianMarius}.  

We also fix the notion of the Euclidean length of the curve $\gamma \colon \mathbb{S}^1 \to \mathbb{H}^2 \subset \mathbb{R}^2$, which is given by $\mathcal{L}_{\mathbb{R}^2}(\gamma)$. We also consider the \emph{Euclidean curvature} of $\gamma: \mathbb{S}^1 \rightarrow \mathbb{R}^2$ which we will denote by
$\vec{\kappa}_{euc}[\gamma] := \frac{1}{|\gamma'|} \frac{d}{dt} \frac{\gamma'}{|\gamma'|}$
 and the \emph{Euclidean scalar curvature}
$\kappa_{euc}[\gamma] := \frac{1}{|\gamma'|^2}\langle \gamma'', n \rangle_{\mathbb{R}^2}$. To finish this section we discuss some relations between Euclidean and hyperbolic length.

 \begin{lemma} 
Let $\gamma \in C^\infty(\mathbb{S}^1, \mathbb{H}^2)$ and $a,b \in [0,1]$. Then 
\begin{equation}\label{eq:hyplen}
\gamma^{(2)}(b) e^{-\LH(\gamma)} \leq \gamma^{(2)}(a) \leq \gamma^{(2)}(b)e^{\LH(\gamma)}
\end{equation} 
and
\begin{equation}\label{eq:hyplen2}
\mathcal{L}_{\mathbb{H}^2}(\gamma) \geq \frac{\mathcal{L}_{\mathbb{R}^2}(\gamma)}{\sup_{\mathbb{S}^1} \gamma^{(2)} }
\end{equation} 
\end{lemma}

\begin{proof}
For $\gamma, a, b$ as in the statement we find by \eqref{eq:lengthhyp}
\begin{equation*}
\LH(\gamma) \geq \int_a^b \frac{|(\gamma^{(2)})'|}{\gamma^{(2)}} \diff x \geq |\log \gamma^{(2)}(b)  - \log \gamma^{(2)}(a)|,
\end{equation*}
and therefore 
$\log \gamma^{(2)}(b) - \LH(\gamma) \leq \log \gamma^{(2)}(a) \leq \log \gamma^{(2)}(b) + \LH(\gamma)$. Taking exponentials \eqref{eq:hyplen} follows.
For \eqref{eq:hyplen2} we simply estimate 
\begin{equation*}
\LH(\gamma) = \int_{\mathbb{S}^1} \frac{|\gamma'(u)|}{\gamma^2(u)} \; \mathrm{d} u  \geq \frac{1}{\sup_{\mathbb{S}^1}  \gamma^{(2)} } \int_{\mathbb{S}^1} |\gamma'(u)| \; \mathrm{d}u  = \frac{\mathcal{L}_{\mathbb{R}^2}(\gamma)}{\sup_{\mathbb{S}^1} \gamma^{(2)}}. 
\end{equation*}
\end{proof}
 
%%%%%%%%%%%%%%%%%%%%%%%%%%%%%%%%%%%%%%%%%%%%%%%%%
\subsection{Tori of revolution in $\R^3$}
\label{sec:tori}
%%%%%%%%%%%%%%%%%%%%%%%%%%%%%%%%%%%%%%%%%%%%%%%%%

Here we collect some basic facts about tori of revolution. More precisely we express some geometric quantities associated to tori of revolution using only their profile curves. 
If $F_\gamma : \mathbb{S}^1 \times \mathbb{S}^1\rightarrow \mathbb{R}^3$ is chosen as in Definition \ref{def:toriofrevolution} we can compute the first fundamental form with respect to the local coordinates $(u,v)$ of $\mathbb{S}^1 \times \mathbb{S}^1$. This yields the associated surface measure on the Riemannian manifold {$(\mathbb{S}^1 \times \mathbb{S}^1 , g=  F_\gamma^* g_{\mathbb{R}^3})$} given by
\begin{equation}\label{eq:areaform}
 \diff \mu_{g} = 2\pi \gamma^{(2)}(u){|\gamma'(u)|}_{\R^2} \diff u \diff v.
\end{equation}
As we have already announced, the Willmore energy of $F_\gamma$ can also be expressed only in terms of $\gamma$ using the fundamental relationship 
\begin{equation}\label{eq:willelaEnergy}
\W(F_\gamma) = \frac{\pi}{2}\E(\gamma),
\end{equation} 
see \cite{LangerSingerWillmore} and \cite[Thm 4.1]{DS18}. Moreover, let $\kappa$ be the hyperbolic curvature vector field of $\gamma$ in $\mathbb{H}^2$. Then
\begin{equation}\label{eq:willelaGrad}
-\langle(\nabla_{s}^\perp)^2 \kappa + \frac{1}{2} |\kappa|_{\mathbb{H}^2}^2 \kappa - \kappa, n \rangle_\Hyp =  {2} (\gamma^{(2)})^4 ( \Delta H + 2 H \big(\tfrac14 H^2 - K\big) ),
\end{equation}
where $n  = (-\partial_s \gamma^{(2)},\partial_s \gamma^{(1)})$ is the normal vector field along $\gamma$ (see \cite[Thm 4.1]{DS18}). 
In particular, $F_\gamma$ is a Willmore torus of revolution if and only if $\gamma$ is a hyperbolic elastica. 
{In Appendix \ref{sec:AppendixGeometry} we discuss the relationship between \eqref{eq:willelaGrad} and \eqref{eq:WillmoreEquation}.}

An immediate consequence of \cite[Proposition 6.5]{AdrianMarius} (that builds on findings in \cite{LangerSinger1}) is the following.

\begin{prop}[A gap theorem for Willmore tori of revolution]
\label{prop:willtori}
 Let $f: \mathbb{S}^1 \times \mathbb{S}^1 \rightarrow \mathbb{R}^3$ be a {Willmore torus of revolution} that satisfies $\mathcal{W}(f) \leq 8 \pi$. Then $f$ is, up to reparametrization, the Clifford torus possibly rescaled and translated in direction $(1,0,0)^T$. 
\end{prop}

\begin{proof}
Let $f =F_\gamma$ be as in the statement with profile curve $\gamma \in C^\infty(\mathbb{S}^1, \mathbb{H}^2)$. From \eqref{eq:willelaGrad} we know that $\gamma$ is a hyperbolic elastica. From \eqref{eq:willelaEnergy} we can conclude that $\E(\gamma) \leq 16$. By \cite[Proposition 6.5]{AdrianMarius} we obtain that $\gamma$ has to coincide  (up to reparametrization) with the profile curve of the Clifford torus up to isometries of $\mathbb{H}^2$. 
This however implies that $f$ is, up to reparametrization, the Clifford torus possibly rescaled and translated in direction $(1,0,0)$. 
\end{proof}

Another important quantity for our discussion is the second fundamental form $A[F_\gamma]$, which we will also express in terms of $\gamma$. A property which we will later make extensive use of is the fact that for a torus of revolution $f= F_\gamma$, $|A[F_\gamma]|^2 \in C^\infty(\mathbb{S}^1 \times \mathbb{S}^1)$ is a function that depends only on $u$ (parameter that describes the profile curve) and not on $v$ (parameter that describes the revolution). This is the reason why curvature concentration is `passed along' the revolution. We will describe this more precisely in Section \ref{sec:rotsymconc}.
{For this section it is enough to observe by a direct computation (cf. \cite[p.118]{DS18}) that with respect to the normal $N_{F_\gamma} = \frac{ \partial_u F_\gamma \times \partial_v F_\gamma}{|\partial_u F_\gamma \times \partial_v F_\gamma|}$
the principal curvatures are given by 
$\kappa_1[F_\gamma](u,v) = - \kappa_{euc}[\gamma](u)$ and $\kappa_2[F_\gamma](u,v) = \frac{(\gamma^{(1)})'(u)}{|\gamma'(u)|\gamma^{(2)}(u)}$.
}
{This at hand, one can derive a useful bound for the length of the profile curve in terms of surface quantities.
\begin{lemma}\label{lem:spurbound}
 Suppose that $f= F_\gamma : \mathbb{S}^1 \times \mathbb{S}^1 \rightarrow \mathbb{R}^3$ is a torus of revolution with profile curve $\gamma$. Then 
 \begin{equation*}
     \mathcal{L}_{\mathbb{R}^2}(\gamma)  \leq  \mu_{g_f}(\mathbb{S}^1 \times \mathbb{S}^1)^\frac{1}{2} \mathcal{W}(f)^\frac{1}{2}.
 \end{equation*}
\end{lemma}
\begin{proof}
We may without loss of generality assume that $\gamma$ is parametrized with constant velocity, i.e. $|\gamma'| = \mathcal{L}_{\mathbb{R}^2}(\gamma)=: L $. Recall from Appendix  \ref{sec:AppendixGeometry} that 
$\vec{H}(u,v) = (\kappa_1(u,v)+ \kappa_2(u,v))N_f(u,v)$, where
\begin{equation*}
    N_f(u,v) = \frac{1}{\mathcal{L}_{\mathbb{R}^2}(\gamma) }\begin{pmatrix}
     (\gamma^{(2)})'(u) \\ - (\gamma^{(1)})'(u) \cos(2\pi v)  \\ - ( \gamma^{(1)})'(u) \sin(2\pi v) 
    \end{pmatrix} \quad \mbox{ with }u,v \in \mathbb{S}^1 .
\end{equation*}
We show next that 
\begin{equation}\label{eq:2L}
    -2 L = \int_{\mathbb{S}^1 \times [0, \frac{1}{2}]} \vec{H} \cdot e_3 \; \mathrm{d}\mu_{g_f}.
\end{equation}
Plugging in the quantities characterized in this section and using $(\gamma^{(1)})'^2 + ( \gamma^{(2)})'^2 = L^2$, $(\gamma^{(1)})'' ( \gamma^{(1)})' + ( \gamma^{(2)})'' (\gamma^{(2)})' = 0$ we obtain 
\begin{align*}
& \qquad \int_{\mathbb{S}^1 \times [0, \frac{1}{2}]} \vec{H} \cdot e_3 \; \mathrm{d}\mu_{g_f}  \\
& = 2 \pi \int_0^1 \int_0^{\frac{1}{2}} (\kappa_1(u,v) + \kappa_2(u,v)) (N_f(u,v) \cdot e_3) |\gamma'(u)| \gamma^{(2)}(u) \; \mathrm{d}v \; \mathrm{d}u \\ &  = 2\pi \int_0^1 \int_0^{\frac{1}{2}} \left( -\kappa_{euc}[\gamma](u) + \frac{(\gamma^{(1)})'(u)}{L \gamma^{(2)}(u)} \right)   [-(\gamma^{(1)})'(u) \sin(2\pi v)]  \gamma^{(2)}(u) \; \mathrm{d}v \; \mathrm{d}u
\\ &= -\Big[ -  \cos(2\pi v) \Big]^{\frac{1}{2}}_0  \int_0^1 \left( \frac{(\gamma^{(1)})''(\gamma^{(2)})'- (\gamma^{(2)})'' ( \gamma^{(1)})'}{L^3} +  \frac{(\gamma^{(1)})'}{L\gamma^{(2)}}\right) ( \gamma^{(1)})' \gamma^{(2)} \; \mathrm{d}u \\
 & = - 2 \frac{1}{L^3}\int_0^1 \left((\gamma^{(1)})'' (\gamma^{(1)})' (\gamma^{(2)})' - (\gamma^{(2)})'' ( \gamma^{(1)})'^2 \right) \gamma^{(2)} \; \mathrm{d}u  -\frac{2}{L}\int_0^1  (\gamma^{(1)})'^2 \; \mathrm{d}u 
 \\ & = - 2 \frac{1}{L^3}\int_0^1 \left(-(\gamma^{(2)})'' (\gamma^{(2)})'^2 - (\gamma^{(2)})'' ( \gamma^{(1)})'^2 \right) \gamma^{(2)} \; \mathrm{d}u  -\frac{2}{L}\int_0^1  (\gamma^{(1)})'^2 \; \mathrm{d}u 
 \\ &= \frac{2}{L^3} \int_0^1  (\gamma^{(2)})'' L^2 \gamma^{(2)} \; \mathrm{d}u -\frac{2}{L}\int_0^1  (\gamma^{(1)})'^2 \; \mathrm{d}u  \\
& = - \frac{2}{L} \int_0^1 (\gamma^{(2)})'^2  \; \mathrm{d} u - \frac{2}{L} \int_0^1 ( \gamma^{(1)})'^2 \; \mathrm{d}u,
\end{align*}
where we have used integration by parts in the last step. Adding up the integrands and once again using $(\gamma^{(1)})'^2 + ( \gamma^{(2)})'^2  = L^2$ we obtain \eqref{eq:2L}. From \eqref{eq:2L} and the Cauchy Schwarz inequality we also conclude 
\begin{equation*}
    2L \leq \int_{\mathbb{S}^1 \times \mathbb{S}^1} |\vec{H}| \; \mathrm{d}\mu_{g_f} \leq 2 \mathcal{W}(f)^\frac{1}{2} \mu_{g_f}( \mathbb{S}^1 \times \mathbb{S}^1) ^\frac{1}{2}. \qedhere
\end{equation*}
\end{proof}
}
A quantity which we will also study is the diameter.

\begin{lemma}\label{lem:drehdiam}
Let $f = F_\gamma : \mathbb{S}^1 \times \mathbb{S}^1 \rightarrow \mathbb{R}^3$ be a torus of revolution with profile curve $\gamma$. Then,
$$\diam(F_\gamma(\mathbb{S}^1 \times \mathbb{S}^1)) \leq \tfrac12 \mathcal{L}_{\mathbb{R}^2 }(\gamma) + 2 ||\gamma^{(2)}||_{L^\infty} \, .$$
\end{lemma}

\begin{proof} 
Let $(u,v), (u',v') \in \mathbb{S}^1 \times \mathbb{S}^1$ and $f= F_\gamma$ be as in the statement. Without loss of generality we can assume that $\gamma^{(2)}(u) \leq \gamma^{(2)}(u')$. We start proving 
\begin{equation*}
|f(u,v) - f(u',v')| \leq |\gamma(u) - \gamma(u')| +\sqrt{2} 
\gamma^{(2)}(u)
\sqrt{1 - \cos(2\pi (v-  v')) }.
\end{equation*}
First observe that $|f(u,v) - f(u',v') |\leq  |f(u',v') - f(u,v')|+ |f(u,v') - f(u,v)|$. Using the definition of the Euclidean distance we find
$|f(u' , v') - f(u,v')|  
=  |\gamma(u) - \gamma(u') |$.  
Similarly,
\begin{align*}
|f(u,v') - f(u,v)| 
&  = \gamma^{(2)}(u) \sqrt{(\cos(2 \pi v) - \cos(2 \pi v'))^2 + ( \sin( 2\pi v) - \sin(2 \pi v'))^2 } \\  
& = \gamma^{(2)}(u) \sqrt{2 - 2 \cos(2\pi (v-v')) }.
\end{align*}
Both computations imply the desired estimate, and the asserted diameter bound follows immediately.
\end{proof}

%%%%%%%%%%%%%%%%%%%%%%%%%%%%%%%%%%%%%%%%%%%%%%%%%%%%%%
%%%%%%%%%%%%%%%%%%%%%%%%%%%%%%%%%%%%%%%%%%%%%%%%%%%%%%
\section{The Willmore flow of tori of revolution}
%%%%%%%%%%%%%%%%%%%%%%%%%%%%%%%%%%%%%%%%%%%%%%%%%%%%%
%%%%%%%%%%%%%%%%%%%%%%%%%%%%%%%%%%%%%%%%%%%%%%%%%%%%%

In this section we understand the interplay between the rotational symmetry and the \emph{curvature concentration criterion}, which is able to dectect \emph{singularities} of the Willmore flow. This gives us a better understanding of the singularities that can arise in our symmetric setting. We will then prove the main theorems by excluding those singularities in certain circumstances.

%%%%%%%%%%%%%%%%%%%%%%%%%%%%%%%%%%%%%%%%%%%%%%%%%%%%%%%%
%%%%%%%%%%%%%%%%%%%%%%%%%%%%%%%%%%%%%%%%%%%%%%%%%%%%%%%%
\subsection{Singularities of the Willmore flow}
%%%%%%%%%%%%%%%%%%%%%%%%%%%%%%%%%%%%%%%%%%%%%%%%%%%%%%%%
%%%%%%%%%%%%%%%%%%%%%%%%%%%%%%%%%%%%%%%%%%%%%%%%%%%%%%%%

In this section we summarize how singularities of the Willmore flow look like. The following result summarizes a list of results that have been obtained previously in other articles on the Willmore flow. It exposes the {diameter of appropriate parabolic rescalings} as a quantity whose control is sufficient for convergence. The appropriate rescaling is given by a \emph{concentration property} of the Willmore flow, see Appendix \ref{Apptech}. In the following discussion we will use the two parameters $\varepsilon_0$ and $c_0$ which have been introduced in Theorem \ref{thm:C1}.   

\begin{theorem}[Convergence criterium of the Willmore flow, Proof in Appendix \ref{Apptech}]\label{thm:THM22NEW}

Let $\Sigma$ be a compact two-dimensional manifold without boundary and let $f : [0,T) \times \Sigma \rightarrow \mathbb{R}^{{n}}$ be a maximal evolution by the Willmore flow with initial datum $f_0$. Consider an arbitrary sequence $(t_j)_{j \in \N} \subset (0,T)$ with 
$t_j \rightarrow T$. 
Then,
the \emph{concentration radii}
\begin{equation}\label{eq:radii}
r_j := \sup \left\lbrace r > 0 : \forall x \in \mathbb{R}^{{n}} \; \textrm{one has}  \; \int_{f(t_j)^{-1} (B_r(x)) } |A(t_j)|^2 \; \mathrm{d} \mu_{g_{f(t_j)}} \leq \varepsilon_0 \right\rbrace , 
\end{equation} 
$j \in \N$, satisfy $t_j+ c_0 r_j^4 < T$ for all $j \in \N$. Further, the maps 
\begin{equation*}
\tilde{f}_{j,c_0}: \Sigma \to \R^3, \quad \tilde{f}_{j,c_0} := \frac{f(t_j + c_0 r_j^4)}{r_j},
\end{equation*}
are called \emph{concentration rescalings} and one of the following alternatives occurs
\begin{enumerate}
\item[\textbf{Case 1:}] {\emph{(Convergent evolution).} \\There exists $\delta > 0 $ such that $\delta < r_j < \frac{1}{\delta}$.} 
Then $T = \infty$. If additionally $(\mathrm{diam}(\tilde{f}_{j,c_0}))_{j \in \mathbb{N}}$ is uniformly bounded then the Willmore flow converges {to a Willmore immersion}. More precisely there exists a Willmore immersion $f_\infty: \Sigma \rightarrow \mathbb{R}^{{n}}$ such that $f(t) \rightarrow f_\infty$ in $C^k$ for all $k \in \mathbb{N}$ as $t \rightarrow \infty$.  
\vspace{1mm}
\item[\textbf{Case 2:}] \emph{(Blow-up or Blow-down.)} \\ A subsequence of $(r_j)_{j \in \mathbb{N}}$ goes either to zero or to infinity. In this case one has $\mathrm{diam}(\tilde{f}_{j,c_0}) \rightarrow \infty$ as $j \rightarrow \infty$. 
\end{enumerate}
In particular, if $(\mathrm{diam}(\tilde{f}_{j,c_0}))_{j \in \mathbb{N}}$ is uniformly bounded, then $T = \infty$ and the Willmore flow converges {to a Willmore immersion} $f_\infty: \Sigma \rightarrow \mathbb{R}^{{n}}$ in $C^k$ for all $k \in \mathbb{N}$.
\end{theorem}

In the coming sections we will study the relation between the diameter of the concentration rescalings and the hyperbolic length of the profile curves. Having understood this we will finally be able to obtain Theorem \ref{thm:boundhypmain} and Theorem \ref{thm:MainThmPrecisely}.

%%%%%%%%%%%%%%%%%%%%%%%%%%%%%%%%%%%%%%%%%%%%%%%%%%%%%%
\subsection{Dimension reduction}
%%%%%%%%%%%%%%%%%%%%%%%%%%%%%%%%%%%%%%%%%%%%%%%%%%%%%%

We have already announced that the rotational symmetry is preserved along the flow. This section is devoted to the proof of this fact, see Lemma \ref{lem:sympres}. In the proof of Lemma \ref{lem:sympres} we will make use of an alternative characterization {of tori of revolution, see Definition \ref{def:toriofrevolution}}, which we state next.

\begin{prop}\label{prop:consiBlatt}
{Let  $f: \mathbb{S}^1 \times \mathbb{S}^1 \rightarrow \mathbb{R}^3$ be a smooth immersion. Then, $f$} is a torus of revolution if and only if 
\begin{align}\label{eq:D2}
\forall \phi \in \mathbb{S}^1& : \quad  f(u,v+ \phi) = R_{2\pi\phi}  f(u,v), \; \;  \textrm{where} \;  R_z = \begin{pmatrix}
1 & 0 &0 \\ 0 & \cos z & - \sin z  \\ 0 & \sin z &  \cos z
\end{pmatrix},\\
\label{eq:D3}
\forall u \in \mathbb{S}^1 & : \quad  f^{(3)}(u,0) = 0 \; \textrm{and $f^{(2)}(u_0,0) \geq 0$ for one value $u_0 \in \mathbb{S}^1$}.
\end{align}
\end{prop}
\begin{proof}
If $f$ is a torus of revolution then \eqref{eq:D2} and \eqref{eq:D3} can be checked by direct computation. If \eqref{eq:D2} and \eqref{eq:D3} hold for some immersion $f: \mathbb{S}^1 \times \mathbb{S}^1 \rightarrow \mathbb{R}^3$ then one can define a smooth curve $\gamma: \mathbb{S}^1 \rightarrow \mathbb{R}^2$ by $\gamma(u) := (f^{(1)}(u,0),f^{(2)}(u,0))$. Equation \eqref{eq:D1} is then easy to check, but it also needs to be shown that $\gamma(u) \in \mathbb{H}^2$ for all $u \in \mathbb{S}^1$. So far we have 
\begin{equation*}
f(u,v)=  \begin{pmatrix}
\gamma^{(1)}(u), \gamma^{(2)}(u) \cos(2\pi v), \gamma^{(2)}(u) \sin(2\pi v)
\end{pmatrix} \quad \forall (u,v) \in \mathbb{S}^1 \times \mathbb{S}^1.
\end{equation*}
If now there exists a point {$u_0 \in \mathbb{S}^1 \times \mathbb{S}^1$} such that $\gamma^{(2)}(u_0) = 0$ then one can compute 
\begin{equation*}
\partial_v f(u_0,v) = (0,0,0)^T \quad \forall v \in \mathbb{S}^1,
\end{equation*}
which is a contradiction to the fact that $f$ is an immersion. Hence $\gamma^{(2)}$ may not change sign or attain the value zero. As a consequence, $\gamma^{(2)}> 0$ and the claim follows. 
\end{proof}

In particular, given a torus of revolution its profile curve is given by the formula  {$\gamma(u) := (f^{(1)}(u,0),f^{(2)}(u,0))$}. Note that -- by inspection of the previous proof -- each immersion $f: \mathbb{S}^1 \times \mathbb{S}^1 \rightarrow \mathbb{R}^3$ that fulfills \eqref{eq:D2} as well as $f^{(3)}(u,0)= 0 $ for all $u \in \mathbb{S}^1$, must satisfy {$f^{(2)}(\cdot,0) \ne 0$}.  
In particular it cannot change sign. Thus, either $f^{(2)} ( \cdot ,0 ) > 0 $ or $f^{(2)} (\cdot, 0) <0$. {In} the latter case $f( \cdot, \cdot + 1)$ defines a torus of revolution. This shows also consistency of our definition with  \cite[Def. 2.2]{Blatt}, whose results we will need later. 

When it comes to evolutions $(f(t))_{t \geq 0}$, we however want to work without repara-metrizations of $f(t)$ along the flow and hence we specify $\gamma^{(2)}= f^{(2)}(\cdot,0) > 0$ (and we check that this remains satisfied along the flow).  

\begin{lemma}\label{lem:sympres}
Let $f_0 : \mathbb{S}^1 \times \mathbb{S}^1 \rightarrow \mathbb{R}^3$ be a torus of revolution and let {$(f(t))_{t \in [0, T)}:\mathbb{S}^1 \times \mathbb{S}^1 \rightarrow \mathbb{R}^3$ evolve by the Willmore flow with initial datum $f_0$. Then $(f(t))_{t \in [0,T)}$ is a torus of revolution  for all $t \in [0,T)$.}  
\end{lemma}
\begin{proof}
{We prove that $(f(t))_{t \in [0, T)}$ satisfies  \eqref{eq:D2} and \eqref{eq:D3} for all $t \in [0,T)$ so that the claim follows from Proposition \ref{prop:consiBlatt}.}

Let $\phi \in \mathbb{S}^1$. We observe that $R_{2\pi\phi}$ is an isometry in $\mathbb{R}^3$ and $(u,v) \mapsto (u , v + \phi)$ is a diffeomorphism. Hence $(R_{2\pi\phi}^{-1} f(t) (\cdot, \cdot + \phi ))_{t \in [0, T)}:\mathbb{S}^1 \times \mathbb{S}^1 \rightarrow \mathbb{R}^3$ is an evolution by Willmore flow with initial value $R_{2\pi\phi}^{-1} f_0( \cdot, \cdot + \phi )$. Recall now that $f_0$ satisfies \eqref{eq:D2}, i.e. $ R_{2\pi\phi}^{-1} f_0( \cdot, \cdot + \phi ) = f_0 $.   
By the uniqueness result for the Willmore flow, see \cite[Prop. 1.1]{KS1}, we obtain that 
\begin{equation*}
R_{2\pi\phi}^{-1} f(t) ( u ,v + \phi)  = f(t)(u,v)  \quad \forall (u,v) \in \mathbb{S}^1 \times \mathbb{S}^1, 
\end{equation*}
that is \eqref{eq:D2}. In particular, there exist smooth functions $x,y,z:[0,T)\times\mathbb{S}^1 $ such that 
{\begin{equation}\label{eq:D10}
f(t)(u,v) = R_{2\pi v} (f(t) (u,0)) = R_{2\pi v} \begin{pmatrix}
x(t,u) \\ y(t,u) \\ z(t,u) 
\end{pmatrix}.
\end{equation}}

As an intermediate step for \eqref{eq:D3} we show that $f(t)^{(3)}(u,0) = 0$ for all $t> 0$ and $u \in \mathbb{S}^1$, i.e. $z\equiv 0$ on $[0,T)\times \mathbb{S}^1$. 
 Set 
\begin{equation*}
S := \sup \{ s \in [0,T): f(t) \; \textrm{is a torus of revolution for all} \;  t \in [0, s] \}. 
\end{equation*}
We show that $S = T$. If $S<T$ then observe that $z(S,u) = 0$ for all $u \in \mathbb{S}^1$ by smoothness of $(f(t))_{t \in [0, T) }$ and the fact that $f(t)^{(3)}(u,0) = 0$ for all $t \in [0,S)$ and $u \in \mathbb{S}^1$. As additionally $y(S,\cdot)$ is non-negative and $f(S)$ is an immersion,  $f(S)$ is a torus of revolution by Proposition \ref{prop:consiBlatt}.

Restart the flow with $\widetilde{f}_0 := f(S)$ (if $S= 0$ there is no need to restart). Choose now $c_0 , \rho$ for $\widetilde{f}_0$ to be as in Theorem 
\ref{thm:C1} and consider the time interval $I := [S, S + \frac{1}{c_0} \rho^4]$. The Willmore flow equation 
in the local coordinates $(u,v)$ of $\mathbb{S}^1 \times \mathbb{S}^1$ reads  
\begin{equation*}
\partial_t f(t) = P(A(t), \nabla^\perp A(t), (\nabla^\perp)^2 A(t)) \vec{N}_{f(t)},
\end{equation*}
where $\vec{N}_{f(t)}:= \frac{\partial_u f(t) \times \partial_v f}{|\partial_u f(t) \times \partial_v f(t)|}$ 
and $P(A, \nabla^\perp A, (\nabla^\perp)^2 A)$ is a scalar quantity that can be bounded in terms of $||g||_{L^\infty(\mathbb{S}^1 \times \mathbb{S}^1)}, ||(\nabla^\perp)^k A||_{L^\infty(\mathbb{S}^1 \times \mathbb{S}^1)},$ $(k=0,1,2)$.  All of those remain bounded in $I$ by \eqref{eq:gradbound} and the explanation afterwards.
The idea now is to consider the evolution equation satisfied by $z(t,u)^2$. Since 
\begin{equation*}
\vec{N}_{f(t)}(u,v)  =\frac{1}{\sqrt{\det(g(t))}} R_{2\pi v}  \begin{pmatrix}
  y(t,u) \partial_u y(t,u) + \partial_u z(t,u) z(t,u) \\ - y(t,u) \partial_u x(t,u)  \\ - z(t,u) \partial_u x(t,u) 
\end{pmatrix} ,
\end{equation*}
we find 
\begin{align*}
\partial_t( z(t,u)^2) & =2 z(t,u) \partial_tz(t,u) =2 z (t,u) P(A(t), \nabla^\perp A(t) , (\nabla^\perp)^2 A(t) )  \vec{N}_{f(t)}^{(3)} (u,0) 
\\ & = - 2 \frac{1}{\sqrt{\mathrm{det}(g(t))}} P(A(t), \nabla^\perp A(t) , (\nabla^\perp)^2 A(t))  \partial_u x(t,u) z(t,u)^2 . 
\end{align*}
By Theorem \ref{thm:C1}
for fixed $u \in \mathbb{S}^1$ we have obtained 
\begin{equation*}
\begin{cases}
\partial_t (z(t,u)^2) \leq C z(t,u)^2  & t \in I, \\
z(S,u)^2 = 0,  
\end{cases}
\end{equation*}
and hence $z(t,u) = 0 $ for all $t \in I$ and all $u \in \mathbb{S}^1$, as $u$ was chosen arbitrarily.  Similar to before, again by Proposition \ref{prop:consiBlatt} and the discussion afterwards it can be shown that $y(t,\cdot) > 0$ for all  $t \in I$.  This is finally a contradiction to the choice of $S$ and thus $S = T$. The claim follows. \qedhere 
\end{proof}

The previous lemma implies that for each Willmore evolution $(f(t))_{t \geq 0}$ starting at a torus of revolution $f_0 : \mathbb{S}^1 \times \mathbb{S}^1 \rightarrow \mathbb{R}^3$ there exists a unique smooth evolution of curves $(\gamma(t))_{t \in [0,T) } \subset  C^\infty(\mathbb{S}^1, \mathbb{H}^2), \gamma(t)(u) = f(t)(u,0)$  such that 
\begin{equation}\label{eq:D17}
f(t)(u,v) =  \begin{pmatrix}
\gamma^{(1)}(t)(u) \\ \gamma^{(2)}(t)(u) \cos(2\pi v) \\ \gamma^{(2)}(t)(u) \sin(2 \pi v) 
\end{pmatrix}, 
\end{equation}
whereupon the flow can also be seen as an evolution of $(\gamma(t))_{t \in [0,T)} $. 

%%%%%%%%%%%%%%%%%%%%%%%%%%%%%%%%%%%%%%%%%%%%%%%%%%%%%%%%%%%%%%%%%
\subsection{Symmetry of the limit immersion}
%%%%%%%%%%%%%%%%%%%%%%%%%%%%%%%%%%%%%%%%%%%%%%%%%%%%%%%%%%%%%%%%%

Theorem \ref{thm:THM22NEW} provides us with a general convergence criterion for the Willmore flow and yields a smooth limit immersion $f_\infty$, which is a Willmore immersion. In this section we need to check that the revolution symmetry is passed along to the limit, i.e. we will prove that under certain conditions the limit immersion $f_\infty$ is a (Willmore) torus of revolution. Let us stress that this not trivial because the notion of convergence is \emph{geometric}, i.e. invariant with respect to reparametrization. Hence classical results about pointwise convergence can not be applied.

The arguments in this section make frequent use of the fact that to each torus of revolution $f= F_\gamma : \mathbb{S}^1 \times \mathbb{S}^1 \rightarrow \mathbb{R}^3$ one can easily associate a smooth orthonomal frame with respect to $g_f$, given by 
\begin{equation}\label{eq:Orthoframe} 
E_1(u,v):= \frac{1}{|\gamma'(u)|} \frac{\partial}{\partial u} , \quad E_2(u,v) := \frac{1}{2\pi \gamma^{(2)}(u)} \frac{\partial}{\partial v}. 
\end{equation}
This orthonormal frame also has some further interesting properties, for example that it diagonalizes the second fundamental form $A[f]$, and hence yields the principal curvatures of $f$. 
The first principal curvature $\kappa_1[f] = \langle A[f]_{(u,v)}(E_1,E_1), N_f \rangle_{\mathbb{R}^3}= -\kappa_{euc}[\gamma](u) $ coincides up to a sign with the Euclidean scalar curvature of the profile curve, while the second principal curvature $\kappa_2[f] = \langle A[f]_{(u,v)}(E_2,E_2), N_f \rangle_{\mathbb{R}^3} = \frac{(\gamma^{(1)})'(u)}{|\gamma'(u)|\gamma^{(2)}(u)}$ depends heavily on the distance of the profile curve to the revolution axis.   
This will be of great use when it comes to explicit estimates involving the second fundamental form.

\begin{lemma}[Revolution symmetry of the limit]\label{lem:LEMMANEW}
Suppose that $f: [0,\infty) \times ( \mathbb{S}^1 \times \mathbb{S}^1)  \rightarrow \mathbb{R}^3$ is a global evolution by Willmore flow, convergent to some Willmore immersion $f_\infty: \mathbb{S}^1 \times \mathbb{S}^1 \rightarrow \mathbb{R}^3$ in $C^k$ for all $k \in \mathbb{N}$. Suppose further that $f(0)$ is a torus of revolution and $(\gamma(t))_{t \in [0,\infty)} \subset C^\infty(\mathbb{S}^1,\mathbb{R}^2)$ is as in \eqref{eq:D17}.
Then $f_\infty$ is (up to reparametrization) a Willmore torus of revolution. A profile curve $\gamma_\infty$ of $f_\infty$ can be obtained by a $C^m(\mathbb{S}^1, \mathbb{R}^2)$-limit of appropriate reparametrizations of a sequence $(\gamma(t_j))_{j \in \mathbb{N}}$, $t_j \rightarrow \infty$. Here $m \in \mathbb{N}$ is arbitrary. In particular $\gamma_\infty \in C^\infty( \mathbb{S}^1, \mathbb{H}^2)$ is a hyperbolic elastica. 
\end{lemma}
\begin{proof}
Let $(t_j)_{j \in \mathbb{N}} \subset [0, \infty)$ be an arbitrary sequence such that $t_j \rightarrow \infty$. \\
\textbf{Step 1. Bounds for the profile curves.}
After reparametrization we may assume without loss of generality that {$(\gamma(t_j))_{j \in \N}$} is parametrized with constant Euclidean speed. 

Now fix $m \in \mathbb{N}$ arbitrary. To bound the $W^{m,2}$-norm of {$(\gamma(t_j))_{j \in \N}$}  
we first bound $||\gamma(t_j)||_{L^\infty(\mathbb{S}^1, \mathbb{R}^2)}$. To this end we observe by \eqref{eq:D17} that 
\begin{equation*}
||\gamma(t_j) ||_{L^\infty(\mathbb{S}^1, \mathbb{R}^2) } = ||f(t_j)||_{L^\infty(\mathbb{S}^1 \times \mathbb{S}^1, \mathbb{R}^3)}. 
\end{equation*}
Now $||f(t_j)||_{L^\infty}$ is uniformly bounded because it converges in $C^k$ for all $k \in \mathbb{N}$ to $f_\infty$, whose image is a compact subset of $\mathbb{R}^3$. Note that we have used here that the $L^\infty$-norm is not affected by reparametrization. Next we bound $\mathcal{L}_{\mathbb{R}^2}(\gamma(t_j))= ||\partial_u\gamma(t_j)||_{L^\infty}$. 
{We use Lemma \ref{lem:spurbound} and Lemma \ref{lem:simon}  to compute
\begin{equation*}
   \mathcal{L}_{\mathbb{R}^2}(\gamma(t_j)) \leq  \mathcal{W}(f(t_j))^\frac{1}{2} \mu_{g_{f(t_j)}}(\mathbb{S}^1 \times \mathbb{S}^1)^\frac{1}{2} \leq  \mathrm{diam}(f(t_j)(\mathbb{S}^1 \times \mathbb{S}^1)) \mathcal{W}(f(t_j)).
\end{equation*}
Notice that 
$\mathrm{diam}(f(t_j)( \mathbb{S}^1 \times \mathbb{S}^1) \leq 2 ||f(t_j)||_{L^\infty(\mathbb{S}^1 \times \mathbb{S}^1, \mathbb{R}^3)}$, which is uniformly bounded in $j$.
By Lemma \ref{lem:appsemi} and the fact that $\mathbb{S}^1 \times \mathbb{S}^1$ is compact we infer that $\mathcal{W}(f(t_j)) \rightarrow \mathcal{W}(f_\infty)$ and hence $(\mathcal{W}(f(t_j)))_{j \in \mathbb{N}}$ is also uniformly bounded. We conclude the boundedness of $(\mathcal{L}_{\mathbb{R}^2}(\gamma(t_j)))_{j \in \mathbb{N}}$. 
}

Further, we bound second derivatives uniformly in $j$. To this end we introduce the following notation. For a torus of revolution $f: \mathbb{S}^1 \times \mathbb{S}^1 \rightarrow \mathbb{R}^3$ with profile curve $\gamma \in C^\infty(\mathbb{S}^1, \mathbb{H}^2)$  we introduce the vector field on $\mathbb{S}^1 \times \mathbb{S}^1$ 
\begin{equation*}
\partial_s \Big\vert_{(u,v)} = \frac{1}{|\partial_u \gamma(u)|_{euc}}  \frac{\partial}{\partial u} \Big\vert_{(u,v)}. 
\end{equation*}
One easily checks that $g_f( \partial_s , \partial_s ) = 1$ and 
 \begin{equation*}
 \begin{pmatrix}
 -\vec{\kappa}_{euc}[\gamma] (u) \\ 0 
 \end{pmatrix} = A_{(u,0)} [f] (\partial_s , \partial_s) \quad \forall u \in \mathbb{S}^1. 
 \end{equation*}
 By Remark \ref{rem:convsecFF} $||A[f(t_j)]||_{L^\infty}$ is uniformly bounded  in $j$. This is why 
 \begin{equation*}
 ||\vec{\kappa}_{euc}[\gamma(t_j)]||_{L^\infty} \leq ||A[f(t_j)]||_{L^\infty} ||g_{f(t_j)}(\partial_s , \partial_s )||_{L^\infty}^2,
 \end{equation*}
 is also uniformly bounded in $j$. We next control all higher order arclength derivatives of the curvature of $\gamma(t_j)$ uniformly in $j$. Easy tensor calculus and $\partial_s = \frac{1}{|\partial_u \gamma(t_j)(u)|} \partial_u$  implies with \eqref{eq:DXundnablaX}
 {\begin{align}\nonumber
 \frac{1}{|\partial_u \gamma(t_j)(u)| } &\begin{pmatrix}
 -\partial_u \vec{\kappa}_{euc}[\gamma(t_j)](u) \\ 0 
 \end{pmatrix} \\ \label{eq:pandoro}
& =  - D_{\partial_s } \begin{pmatrix}
  \vec{\kappa}_{euc}[\gamma(t_j)] (u) \\ 0 
\end{pmatrix}   = D_{\partial_s } A[f(t_j)]( \partial_s , \partial_s ) 
 \\ \nonumber& =   \nabla_{\partial_s}^\perp A(\partial_s , \partial_s ) - \sum_{i=1}^2 \langle A(\partial_s, \partial_s), A(\partial_s, E_i) \rangle_{\R^3} D_{E_i} [f(t_j)],
 \end{align}}
 where $\{ E_1, E_2 \}$ is an arbitrary orthonormal basis of $T_{(u,0)} (\mathbb{S}^1 \times \mathbb{S}^1)$ with respect to $g_{f(t_j)}$ and we have used the (slightly ambiguous) shorthand notation $A$ for $A[f(t_j)]$. Choosing $E_1 = \partial_s$ and $E_2(u,v) = \frac{1}{\gamma(t_j)^{(2)}(u)} \frac{\partial}{\partial v} \Big\vert_{(u,v)}$ we obtain with \eqref{eq:tnsderiv}
\begin{align}\nonumber
 \frac{1}{|\partial_u \gamma(t_j)(u)|} & \begin{pmatrix}
 - \partial_u \vec{\kappa}_{euc}[\gamma(t_j)](u) \\ 0 
 \end{pmatrix} \\ \nonumber
& = \nabla_{\partial_s}^\perp A(\partial_s , \partial_s ) - |A(\partial_s , \partial_s)|^2 D_{\partial_s} f(t_j)
 \\ \nonumber
& =  \nabla^\perp A( \partial_s, \partial_s , \partial_s ) + A( \nabla_{\partial_s} \partial_s, \partial_s ) + A( \partial_s , \nabla_{\partial_s} \partial_s ) - |A(\partial_s, \partial_s )|^2 D_{\partial_s} f(t_j)
 \\ & =\nabla^\perp A( \partial_s, \partial_s, \partial_s ) - |A(\partial_s, \partial_s)|^2 D_{\partial_s} f(t_j)  ,\label{eq:NEW423}
 \end{align}
 where we have used in the last step that $\nabla_{\partial_s} \partial_s = 0$, which is an immediate consequence of the formula $df_p(\nabla_XY) = \nabla^{\mathbb{R}^3}_{df_p(X)} (df_{(\cdot)} (Y)) $ applied with $f= f(t_j)$. 
 Note that 
 \begin{equation*}
 D_{\partial_s} f (u,0) = \frac{1}{|\gamma'(u)|} D_{\partial_u } f = \frac{1}{|\gamma'(u)|} (\partial_u f)
\end{equation*} 
has Euclidean norm equal to  $1$. We obtain, since $g_{f(t_j)}(\partial_s, \partial_s ) \leq 1$, that 
\begin{equation*}
\frac{|\partial_u \vec{\kappa}_{euc} [\gamma(t_j)](u) | }{|\partial_u \gamma(t_j)(u) |} \leq ||\nabla^\perp A[f(t_j)]||_{L^\infty} + ||A||_{L^\infty}^2 .
\end{equation*}
If we introduce the differential operator $\partial^{arc} := \frac{1}{|\partial_u \gamma(t_j) |} \partial_u$ on $\mathbb{S}^1$  we have obtained 
\begin{equation}
||\partial^{arc} \vec{\kappa}_{euc}[\gamma(t_j)]||_{L^\infty} \leq ||\nabla^\perp A[f(t_j)]||_{L^\infty} + ||A||_{L^\infty}^2 .
\end{equation}
Next we obtain by differentiating \eqref{eq:NEW423} and using the shorthand notation $f= f(t_j)$ as well as $\nabla_{\partial_s} \partial_s = 0 $ again proceeding as in \eqref{eq:pandoro} and \eqref{eq:NEW423}
\begin{align*}
 \begin{pmatrix}
- (\partial^{arc})^2 \vec{\kappa}_{euc}[\gamma(t_j)](u) \\ 0 
 \end{pmatrix} & = D_{\partial_s} [ \nabla^\perp A (\partial_s , \partial_s ,\partial_s ) -|A(\partial_s, \partial_s)|^2 D_{\partial_s} f] \\
 & = D_{\partial_s} \nabla^\perp A( \partial_s , \partial_s , \partial_s  )
 \\ & \quad \quad  \quad \quad \quad -\partial_s (|A(\partial_s, \partial_s)|^2) D_{\partial_s} f - |A(\partial_s, \partial_s)|^2 D_{\partial_s} D_{\partial_s} f \\
 & = \nabla^\perp_{\partial_s} \nabla^\perp A( \partial_s , \partial_s , \partial_s  ) - ( \nabla^\perp A( \partial_s, \partial_s , \partial_s) , A ( \partial_s , \partial_s ) ) D_{\partial_s} f 
 \\ & \quad \quad  \quad \quad \quad -\partial_s (|A(\partial_s, \partial_s)|^2) D_{\partial_s} f - |A(\partial_s, \partial_s)|^2 D_{\partial_s} D_{\partial_s} f 
 \\ & = (\nabla^\perp)^2 A( \partial_s , \partial_s , \partial_s , \partial_s ) - ( \nabla^\perp A( \partial_s, \partial_s , \partial_s) , A ( \partial_s , \partial_s ) ) D_{\partial_s} f 
 \\ & \quad \quad  \quad \quad \quad -\partial_s (|A(\partial_s, \partial_s)|^2) D_{\partial_s} f - |A(\partial_s, \partial_s)|^2 D_{\partial_s} D_{\partial_s} f .
\end{align*}
Note that since $A$ is normal and $\nabla_{\partial_s} \partial_s = 0 $ we have
\begin{align*}
\partial_s |A(\partial_s, \partial_s)|^2 & = 2( D_{\partial_s} A(\partial_s, \partial_s) , A(\partial_s, \partial_s) ) \\
& = 2 ( \nabla_{\partial_s }^\perp A ( \partial_s , \partial_s ) , A ( \partial_s , \partial_s) ) 
= 2 ( \nabla^\perp A( \partial_s , \partial_s , \partial_s ) , A( \partial_s, \partial_s)) .
\end{align*}
Moreover we have
\begin{equation*}
D_{\partial_s} D_{\partial_s} f  = (D_{\partial_s} D_{\partial_s} f )^T + A(\partial_s , \partial_s ). 
\end{equation*}
An easy computation\footnote{Recall that the normal to the curve $\gamma$ coincides up to a sign with the normal to $f(\Sigma)$.} now reveals that $(D_{\partial_s} D_{\partial_s} f )^{T} = 0$ and we obtain 
\begin{align*}
& \begin{pmatrix}
 -(\partial^{arc})^2 \vec{\kappa}_{euc}[\gamma(t_j)](u) \\ 0 
 \end{pmatrix}   =   (\nabla^\perp)^2 A( \partial_s, \partial_s, \partial_s , \partial_s  ) \\
& \qquad \qquad - 3 ( \nabla^\perp A( \partial_s , \partial_s , \partial_s) , A(\partial_s, \partial_s) ) D_{\partial_s} f  - |A(\partial_s, \partial_s)|^2 A(\partial_s , \partial_s) .
\end{align*}
For short we write
\begin{align*}
\begin{pmatrix}
 -(\partial^{arc})^2 \vec{\kappa}_{euc}[\gamma(t_j)](u) \\ 0 
 \end{pmatrix} =  (\nabla^\perp)^2A + \nabla^\perp A * A * D_{\partial_s} f + A* A* A 
\end{align*} 

which implies 
\begin{equation*}
||(\partial^{arc})^2 \vec{\kappa}_{euc} ||_{L^\infty} \leq C [||(\nabla^\perp)^2 A||_{L^\infty} + ||\nabla A||_{L^\infty}\; ||A||_{L^\infty} + ||A||_{L^\infty}^3] .
\end{equation*}
Inductively one shows that  for all $m \in \mathbb{N}$
\begin{align}\label{eq:partialarc} 
& \begin{pmatrix}
 -(\partial^{arc})^m \vec{\kappa}_{euc}[\gamma(t_j)](u) \\  0 
 \end{pmatrix} \\  \nonumber
= &  (\nabla^\perp)^m A +   P_1(A, \nabla^\perp A,...,(\nabla^\perp)^{m-1} A )  * D_{\partial_s} f + P_2( A, \nabla^\perp A, ..., (\nabla^\perp)^{m-2} A )   
\end{align} 
where $P_1$ is a real-valued polynomial of degree $\leq 2$ and $P_2$ is an $\mathbb{R}^3$-valued polynomial of degree $\leq 3$.
%Next let $\gamma_{t_j}$ be a reparametrization of $\gamma(t_j)$ with constant Euclidean velocity, i.e. $\partial_u \gamma_{t_j} = \mathcal{L}_{\mathbb{R}^2}( \gamma(t_j))$. 

 We conclude from \eqref{eq:partialarc} that for all $m \in \mathbb{N}$ 
\begin{equation}
||\partial_u^{m} \gamma(t_j)||_{L^\infty} \leq C(m) \mathcal{L}_{\mathbb{R}^2} (\gamma(t_j) )^m \left[ || (\nabla^\perp)^m A ||_{L^\infty} + \sum_{i = 0}^{m-1} ||(\nabla^\perp)^i A||_{L^\infty}^3 \right].
\end{equation}
Hence for each fixed $m \in \mathbb{N}$ we can bound $(\gamma(t_j))_{j \in \mathbb{N}}$ uniformly in $W^{m+1,\infty} ( \mathbb{S}^1, \mathbb{R}^2)$ and hence obtain a convergent subsequence in $C^{m} ( \mathbb{S}^1, \mathbb{R}^2)$ for any $m$. \\
\textbf{Step 2. The limit curve is a profile curve.}
By a diagonal argument we can also obtain a sequence $t_{j} \rightarrow \infty$ (no relabeling) and $\gamma_\infty \in C^\infty( \mathbb{S}^1, \mathbb{R}^2)$ such that $\gamma(t_{j})$ converges to $\gamma_\infty$ in $C^m(\mathbb{S}^1, \mathbb{R}^2)$ for all $m \in \mathbb{N}$ (classical convergence). Note also that $\gamma_\infty$ is parametrized with constant Euclidean speed and $\gamma_\infty^{(2)} \geq 0 $ on $\mathbb{S}^1$.  We next show that  $\gamma_\infty \in C^\infty( \mathbb{S}^1, \mathbb{H}^2)$, i.e. $\inf_{\mathbb{S}^1} \gamma_\infty^{(2)} > 0 $.
Indeed, assume the opposite, i.e. there exists $u_0 \in \mathbb{S}^1$ such that $\gamma_\infty^{(2)}(u_0) = 0$. Notice that this and $\gamma_\infty^{(2)}\geq 0$ also yields $(\gamma_\infty^{(2)})'(u_0) = 0$. As a consequence, we infer that there exists $C > 0$ and $\delta_0 > 0$ such that $0 \leq \gamma_\infty^{(2)}(u) \leq C|u-u_0|^2$ for all $u \in (u_0- \delta_0, u_0 + \delta_0)$.
The fact that $\gamma_\infty$ is parametrized with constant Euclidean velocity also yields that $|(\gamma_\infty^{(1)})'(u_0)| = \mathcal{L}_{\mathbb{R}^2}(\gamma_\infty) > 0$. With this information we now estimate the following quantity for arbitrary $\delta \in (0,\delta_0)$
\begin{equation*}
    Q := \int_0^1 \frac{|(\gamma_\infty^{(1)})'(u)|^2}{\gamma_\infty^{(2)}(u)} \; \mathrm{d}u \geq \int_{u_0- \delta}^{u_0 + \delta} \frac{|(\gamma_\infty^{(1)})'(u)|^2}{C|u-u_0|^2} \; \mathrm{d}u \geq \frac{1}{C\delta^2} \int_{u_0-\delta}^{u_0 + \delta} |(\gamma_\infty^{(1)})'(u)|^2 \; \mathrm{d}u.
\end{equation*}
Taking the limit $\delta \rightarrow 0+$ yields infinity on the right hand side, since $$\frac{1}{2\delta} \int_{u_0-\delta}^{u_0+\delta} |(\gamma_\infty^{(1)})'(u)|^2 \; \mathrm{d}u \rightarrow |(\gamma_\infty^{(1)})'(u_0)|^2 = \mathcal{L}_{\mathbb{R}^2}(\gamma_\infty)^2  > 0.$$ We infer that $Q = \infty.$ On the other hand, Fatou's lemma and the explicit formula for the second principal curvature $\kappa_2$ of a surface imply that
\begin{align*}
    Q & \leq \liminf_{j \rightarrow \infty} \int_0^1 \frac{|(\gamma(t_{j})^{(1)})'|^2}{\gamma(t_{j})^{(2)}} \; \mathrm{d}u
    = \liminf_{j \rightarrow \infty} \frac{\mathcal{L}_{\mathbb{R}^2}(\gamma(t_{j}))}{2\pi} \int_0^{1}\int_0^1 \frac{2\pi|(\gamma(t_{j})^{(1)})'|^2}{\gamma(t_{j})^{(2)} \mathcal{L}_{\mathbb{R}^2}(\gamma(t_{j})) } \; \mathrm{d}u \; \mathrm{d}v
    \\ & \leq \liminf_{j \rightarrow \infty} \frac{\mathcal{L}_{\mathbb{R}^2}(\gamma(t_{j}))}{2\pi} \int_0^1\int_0^1 \kappa_2[F_{\gamma(t_{j})}]^2 \{2\pi \gamma(t_{j})^{(2)} \mathcal{L}_{\mathbb{R}^2}(\gamma(t_{j})) \} \; \mathrm{d}u \; \mathrm{d}v 
    \\ & 
    \leq \liminf_{j \rightarrow \infty} \frac{\mathcal{L}_{\mathbb{R}^2}(\gamma(t_{j}))}{2\pi} \int_{\mathbb{S}^1 \times \mathbb{S}^1} |A[F_{\gamma(t_{j})}]|^2 \; \mathrm{d}\mu_{F_{\gamma(t_{j})}}
    \\ & = \liminf_{j \rightarrow \infty} \frac{\mathcal{L}_{\mathbb{R}^2}(\gamma(t_{j}))}{2\pi} \int_{\mathbb{S}^1 \times \mathbb{S}^1} |A[f(t_{j})]|^2 \; \mathrm{d}\mu_{f(t_{j})} = \liminf_{j \rightarrow \infty} \frac{2\mathcal{L}_{\mathbb{R}^2}(\gamma(t_{j}))}{\pi} \mathcal{W}(f(t_{j})),
\end{align*}
where the last identity is due to the Gauß-Bonnet theorem, cf. \eqref{eq:secwillmore}. Recall from estimates in Step 1  that  $\mathcal{L}_{\mathbb{R}^2}(\gamma(t_{j}))$ is uniformly bounded. As a consequence of this one infers that $Q < \infty$, a contradiction. 
%Then $\inf_{\mathbb{S}^1}  \gamma^{(2)}(t_j ) %= \inf_{\mathbb{S}^1} \gamma^{(2)}(t_j)
%\rightarrow 0$. By \eqref{eq:hyplen} we also find that $\sup_{\mathbb{S}^1} \gamma^{(2)}(t_j) \rightarrow 0$. This implies by Lemma \ref{lem:drehdiam} {and \eqref{eq:hyplen2}} that 
%\begin{equation*}
%\mathrm{diam} (f(t_j)) \leq \frac{1}{2} \mathcal{L}_{\mathbb{R}^2} ( \gamma(t_j) ) + 2 \sup_{\mathbb{S}^1} \gamma^{(2)}(t_j)\leq  \left( \frac{1}{2} \LH( \gamma(t_j) ) +  2 \right) \sup_{\mathbb{S}^1} \gamma^{(2)}(t_j) \rightarrow 0. 
%\end{equation*}
%By lower semicontinuity one would have $\mathrm{diam}(f_\infty) = 0$, a contradiction to the fact that $f_\infty: \mathbb{S}^1 \times \mathbb{S}^1 \rightarrow \mathbb{R}^3$ is an immersion. 
We obtain therefore that $\gamma_\infty \in C^\infty( \mathbb{S}^1, \mathbb{H}^2)$. \\
\textbf{Step 3. Convergence of the associated surfaces}.
By the following proposition (Proposition \ref{prop:D5}), the tori of revolution $F_{\gamma(t_j)}$ converge to $F_{\gamma_\infty}$ classically in $C^k$ for all $k$. Since $F_{\gamma(t_j)}$ is a reparametrization of %$F_{\gamma(t_j)}=
$f(t_j)$ for all $j \in \mathbb{N}$, also $f(t_j)$ converges to $F_{\gamma_\infty}$ in $C^k$ for all $k$. By assumption however, $f(t_j)$ also converges to $f_\infty$ in $C^k$ for all $k$ {(in general not anymore classically, but in the sense of Definition \ref{def:Clconv})}. Applying Corollary \ref{cor:unique} we infer that $f_\infty$ coincides up to reparametrization with $F_{\gamma_\infty}$. In particular $f_\infty$ is (up to reparametrization) a torus of revolution. Since $f_\infty$ is also a Willmore immersion it must (up to reparametrization) be a Willmore torus of revolution.  By \eqref{eq:willelaGrad} we infer also that $\gamma_\infty$ is a hyperbolic elastica. 
\end{proof}

The following proposition is needed to complete the proof of the previous lemma.

%We will use this again to obtain the following proposition.% that is needed to finish the proof of the previous one. 

\begin{prop}\label{prop:D5}
Let $m\geq 1$ and suppose that $(\gamma_j) _{j \in \mathbb{N}} \subset C^\infty(\mathbb{S}^1, \mathbb{H}^2)$ converges in $C^m(\mathbb{S}^1, \mathbb{R}^2)$ (classically) to some immersed curve $\gamma \in C^m(\mathbb{S}^1, \mathbb{H}^2)$. Then $F_{\gamma_j}$ converges classically to $F_\gamma$  in  $C^m(\mathbb{S}^1 \times \mathbb{S}^1)$. 
\end{prop}
\begin{proof}
We will use without further notice the characterization of $C^m$-convergence in Proposition \ref{prop:propsmoconalt}. We show the claim only for $m = 1$, the other cases follow by induction. We define $w_j : \mathbb{S}^1 \times \mathbb{S}^1 \rightarrow \mathbb{R}^3$ via
\begin{equation}\label{eq:D21}
w_j(u,v) :=  F_{\gamma_j}(u,v) - F_\gamma(u,v) = \begin{pmatrix}
\gamma_j^{(1)}(u)- \gamma^{(1)}(u) \\ (\gamma_j^{(2)}(u) - \gamma^{(2)}(u)) \cos(2\pi v) \\  (\gamma_j^{(2)}(u) - \gamma^{(2)}(u)) \sin(2\pi v)
\end{pmatrix} 
\end{equation}
and we show that $||w_j||_{L^\infty(\mathbb{S}^1 \times \mathbb{S}^1, \widehat{g}) },||D w_j||_{L^\infty(\mathbb{S}^1 \times \mathbb{S}^1, \widehat{g})} \rightarrow  0$ as $j \rightarrow \infty$. Here $\widehat{g}= F_\gamma^*g_{\mathbb{R}^3}$ is the metric induced by $F_\gamma$.  The fact that $||w_j||_{L^\infty} \rightarrow 0 $ follows directly from
\eqref{eq:D21} %\eqref{eq:Orthoframe}
by the estimate
\begin{equation*}
||w_j||_{L^\infty} \leq ||\gamma_j- \gamma||_{L^\infty} \rightarrow 0. 
\end{equation*}
Let $E_1,E_2$ be the orthonormal frame as in \eqref{eq:Orthoframe}. Then
\begin{equation}\label{eq:D24}
||Dw_j||_{L^\infty} = \sup_{\mathbb{S}^1 \times \mathbb{S}^1}  \sup_{g(X,X) \leq 1}  |D w_j(X)| = \sup_{ \mathbb{S}^1 \times \mathbb{S}^1} \sup_{ \theta_1^2 + \theta_2^2 \leq 1 }  |Dw_j ( \theta_1 E_1 + \theta_2 E_2 )| ,
\end{equation}
and
%Using \eqref{eq:Orthoframe} we compute
\begin{align*}
|D w_j ( E_1 )|& = \frac{1}{|\gamma'(u)|} \left\vert \frac{\partial w_j}{\partial u} \right\vert \leq \frac{1}{|\gamma'(u)|} ||\gamma_j '- \gamma' ||_\infty  \leq \frac{1}{\inf_{\mathbb{S}^1} |\gamma'|} ||\gamma_j' - \gamma'||_\infty ,
\\
|Dw_j ( E_2 )| & = \frac{1}{2 \pi \gamma^{(2)}(u)} \left\vert \frac{\partial w_j}{\partial v} \right\vert  \leq \frac{1}{ \inf_{\mathbb{S}^1} \gamma^{(2)} } ||\gamma_j - \gamma||_{L^\infty} .
\end{align*}
Note that $\inf_{\mathbb{S}^1} |\gamma'| > 0 $ as $\gamma$ is immersed and $\inf_{\mathbb{S}^1} \gamma^{(2)} > 0$ since $\gamma \in C^\infty(\mathbb{S}^1,  \mathbb{H}^2)$ and $\mathbb{S}^1$ is compact.  
The claim follows from \eqref{eq:D24} since $\gamma_j \rightarrow \gamma$ in $C^1$.
\end{proof}

%%%%%%%%%%%%%%%%%%%%%%%%%%%%%%%%%%%%%%%%%%%%%%%%%%%%%%
\subsection{Rotational symmetry and concentration}\label{sec:rotsymconc}
%%%%%%%%%%%%%%%%%%%%%%%%%%%%%%%%%%%%%%%%%%%%%%%%%%%%%%
In this section we will prove a lemma that controls the distance of the concentration points to the axis of revolution.  Here the revolution symmetry will play an important role. The following lemma is the main observation that rules out Case (2) in Theorem \ref{thm:THM22NEW}. 
%\changedAn{In the proof we use \emph{cylindrical coordinates} defined as follows. We write $x=(x^{(1)},x^{(2)},x^{(3)})\in \R^3$ in cylindrical coordinates as $(h,\rho \sigma)$ with $h=x^{(1)} \in \R$, $\rho=\sqrt{(x^{(2)})^2+(x^{(3)})^2}\geq 0$ and $\sigma =(x^{(2)},x^{(3)})/\rho\in \mathbb{S}^1$ .}
 
\begin{lemma}[Distance control for concentration points]\label{lem:evilcirclelemma}
Let $f: [0,T) \times (\mathbb{S}^1 \times \mathbb{S}^1) \rightarrow \mathbb{R}^3$ be a maximal evolution by Willmore flow such that $f(0)$ is a torus of revolution. 
Suppose that $t_j \rightarrow T$. Let $(r_j)_{j \in \mathbb{N}}$ be as in Theorem \ref{thm:THM22NEW} and let $x_j \in \mathbb{R}^3$ be such that
\begin{equation}\label{eq:concentrationxj}
\int_{f(t_j)^{-1} ( \overline{B_{r_j}(x_j)} ) } |A[f(t_j)]|^2  \; \mathrm{d} \mu_{g_{f(t_j)}} \geq \varepsilon_0.
\end{equation}
Let $h_j \in \mathbb{R}, \; \rho_j > 0$  and $\sigma_j \in \mathbb{S}^1$ such that $\frac{x_j}{r_j}$ {is expressed in cylindrical }%can be expressed in cylinder 
coordinates by \footnote{i.e. $h_j=x_j^{(1)}/r_j \in \mathbb{R}$, $\rho_j=\sqrt{(x_j^{(2)})^2+(x_j^{(3)})^2}/r_j\geq 0$ and $\sigma_j =(x_j^{(2)},x_j^{(3)})/(\rho_j r_j)\in \mathbb{S}^1$. We consider a cylinder with axis in direction $(1,0,0)$.}  $\frac{x_j}{r_j} = (h_j, \rho_j \sigma_j)$. Then $(\rho_j)_{j \in \mathbb{N}}$ is bounded. 
\end{lemma} 
\begin{proof}
We first use scaling properties to obtain that 
\begin{equation}\label{eq:new414}
\int_{\left( \frac{f(t_j)}{r_j} \right)^{-1} ( \overline{B_{1}(\frac{x_j}{r_j})} ) } \left\vert A\left[ \frac{f(t_j)}{r_j} \right] \right\vert^2  \; \mathrm{d} \mu_{g_{\frac{f(t_j)}{r_j}}} \geq \varepsilon_0.
\end{equation}
Now write $\frac{x_j}{r_j} = (h_j, \rho_j \sigma_j)$ as in the statement. Since $\frac{f(t_j)}{r_j}$ has a revolution symmetry {(see Lemma \ref{lem:sympres})} we conclude from \eqref{eq:new414} that the curvature concentration does not only happen at points but actually on circles.  %is passed along the symmetry axis. 
More precisely,
\begin{equation}\label{eq:new423}
\int_{\left( \frac{f(t_j)}{r_j} \right)^{-1}  (   \overline{B_1 (h_j, \rho_j \sigma) }) } \left\vert A\left[ \frac{f(t_j)}{r_j} \right] \right\vert^2  \; \mathrm{d} \mu_{g_{\frac{f(t_j)}{r_j}}} \geq \varepsilon_0 \quad \forall \sigma \in \mathbb{S}^1.
\end{equation}
Next, we define for each $\rho > 0$ the maximal number of disjoint closed balls of radius $1$ needed to cover the circle $(0,\rho \ \mathbb{S}^1) \subset \R^3$ %a circle of radius $\rho$
\begin{align*}
N(\rho) & := \max \big\lbrace l \in \mathbb{N} : \exists \omega_1, ..., \omega_l \in \mathbb{S}^1 \\
&  \qquad \quad \textrm{ s.t. }  \overline{B_1((0,\rho \omega_1))}, ..., \overline{B_1((0, \rho \omega_l))}  \textrm{ are pairwise disjoint} \big\rbrace.
\end{align*}
This number depends only on the radius of the circle and not on its position in $\R^3$. %Note that 
By compactness of $\mathbb{S}^1$, $N(\rho)$ is well-defined and finite. Moreover, using \eqref{eq:new423} on $N(\rho_j)$ disjoint balls that cover $(h_j, \rho_j \ \mathbb{S}^1)$ and that 
preimages of disjoint sets are always disjoint, we infer
\begin{equation*}
\int_{\mathbb{S}^1 \times \mathbb{S}^1}  \left\vert A\left[ \frac{f(t_j)}{r_j} \right] \right\vert^2  \; \mathrm{d} \mu_{g_{\frac{f(t_j)}{r_j}}} \geq N(\rho_j) \varepsilon_0 .
\end{equation*}
Note that this implies by scaling properties {and the Gauss-Bonnet Theorem} that 
\begin{equation*}
N(\rho_j) \leq \frac{1}{\varepsilon_0} \int_{\mathbb{S}^1 \times \mathbb{S}^1} |A[f(t_j)]|^2 d\mu_{g_{f(t_j)}} = \frac{1}{\varepsilon_0} \mathcal{W}(f(t_j)) \leq \frac{\mathcal{W}(f_0)}{\varepsilon_0}.
\end{equation*}
To infer that $\rho_j$ is bounded it suffices now to show that $N(\rho) \rightarrow \infty$ as $\rho \rightarrow \infty$. To this end we prove that 
\begin{equation}\label{eq:new427}
N(\rho) \geq  \frac{\pi}{4 \arccos ( 1-  \frac{8}{ \rho^2} ) } \mbox{ for }\rho \geq 4. 
\end{equation}
Let us first fix {$\rho \geq 4$}. % and call  call the right hand side  of \eqref{eq:new427} $\tilde{N}[=\tilde{N}(\rho)]$.
Note first that the squared Euclidean distance  in $\mathbb{R}^3$ 
between $( 0, \rho \cos( \alpha), \rho \sin( \alpha) )$ and $ ( 0, \rho \cos(\beta),  \rho \sin(\beta))$ is given by 
\begin{equation*}
d^2_{\alpha, \beta} := 2 \rho^2 ( 1- \cos( \alpha- \beta) ) .
\end{equation*} 
Also observe that the balls $ \overline{B_1((0, \rho \cos( \alpha), \rho \sin( \alpha) )}, \overline{B_1( (0, \rho \cos( \beta) , \rho \sin(\beta) )}$ are disjoint if and only if $d^2_{\alpha, \beta} >
4$. Hence it suffices to find distinct values $\alpha_1,...,\alpha_{\tilde{N}} \in [0,2\pi)$ such that for all $i,j \in \{1,..., \tilde{N} \}$ one has 
\begin{equation*}
d^2_{\alpha_i,\alpha_j} \geq 16 \quad \forall i, j \in \{ 1, ..., \tilde{N} \}. 
\end{equation*}
We claim that the choice of $\alpha_j := j \arccos(1- \frac{8}{\rho^2})$ , $j= 1,..., \tilde{N}$ { with
$$ \tilde{N}= \big\lfloor \frac{\pi}{4 \arccos ( 1-  \frac{8}{ \rho^2} )}\big\rfloor $$ has the desired properties.} %does the job. 
Indeed, note that $\alpha_1,...,\alpha_{\tilde{N}} \in [0, \frac{\pi}{4}]$ which implies that $|\alpha_i- \alpha_j| \in [0, \frac{\pi}{2}]$ for all $i,j$. Using evenness of $\cos$ and monotonicity of $\cos$ in $[0, \frac{\pi}{2}]$ we obtain for all $i,j \in \{ 1,...,\tilde{N} \}$
\begin{align*}
d^2_{\alpha_i, \alpha_j} & = 2\rho^2 ( 1- \cos( \alpha_i- \alpha_j) ) 
=2\rho^2 ( 1- \cos ( |i-j| \arccos(1- \frac{8}{\rho^2}) ) \\ &  \geq 2\rho^2 ( 1- \cos(1 \cdot \arccos( 1- \frac{8}{\rho^2}) ) = 16.
\end{align*}
We have thus shown \eqref{eq:new427} and thus the claim follows.
\end{proof}

\begin{remark}
The lemma reveals an interesting property of the Willmore flow of tori of revolution. Suppose that $T< \infty$. Then by Theorem \ref{thm:THM22NEW} {and in particular the property $t_j +c_0 r_j^4<T$, necessarily }$r_j \rightarrow 0$. Now let $(x_j)_{j \in \mathbb{N}}$ be a collection of \emph{points of concentration}, i.e. points where \eqref{eq:concentrationxj} holds true. From the previous lemma we know that the distance of $\frac{x_j}{r_j}$ to the $x$-axis is bounded. Hence the distance $(x_j)_{j \in \mathbb{N}}$ to the $x$-axis tends to zero. With other words: Finite-time-concentration may only happen close to the $x$-axis.  
\end{remark}

%%%%%%%%%%%%%%%%%%%%%%%%%%%%%%%%%%%%%%%%%%%%%%%%%
\subsection{Proof of Theorems \ref{thm:boundhypmain} and \ref{thm:MainThmPrecisely}}
%%%%%%%%%%%%%%%%%%%%%%%%%%%%%%%%%%%%%%%%%%%%%%%%%

\begin{proof}[Proof of Theorem \ref{thm:boundhypmain}] Let $f: [0,T) \times \mathbb{S}^1 \times \mathbb{S}^1 \rightarrow \mathbb{R}^3$ be as in the statement.
 That $f(t)$ is a torus of revolution for all $t \in [0,T)$ follows from Lemma \ref{lem:sympres}. Thus we can actually choose $(\gamma(t))_{t \in [0,T)}$ as in the statement, see also the discussion after Lemma \ref{lem:sympres}. 
Let $t_j \rightarrow T$ be such that $\LH( \gamma(t_j)) \leq M$ for some $M> 0$ and let $r_j > 0 $ and $\tilde{f}_{j,c_0}$ be as in Theorem \ref{thm:THM22NEW}. By Theorem \ref{thm:THM22NEW} it is sufficient for the convergence of the Willmore flow that $(\mathrm{diam}(\tilde{f}_{j,c_0}))_{j \in \mathbb{N}}$ is bounded. {Notice that we assume a bound on $\LH$ at $t_j$ and we want a bound on the diameter at $t_j+c_0 r_j^4$.} To this end we define $\tilde{f}_{j,0} := \frac{f(t_j)}{r_j}$ and choose for all $j \in \mathbb{N}$, $x_j$ as in \eqref{eq:concentrationxj}. Such a choice of $x_j$  exists due to the definition of $r_j$ in Theorem \ref{thm:THM22NEW}. We write $\frac{x_j}{r_j} = (h_j, \rho_j \sigma_j)$, $\rho_j > 0$ and $\sigma_j \in \mathbb{S}^1$ as in Lemma \ref{lem:evilcirclelemma} and infer from Lemma \ref{lem:evilcirclelemma} that $(\rho_j)_{j \in \mathbb{N}}$ is bounded, say $\rho_j \leq C$ for all $j \in \mathbb{N}$. Note that {by the choice of $x_j$, in particular \eqref{eq:new414},} for all $j \in \mathbb{N}$ one has $\mathrm{dist}( \frac{x_j}{r_j} , \tilde{f}_{j,0}(\mathbb{S}^1 \times \mathbb{S}^1) ) \leq 1$. %,  otherwise $x_j$ could not be as in \eqref{eq:concentrationxj}. 
Now we look at $\tilde{\gamma}_j= \frac{\gamma(t_j)}{r_j}$, which is clearly a profile curve of $\tilde{f}_{j,0}$ and satisfies also $\LH(\tilde{\gamma}_j) \leq M$ by scaling invariance of the hyperbolic length. By the distance estimate we can find $u_j ,v_j \in \mathbb{S}^1$ such that 
{$$\left\vert  \frac{1}{r_j} \left[ (x_j^{(2)}, x_j^{(3)}) - \gamma_j^{(2)}(u_j) (\cos(2 \pi v_j ), \sin(2 \pi v_j)) \right] \right\vert \leq 1.$$}
Hence we infer that 
\begin{equation*}
\tilde{\gamma}_j^{(2)} (u_j) \leq 1 + |\frac{1}{r_j} (x_j^{(2)} , x_j^{(3)} ) | \leq 1 + \rho_j  \leq 1 + C.
\end{equation*}
From the bounded hyperbolic length and \eqref{eq:hyplen} we infer that
\begin{equation*}
\sup_{\mathbb{S}^1} \tilde{\gamma}_j^{(2)} \leq \tilde{\gamma}_j^{(2)}(u_j) e^{\LH( \tilde{\gamma}_j ) } \leq (1+C)e^M.
\end{equation*}
%which is uniformly bounded in $j$. 
{This} implies also by \eqref{eq:hyplen2} that 
\begin{equation*}
\mathcal{L}_{\mathbb{R}^2} ( \tilde{\gamma}_j) \leq \sup_{\mathbb{S}^1} \tilde{\gamma}_j^{(2)} \LH( \tilde{\gamma}_j) \leq M(1+C) e^M ,
\end{equation*}
%which is also uniformly bounded in $j$. 
and from Lemma \ref{lem:drehdiam} we now infer 
{\begin{equation}\label{eq:etchiu}
\mathrm{diam}(\tilde{f}_{j,0}) \leq \frac{1}{2}\mathcal{L}_{\mathbb{R}^2}(\tilde{\gamma}_j) + 2 \sup_{\mathbb{S}^1} \tilde{\gamma}_j^{(2)} \leq D,
\end{equation}
for some constant $D\geq 0$.}
%which is now also uniformly bounded in $j$, say $\mathrm{diam}(\tilde{f}_{j,0}) \leq D$. 
We now define $\tilde{f}_j(s) := \frac{f(t_j  + s r_j^4) }{r_j}$, $s \in [0,c_0]$, taking into account the parabolic scaling. It is easy to see that then $\tilde{f}_j$ is a solution of the Willmore flow equation and $\tilde{f}_j(0) = \tilde{f}_{j,0}$ and $\tilde{f}_j(c_0) = \tilde{f}_{j,c_0}$. Hence we can estimate by Lemma \ref{lem:diamflow} 
\begin{equation*}
\mathrm{diam}(f_{j,c_0}) \leq C( \mathcal{W}(\tilde{f}_{j,0}) ) ( \mathrm{diam}( \tilde{f}_{j,0} ) + c_0^\frac{1}{4}).
\end{equation*}
Using that by scaling invariance $\mathcal{W}(\tilde{f}_{j,0})= \mathcal{W}(f(t_j)) \leq \mathcal{W}(f_0)$ and \eqref{eq:etchiu} %the previous estimates 
we obtain 
\begin{equation}
\mathrm{diam}(f_{j,c_0}) \leq C(\mathcal{W}(f_0)) ( D + c_0^\frac{1}{4} ). 
\end{equation}
By Theorem \ref{thm:THM22NEW} this implies that $T= \infty$ and $(f(t))_{t \in [0, \infty)}$ is a convergent evolution. It only remains to show that the limit is a torus of revolution. This is however a direct consequence of Lemma \ref{lem:LEMMANEW}. 
\end{proof}

\begin{proof}[Proof of Theorem \ref{thm:MainThmPrecisely}]
Let $(f(t))_{t \in [0,T) }$ and $(\gamma(t))_{t \in [0,T)}$ be as in the statement. We distinguish two cases.  \\
\textbf{Case 1:} $\mathcal{W}(f_0) < 8\pi$. 
To show long-time existence and convergence of the evolution we apply Theorem \ref{thm:boundhypmain}. To this end we need to show that 
\begin{equation*}
 \liminf_{t \rightarrow T} \LH(\gamma(t)) < \infty. 
\end{equation*}
First we observe that $(\gamma(t))_{t \in [0,T)}$ satisfies 
\begin{equation*}
\mathcal{E}(\gamma(t)) = \frac{2}{\pi} \mathcal{W}(F_{\gamma(t)}) = \frac{2}{\pi} \mathcal{W}(f(t)) \leq \frac{2}{\pi} \mathcal{W}(f_0) <16.
\end{equation*} 
We apply Theorem \ref{thm:reilly} with $\varepsilon := 16- \frac{2}{\pi} \mathcal{W}(f_0)$ to find that for each $t \in [0,T) $ one has 
\begin{equation*}
\LH(\gamma(t)) \leq \frac{1}{c(\varepsilon)} \mathcal{E}(\gamma(t)) = \frac{2}{\pi c(\varepsilon)} \mathcal{W}(f(t)) \leq  \frac{2}{\pi c(\varepsilon)} \mathcal{W}(f_0),
\end{equation*}
and hence the hyperbolic length is uniformly bounded for $t \in [0,T)$. By Theorem \ref{thm:boundhypmain} the evolution converges in $C^k$ for all $k$ and the limit, say $f_\infty: \mathbb{S}^1 \times \mathbb{S}^1 \rightarrow \mathbb{R}^3$, is a Willmore torus of revolution. By the gradient flow properties of the Willmore flow and Lemma \ref{lem:appdiam} we obtain that $\mathcal{W}(f_\infty) \leq \mathcal{W}(f_0) < 8\pi$.  We obtain from Proposition \ref{prop:willtori} that $f_\infty$ is, up to reparametrization, a Clifford torus, possibly rescaled and translated in the direction $(1,0,0)^T$. The claim follows.\\
\textbf{Case 2:} $\mathcal{W}(f_0) = 8\pi$. We first claim that $f_0$ is not a Willmore surface. Indeed, if it were then it would by Proposition \ref{prop:willtori} be a rescaled and translated reparametrization of a Clifford Torus. But the Willmore energy of the Clifford torus is $2\pi^2$, contradicting $\mathcal{W}(f_0)= 8\pi$. Hence 
\begin{equation*}
\frac{d}{dt} \mathcal{W}(f(t)) \Big\vert_{t = 0} = -|| \nabla_{L^2} \mathcal{W}(f_0)||_{L^2(\Sigma)}^2 < 0 ,
\end{equation*}
which implies that there exists $t_0 > 0 $ such that $\mathcal{W}(f(t_0)) < 8\pi$. We restart the Willmore flow with $f(t_0)$ which satisfies the assumptions of Case 1 and hence converges to a reparametrization of the Clifford torus, possibly rescaled and translated in direction $(1,0,0)^T$. The claim follows. 
\end{proof}

%%%%%%%%%%%%%%%%%%%%%%%%%%%%%%%%%%%%%%%%%%%%%%%%%%%%%%%%%%%%%%%%%
%%%%%%%%%%%%%%%%%%%%%%%%%%%%%%%%%%%%%%%%%%%%%%%%%%%%%%%%%%%%%%%%%

\subsection{Optimality}
%%%%%%%%%%%%%%%%%%%%%%%%%%%%%%%%%%%%%%%%%%%%%%%%%%%%%%%%%%%%%%%%%
%%%%%%%%%%%%%%%%%%%%%%%%%%%%%%%%%%%%%%%%%%%%%%%%%%%%%%%%%%%%%%%%%

We show that the upper bound of $8 \pi$ on the Willmore energy of the initial datum in Theorem \ref{thm:MainThmPrecisely} is sharp by proving Theorem \ref{thm:optii}. In the statement of this theorem, the geometric quantities that may possibly degenerate along the flow are the second fundamental form or the diameter. On contrary, the statement of Theorem \ref{thm:boundhypmain} suggests another  quantity which must degenerate -- the hyperbolic length. In the following we will construct the non-convergent evolutions and study the relation between the degenerating quantities. %In the proof another quantity will play an important role -- the hyperbolic length of the profile curve.

\begin{lemma}[The singular evolutions]\label{thm:singex}
 For any $\varepsilon > 0$ there exists a  torus of revolution $f_0: \mathbb{S}^1 \times \mathbb{S}^1 \rightarrow \mathbb{R}^3$ such that $\mathcal{W}(f_0) < 8\pi + \varepsilon$, and the maximal Willmore flow $(f(t))_{t \in [0,T)}$ {starting at $f_0$ satisfies 
$
\lim_{t \rightarrow T} \LH(\gamma(t)) = \infty. 
$}
% \end{enumerate}
\end{lemma}

The main idea is to start the flow with an immersed curve that has \emph{total curvature} 
\begin{equation}\label{eq:totalcurv}
T[\gamma] :=  \frac{1}{2\pi}\int_\gamma \kappa_{euc}[\gamma] \; \mathrm{d}\mathbf{s}
\end{equation}
equal to zero. This quantity $T[\cdot]$ turns out to be a flow invariant and can hence be helpful to classify possible limits of convergent evolution. This in turn can also be used to show that some evolutions cannot be convergent.  
\begin{lemma}\label{lem:totcurv}
The total curvature $T$, defined on curves in $W^{2,2}(\mathbb{S}^1,\mathbb{R}^2)_{imm} := \{ \gamma \in W^{2,2}(\mathbb{S}^1, \mathbb{R}^2) : \gamma  \mbox{ immersed} \}$ is integer-valued and  weakly continuous in the relative topology of $W^{2,2}(\mathbb{S}^1,\mathbb{R}^2)_{imm}$.
Moreover it is a flow invariant for the Willmore flow of tori of revolution, i.e. if $(f(t))_{t \in [0,T)}$ is an evolution by the Willmore flow with profile curve $(\gamma(t))_{t \in [0,T)}$
then $T[\gamma(t)] = T[\gamma(0)]$ for all $ t\in [0,T)$. 
\end{lemma}
\begin{proof}
The fact that $T[\cdot]$ is integer valued and an invariant with respect to regular homotopies is very classical and follows from the Whitney Graustein Theorem. Since $\gamma(t)= f(t)(u,0)$ {(see \eqref{eq:D17})} and $t \mapsto f(t)$  is a regular homotopy, so is $t \mapsto \gamma(t)$. Hence we can also conclude that it is a Willmore flow invariant. The weak $W^{2,2}$-continuity follows immediately from the formula 
\begin{equation*}
T[\gamma] := \frac{1}{2\pi} \int_0^1  \frac{1}{|\gamma'|} \left( (\gamma^{(2)})''(\gamma^{(1)})'- (\gamma^{(1)})''(\gamma^{(2)})' \right) \dx 
\end{equation*}
and the compact embedding $W^{2,2} \hookrightarrow C^1$. 
\end{proof}

\begin{proof}[Proof of Lemma \ref{thm:singex}]
Fix $\varepsilon > 0$. By {\cite[Cor. 6.4]{AdrianMarius}} there exists a curve $\gamma_\varepsilon$ such that $16 \leq \mathcal{E}(\gamma_\varepsilon) < 16 + \varepsilon$ and $T[\gamma_\varepsilon] = 0$, where $T[\cdot]$ is given as in  \eqref{eq:totalcurv}. Now start the flow with {$f_0=F_{\gamma_\varepsilon} :\mathbb{S}^1 \times \mathbb{S}^1 \rightarrow \mathbb{R}^3$} defined as in \eqref{eq:D1} with profile curve $\gamma_\varepsilon$ and let $(f(t))_{t \in [0,T)}$ be the corresponding evolution {by the Willmore flow}. Assume that  for  $(\gamma(t))_{t \in [0,\infty)}$ as in \eqref{eq:D17} one has
\begin{equation*}
\liminf_{t \rightarrow T} \LH(\gamma(t)) < \infty. 
\end{equation*}
By Theorem \ref{thm:boundhypmain} we obtain that then $T = \infty$ and $(f(t))_{t \in [0,\infty)}$ is convergent to a Willmore torus of revolution $f_\infty$.
Let now $t_j \rightarrow \infty$  be a sequence such that $\LH ( \gamma(t_j)) \leq M < \infty$ for all $j \in \mathbb{N}$.
By Lemma \ref{lem:LEMMANEW} we obtain that an appropriate reparametrization of $\gamma(t_j)$ converges in $C^k(\mathbb{S}^1,\mathbb{R}^2)$ to some  $\gamma_\infty \in C^\infty( \mathbb{S}^1, \mathbb{H}^2)$, which is a profile curve of $f_\infty$, i.e. up to reparametrization  $f_\infty = F_{\gamma_\infty}$. By \eqref{eq:willelaGrad}  we infer  that $\gamma_\infty$ is a hyperbolic elastica.
 
 Now we choose $\phi_j \in C^4(\mathbb{S}^1, \mathbb{S}^1)$ such that $\gamma(t_j) \circ \phi_j$ converges to $\gamma_\infty$ classically in $C^4(\mathbb{S}^1,\mathbb{R}^2)$. Then, by the previous lemma
\begin{equation*}
T[\gamma_\infty] = \lim_{j \rightarrow \infty} T[\gamma(t_j)] = T[\gamma(0)] = 0. 
\end{equation*}
Hence $\gamma_\infty$ is a hyperbolic elastica with vanishing  Euclidean total curvature. By \cite[Cor. 5.8]{AdrianMarius} there exist no hyperbolic elastica of vanishing total curvature. We obtain a contradiction and the claim follows. 
\end{proof}
{As an important ingredient for case (2) in Theorem \ref{thm:optii}, we need to show} that global evolutions under the Willmore flow of tori of revolution with unbounded hyperbolic length and no curvature concentration must have unbounded diameter.% This is actually true for any global evolution and not just for the tori of revolution constructed in  the proof above.

\begin{lemma}[Diameter Blow-Up]\label{lem:diamblow}
Let $f_0: \mathbb{S}^1 \times \mathbb{S}^1 \rightarrow \mathbb{R}^3$ be a torus of revolution and let $(f(t))_{t \in [0,\infty)}$ evolve by the Willmore flow with initial datum $f_0$. Let $\gamma(t)=f(t)(\cdot,0)$ be the profile curve of $f(t)$ for all $t\geq 0$. Assume  that 
$(A(t))_{t \in [0,\infty)}$ is bounded in $L^\infty(\Sigma)$ and 
$\lim_{t \rightarrow \infty} \LH( \gamma(t)) = \infty$.  Then 
\begin{equation*}
\lim_{t \rightarrow \infty} \mathrm{diam}(f(t)(\mathbb{S}^1 \times \mathbb{S}^1) ) = \infty. 
\end{equation*} 
\end{lemma}
\begin{proof}
We first introduce the constant $D := \sup_{t \in [0,\infty)} ||A(t)||_{L^\infty} < \infty$. Next we assume for a contradiction that there exists some $t_j \rightarrow T= \infty$ such that $\mathrm{diam}(f(t_j) ( \mathbb{S}^1 \times \mathbb{S}^1) \leq M < \infty$ for all $j \in \mathbb{N}$.  
Let $(r_j)_{j \in \mathbb{N}} \subset (0,\infty)$ be as in Theorem \ref{thm:THM22NEW}. Note that there exists $x_j \in \mathbb{R}^3$
\begin{equation*}
\varepsilon_0 \leq \int_{f(t_j)^{-1} (\overline{B_{r_j}(x_j)})} |A[f(t_j)]|^2 d\mu_{g_{f(t_j)}} \leq D^2 \mu_{g_{f(t_j)}} ( f(t_j)^{-1} ( \overline{ B_{r_j}(x_j) )} ) .
\end{equation*}
By \eqref{eq:C3} we have that 
\begin{equation*}
\mu_{g_{f(t_j)}} ( f(t_j)^{-1} ( \overline{ B_{r_j}(x_j) )} ) \leq C \mathcal{W}(f(t_j))) r_j^2 \leq C\mathcal{W}(f_0) r_j^2.
\end{equation*}
In particular we find by the previous two equations
{\begin{equation}
r_j^2 \geq \frac{\varepsilon_0}{D^2C \mathcal{W}(f_0)},
\end{equation}}
i.e. there exists $\delta > 0$ such that $r_j \geq \delta$ for all $j \in \mathbb{N}$.  Since we have assumed that $\mathrm{diam}(f(t_j)(\mathbb{S}^1 \times \mathbb{S}^1)) \leq M$ we obtain that 
\begin{equation*}
\mathrm{diam}\left( \frac{f(t_j)}{r_j} ( \mathbb{S}^1 \times \mathbb{S}^1 ) \right) \leq \frac{1}{r_j} \mathrm{diam} ( f(t_j) ( \mathbb{S}^1 \times \mathbb{S}^1) ) \leq  \frac{1}{\delta} M.
\end{equation*}
Now recall that $\tilde{f}_j(s) := \frac{f(t_j + s r_j^4)}{r_j}, s \in [0,c_0]$ defines a solution of the Willmore flow, %equation, 
with $\tilde{f}_j(0) = \frac{f(t_j)}{r_j}$ and $\tilde{f}_j(c_0) = \tilde{f}_{j,c_0}$, defined as in Theorem \ref{thm:THM22NEW}. With Lemma \ref{lem:diamflow} we obtain thus  that 
\begin{equation*}
\mathrm{diam} ( \tilde{f}_{j,c_0} ) \leq C(\mathcal{W}( \frac{f(t_j)}{r_j}) ) ( \mathrm{diam}( \frac{f(t_j)}{r_j} ) + c_0^\frac{1}{4} ) \leq C(\mathcal{W}(f_0)) \left( \frac{M}{\delta} + c_0^\frac{1}{4} \right),
\end{equation*}
which is uniformly bounded in $j$. This implies by Theorem \ref{thm:THM22NEW} that there exists a Willmore immersion $f_\infty : \mathbb{S}^1 \times \mathbb{S}^1 \rightarrow \mathbb{R}^3$ such that $f(t) \rightarrow f_\infty$ in $C^k$ for all $k \in \mathbb{N}$. {By Lemma \ref{lem:LEMMANEW} $f_\infty$ is a Willmore torus of revolution}.  In particular, up to reparametrization one has $f_\infty = F_{\gamma_\infty}$ for some $\gamma_\infty \in C^\infty (\mathbb{S}^1, \mathbb{H}^2)$. We next claim that there exists $\delta > 0$ such that $\inf_{\mathbb{S}^1} \gamma(t)^{(2)} > \delta$ for all $t \in [0,\infty)$. To this end observe 
\begin{align*}
\lim_{t \rightarrow \infty} \inf_{\mathbb{S}^1} \gamma(t)^{(2)} & = \lim_{t \rightarrow \infty} \inf_{\mathbb{S}^1 \times \mathbb{S}^1} \sqrt{(f(t)^{(2)} )^2 + (f(t)^{(3)})^2} \\ & = \inf_{\mathbb{S}^1 \times \mathbb{S}^1} \sqrt{(f_\infty^{(2)} )^2 + (f_\infty^{(3)})^2} = \inf_{\mathbb{S}^1} \gamma_\infty^{(2)} > 0,
\end{align*}
{since $\gamma_\infty^{(2)}(u) > 0$ for all $u \in \mathbb{S}^1$ and $\mathbb{S}^1$ is compact.}
Note that we have used here that the infimum expression is independent of the parametrization of $f(t)$. 
This and the fact that $(f(t))_{t \in [0,\infty)}$ is a smoothly evolving family of tori of revolution implies $\inf_{\mathbb{S}^1} \gamma(t)^{(2)} > \delta$  for all $t \in [0,\infty)$. %the claim. 
Next we look at the surface area of $f(t)$, i.e. 
\begin{equation*}
 \mu_{g_{f(t)}}( \mathbb{S}^1 \times \mathbb{S}^1) =2\pi  \int_0^1  |\gamma(t)'(u)| \gamma^{(2)}(t)(u) \; \mathrm{d}u ,
\end{equation*}
and infer
\begin{equation*}
\mu_{g_{f(t_j)}}(\mathbb{S}^1 \times \mathbb{S}^1) \geq 2\pi \delta^2 \LH(\gamma(t_j)) \rightarrow \infty . 
\end{equation*}
{With Lemma \ref{lem:simon}  it follows %an inequality that we prove in the following lemma we infer 
\begin{equation*}
M \geq \mathrm{diam}(f(t_j)(\mathbb{S}^1 \times \mathbb{S}^1) ) \geq \sqrt{\frac{\mu_{g_{f(t_j)}}(\mathbb{S}^1 \times \mathbb{S}^1)}{\mathcal{W}(f(t_j))}} \geq  \sqrt{\frac{\mu_{g_{f(t_j)}}(\mathbb{S}^1 \times \mathbb{S}^1)}{\mathcal{W}(f_0)}}  \rightarrow \infty.
\end{equation*}
A contradiction. We infer that $\lim_{t \rightarrow \infty} \mathrm{diam}(f(t)( \mathbb{S}^1 \times \mathbb{S}^1)) = \infty$. }
 \end{proof}

In the proof we have used without further notice that the concept of tori of revolution in \cite[Def. 2.2]{Blatt} coincides with our definition in Definition \ref{def:toriofrevolution}, at least up to reparametrization. For details recall Proposition \ref{prop:consiBlatt} and the discussion afterwards.

\begin{proof}[Proof of Theorem \ref{thm:optii}]
Let $\varepsilon > 0 $ be as in the statement and $f_0$ be as in Lemma \ref{thm:singex}. Then the evolution $(f(t))_{t \in [0,T)}$ satisfies  
 $\lim_{t \rightarrow T} \LH(\gamma(t)) = \infty$. 
 Next let $t_j \uparrow T$ be a sequence. {Let $\epsilon_0 > 0, c_0 > 0$ and $(r_j)_{j \in \mathbb{N}}$ be as in Theorem \ref{thm:THM22NEW}. We distinguish now two cases.\\
 \textbf{Case 1: There exists a subsequence of $r_j$ that converges to zero.}
 We claim that then condition (1) in the statement occurs. To this end assume that $(||A(t)||_{L^\infty})_{t \in [0,T)}$ is bounded, say $D := \sup_{t \in [0,T)}||A(t)||_{L^\infty} < \infty$. Then one has by \eqref{eq:radii} that for all $j \in \mathbb{N}$ there exists $x_j \in \mathbb{R}^3$ such that 
 \begin{equation*}
     \epsilon_0 \leq \int_{f(t_j)^{-1}(\overline{B_{r_j}(x_j)})} |A(t_j)|^2 d\mu_{g_{f(t_j)}} \leq D^2 \mu_{g_{f(t_j)}} (f(t_j)^{-1} (\overline{B_{r_j}(x_j)}).
 \end{equation*}
 Using \eqref{eq:C3} we find that 
$
     \epsilon_0 \leq c \mathcal{W}(f_0) D^2 r_j^2. 
$
 This is a contradiction to the condition that up to a subsequence $r_j \rightarrow 0$. Hence we have shown that $(||A(t)||_{L^\infty(\Sigma)})_{t \in [0,T)}$ is unbounded. \\
\textbf{Case 2: There exists $\delta >0$ such that $r_j \geq \delta$ for all $j \in \mathbb{N}$.}
First observe that in this case $T=\infty$ since $t_j + c_0 r_j^4 < T$ by Theorem \ref{thm:THM22NEW}. If condition (1) in the statement holds true, i.e. $(||A(t)||_{L^\infty(\Sigma)})_{t \geq 0 }$ is unbounded, there is nothing to prove. Hence we may assume that $(||A(t)||_{L^\infty(\Sigma)})_{t \geq 0}$ is bounded.} Since $\lim_{t \rightarrow \infty} \LH(\gamma(t)) = \infty$, 
by Lemma \ref{lem:diamblow} we find that $\lim_{t \rightarrow \infty} \mathrm{diam}(f(t))(\mathbb{S}^1 \times \mathbb{S}^1) = \infty$ and hence condition (2) occurs. This proves the claim. 
\end{proof}

%%%%%%%%%%%%%%%%%%%%%%%%%%%%%%%%%%%%%%%%%%%%%%%%%%%%%%%%%%%%%%%%%
%%%%%%%%%%%%%%%%%%%%%%%%%%%%%%%%%%%%%%%%%%%%%%%%%%%%%%%%%%%%%%%%%

\section{An application: Energy minimization among conformal constraints}
%%%%%%%%%%%%%%%%%%%%%%%%%%%%%%%%%%%%%%%%%%%%%%%%%%%%%%%%%%%%%%%%%
%%%%%%%%%%%%%%%%%%%%%%%%%%%%%%%%%%%%%%%%%%%%%%%%%%%%%%%%%%%%%%%%%

A very vivid field of research is 
the minimization of the Willmore energy among all tori that are \emph{conformally equivalent} to a reference torus.  Being conformally equivalent means that the surface can be parametrized with a \emph{conformal} immersion of the reference torus. Taking a reference torus of the form $\frac{\mathbb{C}}{\mathbb{Z} + \omega \mathbb{Z}}$ one can also associate to every torus its \emph{conformal class}, defined as {follows}.

\begin{definition}[{Conformal class, cf. \cite[p.293]{SchCh1}}]\label{def:confclass}
Let $S \subset \mathbb{R}^3$ be a smooth torus. Then there exists a unique  $\omega \in \mathbb{C}$ satisfying $|\omega| \geq 1$,  $\mathrm{Im}(\omega) > 0 $ and $\mathrm{Re}(\omega) \in [0 ,\frac{1}{2}]$ such that there exists a conformal smooth immersion
\begin{equation*}
F : \frac{\mathbb{C}}{\mathbb{Z} + \omega \mathbb{Z}} \rightarrow S, 
\end{equation*}
i.e. 
\begin{equation}\label{eq:conf}
g_{i,j}^F = e^{2u} \delta_{i,j} \quad \textrm{for some $u \in C^\infty(\frac{\mathbb{C}}{\mathbb{Z} + \omega \mathbb{Z}})$}.
\end{equation}
The value $\omega =\omega(S) \in \mathbb{C}$ is then called the \emph{conformal class of $S$}. If $\omega$ is purely imaginary, we call the torus \emph{rectangular}.
\end{definition}

 As it turns out, all tori of revolution are rectangular (see also \cite[Prop.7]{LangerSingerWillmore}). 
\begin{prop}\label{prop:conftor}
Suppose that $\gamma\in C^\infty( \mathbb{S}^1, \mathbb{H}^2)$. Then {$F_{\gamma}(\mathbb{S}^1 \times \mathbb{S}^1)$, the torus with profile curve $\gamma$,} has conformal class 
\begin{equation*}
\omega(F_{\gamma}(\mathbb{S}^1 \times \mathbb{S}^1)) =\begin{cases}  i \frac{1}{2\pi}\LH(\gamma) & \LH(\gamma) \geq 2\pi, \\ i \frac{2\pi}{\LH(\gamma)}  & \LH(\gamma) < 2\pi. \end{cases}
\end{equation*}
In particular, each torus of revolution is rectangular and {$\omega(F_{\gamma}(\mathbb{S}^1 \times \mathbb{S}^1))$} is a continuous function of $\LH(\gamma)$.
\end{prop}
\begin{proof}
Let $\overline{\gamma}: \mathbb{R} \rightarrow \mathbb{R}$ be the $\frac{1}{2\pi}\LH(\gamma)$-periodic reparametrization of $\gamma$ with constant hyperbolic velocity $2\pi$. {If $\LH(\gamma) \geq 2\pi$} we choose {the smooth immersion
\begin{equation*}
F : \frac{\mathbb{C}}{\mathbb{Z}+ \frac{i\LH(\gamma)}{2\pi} \mathbb{Z}} \rightarrow F_{\gamma}(\mathbb{S}^1 \times \mathbb{S}^1)
\end{equation*} }
by 
\begin{equation}\label{eq:35}
F(s+ it) = \begin{pmatrix}
\overline{\gamma}^1(t) \\ \overline{\gamma}^2(t) \cos(2\pi s) \\ \overline{\gamma}^2(t) \sin(2\pi s) 
\end{pmatrix} .
\end{equation} 
An easy computation shows $g_{1,2}^F = g_{2,1}^F = 0$ and 
\begin{equation*}
g_{1,1}^F = (\overline{\gamma}^1)'^2 + (\overline{\gamma}^2)'^2 , \quad g_{2,2}^F = 4\pi^2(\overline{\gamma}^2)^2. 
\end{equation*}
Therefore by our choice of parametrization
\begin{equation*}
\frac{g_{1,1}^F}{g_{2,2}^F }= \frac{(\overline{\gamma}^1)'^2 + (\overline{\gamma}^2)'^2}{4\pi^2(\overline{\gamma}^2(t))^2} = 1. 
\end{equation*}
Hence \eqref{eq:conf} is satisfied {and $F$ is a conformal immersion}. Moreover one readily checks that $\omega = i \frac{\LH(\gamma)}{2\pi}$ meets the requirements of Definition \ref{def:confclass}. 

If $\LH(\gamma) < 2\pi$ we choose
\begin{equation*}
\widetilde{F} : \frac{\mathbb{C}}{\mathbb{Z} + i\frac{2\pi}{\LH(\gamma) }\mathbb{Z}} \rightarrow F_{\gamma}(\mathbb{S}^1 \times \mathbb{S}^1)
\end{equation*}
to be given by 
\begin{equation*}
\widetilde{F} (s + it ): = F(\frac{\LH(\gamma)}{2\pi}t + i\frac{\LH(\gamma)}{2\pi}s),
\end{equation*}
where $F$ is as in \eqref{eq:35} and the claim follows also in this case arguing as before.
\end{proof}

\begin{remark}\label{rem:lengthCT}
The conformal class of the Clifford Torus is $\omega = i$. Indeed,  
its defining curve is
\begin{equation*}
\gamma(t) = \begin{pmatrix}
0 \\ 1 
\end{pmatrix} + \frac{1}{\sqrt{2}} \begin{pmatrix}
\cos(t) \\ \sin(t)
\end{pmatrix} \quad t \in [- \pi, \pi).
\end{equation*}
From this we conclude with the {residue theorem (more precisely \cite[Prop. III.7.10]{Freitag})} that 
\begin{align*}
\LH(\gamma) & = \int_{-\pi}^{\pi} \frac{1}{\sqrt{2} + \sin(t)} \; \mathrm{d}t  = \int_{-\pi}^\pi \frac{1}{\sqrt{2}+ 2 \cos \frac{t}{2} \sin \frac{t}{2}} \; \mathrm{d}t \\ & = \int_{-\pi}^\pi \frac{1}{\sqrt{2} + 2 \frac{\tan \frac{t}{2}}{1 + \tan^2 \frac{t}{2}}} \; \mathrm{d}t = \int_{-\pi}^\pi \frac{1+ \tan^2 \frac{t}{2}}{\sqrt{2}(1+ \tan^2 \frac{t}{2}) + 2 \tan \frac{t}{2}} \; \mathrm{d}t
\\  &= 2 \int_{- \infty}^\infty  \frac{1}{\sqrt{2}(1+z^2) + 2z} \; \mathrm{d}z = 2 \int_{-\infty}^\infty  \frac{1}{\sqrt{2} ( z - \frac{-1+i}{\sqrt{2}}) ( z - \frac{-1-i}{\sqrt{2}}) }
\\& = 2 (2\pi i) \sum_{ a: \mathrm{Im}(a) > 0 }  \mathrm{Res}\left( \frac{-1}{\sqrt{2} ( z - \frac{-1+i}{\sqrt{2}}) ( z - \frac{1-i}{\sqrt{2}}) } , a \right) 
\\ & =  4\pi i \frac{1}{\sqrt{2}( \frac{-1+i}{\sqrt{2}}- \frac{-1-i}{\sqrt{2}})} = 2\pi . \qedhere
\end{align*} 
\end{remark}

An interesting problem is the minimization of the Willmore functional in each conformal class.
\begin{definition}[Conformally constrained Willmore minimization]
For $\omega$ as in Definition \ref{def:confclass} we set 
\begin{equation*}
M_{3,1}(\omega) := \inf \{ \mathcal{W}(f) : f: \frac{\mathbb{C}}{\mathbb{Z}+ \omega \mathbb{Z}} \rightarrow \mathbb{R}^3 \; \textrm{conformal immersion}\}
\end{equation*}
\end{definition}
In \cite[Prop. D.1]{SchCh1} the authors show that there exists some $b_0\geq 1$ such that $b \geq b_0$ implies $M_{3,1}(ib)< 8\pi$. {Our first contribution in this context is the new insight that $b_0=1$. We prove the existence of tori of revolution with Willmore energy smaller than $8 \pi$ in each conformal class $\omega=ib$, $b \geq 1$, via the Willmore flow studied in Theorem \ref{thm:MainThmPrecisely}.} Note that $\frac{\mathbb{C}}{\mathbb{Z} + ib\mathbb{Z}}$ and $\mathbb{S}^1 \times \mathbb{S}^1$ are diffeomorphic with diffeomorphism $\phi: \mathbb{S}^1 \times \mathbb{S}^1 \rightarrow \frac{\mathbb{C}}{\mathbb{Z} + ib\mathbb{Z}}$ being given by $\phi(u,v) = u+ i b v$. Hence the results about the Willmore flow in Theorem \ref{thm:MainThmPrecisely} apply also for surfaces defined on $\frac{\mathbb{C}}{\mathbb{Z} + ib\mathbb{Z}}$.

\begin{theorem}\label{thm:65}
For each $b \geq 1$ there exists a torus of revolution $T_b$ such that $\omega(T_b) = ib $ and $\mathcal{W}(T_b) < 8\pi$. 
\end{theorem}
\begin{proof}
From the construction in the Proof of \cite[Prop. D.1]{SchCh1} follows that there exists $b_0>1$ such that for all $b \geq b_0$ there exists a torus $T_b$ as in the statement. Note that actually the authors construct only a $C^{1,1}$-torus of revolution $T_b$ but by mollification of the profile curve one can easily obtain a smooth torus of revolution that satisfies the same requirements and differs not too much in the conformal class as the hyperbolic length depends continuously on $\gamma$. 

It remains to prove the claim for $b \in [1,b_0)$. For this choose $f_0 : \frac{\mathbb{C}}{\mathbb{Z} + ib_0 \mathbb{Z}} \rightarrow \mathbb{R}^3$ to be a smooth conformal parametrization of $T_{b_0}$ and let $(f(t))_{t \in (0, \infty)}$ be the evolution of $f_0$ by the Willmore flow, which is global and smoothly convergent to the {Clifford torus (possibly rescaled and translated in direction $(1,0,0)$)} by Theorem \ref{thm:MainThmPrecisely}. Moreover, $f(t)$ is a torus of revolution for all $t\geq 0$. Let $\gamma(t)= f(t)(\cdot,0)$ be the profile curve of $f(t)$ for all $t\geq 0$, i.e. $f(t)=F_{\gamma(t)}$. 
By  \eqref{eq:D17},  $t \mapsto \gamma(t)$ is a smooth family of curves for $t\geq0$ and in particular 
$\LH(\gamma(t))$ depends smoothly on $t$. By Proposition \ref{prop:conftor} one obtains that $ t\mapsto  \frac{1}{i}\omega( F_{\gamma(t)})$ is real valued and depends continuously on $t$. 
We show next that along a subsequence $t \mapsto \frac{1}{i} \omega( F_{\gamma(t)}) $ tends to $1$ as $t \rightarrow \infty$. 
By Lemma \ref{lem:LEMMANEW} we obtain that there exists some $t_j \rightarrow \infty$ such that an appropriate reparametrization of $\gamma(t_j)$ converges in $C^2(\mathbb{S}^1, \mathbb{R}^2)$ to $\gamma_\infty \in C^\infty(\mathbb{S}^1, \mathbb{H}^2)$, a profile curve of the Clifford torus (possibly rescaled and translated in direction $(1,0,0)$). Thus we have
 \begin{equation}
 2\pi = \LH ( \gamma_\infty) = \lim_{j \rightarrow \infty} \LH(t_j),
 \end{equation}
 i.e $\frac{1}{i} \omega ( F_{\gamma(t_j)} ) \rightarrow 1$ as $j \rightarrow \infty$. Since $\frac{1}{i}\omega(F_{\gamma(0)} ) = b_0$  each value between $1$ and $b_0$ is attained by the intermediate value theorem. From this the existence of a torus of revolution $T_{b}$ for each $b \in [1,b_0)$ follows.
\end{proof}
\begin{remark}
    Theorem \ref{thm:65} can also be proven using the results in \cite{AdrianMarius} concerning the elastic flow in $\mathbb{H}^2$ (which also dissipates the Willmore energy).
\end{remark}
 
In \cite{KS4} the authors prove that the infimum in a conformal class $\omega$ is attained once one can find a competitor with energy below $8\pi$. For $\omega = ib$ our small energy tori serve as such competitors and show that the infimum is attained. 

\begin{cor}
For each $b \geq 1$ the infimum $M_{3,1}(ib)$ is attained and the map $b \rightarrow M_{3,1}(ib)$ is continuous on $[1,\infty)$. 
\end{cor}
\begin{proof}
{Theorem 7.3 and Proposition 5.1 in \cite{KS4} show that each $b \geq 1$  where  $ M_{3,1}(ib) < 8\pi$ is a point of continuity of $b \mapsto M_{3,1}(ib)$ and a point where the infimum in the definition of $M_{3,1}$ is attained. The claim then follows directly from this results and Theorem \ref{thm:65}}.
\end{proof}

The symmetries of the Willmore energy might suggest that the infimum of the Willmore energy in each class of rectangular tori (i.e. $\omega = ib$) is attained at a torus of revolution. This is in general still open. Far reaching results are obtained using 
a formulation of the Willmore energy in $\mathbb{S}^3$ by means of the \emph{stereographic projection}. Since the stereographic projection is conformal it does also not change the conformal class. Looking at the Willmore energy in $\mathbb{S}^3$ one can find tori with a lot of symmetries: For $\alpha \in (0,1)$ one can look at $\alpha \mathbb{S}^1 + \sqrt{1-\alpha^2} \mathbb{S}^1$. The stereographic projection of all of those are tori of revolution. In particular, those are good candidates for minimizers in their conformal classes $\omega = i\frac{ \sqrt{1-\alpha^2}}{\alpha}$. For $\alpha = \frac{1}{\sqrt{2}}$ we obtain the Clifford torus which is the global minimizer and hence surely the minimizer in its conformal class. In \cite{SchCh2}, \cite{SchCh1} the authors show that for conformal classes close to the Clifford torus one still gets minimizers of the form $\alpha \mathbb{S}^1 \times \sqrt{1-\alpha^2} \mathbb{S}^1$. 
{More precisely, the result \cite[Thm 3.1]{SchCh1} shows that there exists $b_1 > 1$ such that for all $b \leq  b_1$ one has that   $M_{3,1}(b)$ is attained by $\Sigma_b := P \left(  \frac{1}{\sqrt{1+b^2}} \mathbb{S}^1 \times \frac{b}{\sqrt{1+b^2}} \mathbb{S}^1 \right)$, where $P : \mathbb{S}^3 \rightarrow \mathbb{R}^3$ denotes the stereographic projection. The authors also obtain that $b_1 < \infty$.} The critical value $b_1$ can be understood as a point where a symmetry of the minimizers breaks down. They also note that this property has to break down for large conformal classes, cf. \cite[p.293-294]{SchCh1}. In the following we will be able to find an explicit upper bound on the symmetry-breaking value $b_1$. This result is now obtained by energy comparison. There are other (sharper) results using a stability discussion of $\Sigma_b$ in $\mathbb{S}^3$, cf. \cite{KuwertLorenz}.   

\begin{cor}
Let {$b_1\geq 1$ be such that for $b\leq b_1$ the minimizer for $M_{3,1}(b)$ is attained by $\Sigma_b := P \left(  \frac{1}{\sqrt{1+b^2}} \mathbb{S}^1 \times \frac{b}{\sqrt{1+b^2}} \mathbb{S}^1 \right)$, where $P : \mathbb{S}^3 \rightarrow \mathbb{R}^3$ denotes the stereographic projection.} Then 
\begin{equation}\label{eq:dings}
b_1 < \frac{4}{\pi} + \sqrt{\frac{16}{\pi^2}- 1} \simeq 2.06136 \, .
\end{equation}
\end{cor}
\begin{proof} 
Let $b>1$ be such that $\Sigma_b$ is a minimizer and let $T_b$ be the torus constructed in Theorem \ref{thm:65}. Then, necessarily, $\mathcal{W}(\Sigma_b) \leq \mathcal{W}(T_b) < 8\pi$. This inequality implies the claim once we have shown that $\mathcal{W}(\Sigma_b) = \pi^2 ( b+ \frac{1}{b})$.\\
For this according to \cite[Eq. (9)]{Topping} for all $f: \Sigma \rightarrow \mathbb{R}^3$ 
\begin{equation*}
\mathcal{W}(f) = \int_{\Sigma} \left( \frac{1}{4} |\widetilde{H}|_{P^{-1}(f)}^2 + 1 \right) \; \mathrm{d}\mu_f
\end{equation*}
where $\widetilde{H}$ denotes the mean curvature of {$P^{-1}(f)$} in $\mathbb{S}^3$ and $\mu_f$ denotes the surface measure in $\mathbb{S}^3$. By {\cite[Eq. (2.3)]{SchCh2}} we have $|\vec{H}_{\mathbb{R}^4}|^2 = |\widetilde{H}|^2 + 4$ and hence we obtain 
\begin{equation*}
\mathcal{W}(f) = \frac{1}{4} \int_{\Sigma} |\vec{H}_{\mathbb{R}^4}(P^{-1}(f))|^2 \; \mathrm{d}\mu_f
\end{equation*}
Having now arrived in $\mathbb{R}^4$ and using that $P^{-1}(\Sigma_b) = \frac{1}{\sqrt{1+b^2}} \mathbb{S}^1 \times \frac{b}{\sqrt{1+b^2}} \mathbb{S}^1$ we can define $r := \frac{1}{\sqrt{1+b^2}}$ and use the parametrization 
\begin{equation*}
F: \mathbb{S}^1 \times \mathbb{S}^1 \ni (\phi , \theta) \mapsto \begin{pmatrix}
r \cos(2\pi \phi) \\ r \sin(2\pi \phi) \\ \sqrt{1-r^2} \cos(2\pi \theta)  \\ \sqrt{1-r^2} \sin(2 \pi \theta) 
\end{pmatrix} \in \mathbb{R}^4.
\end{equation*}
A computation reveals that 
\begin{equation*}
g = 4\pi^2 \begin{pmatrix}
r^2  & 0 \\ 0 & 1-r^2
\end{pmatrix}.
\end{equation*}
We obtain that $\frac{1}{2\pi r} \frac{\partial}{\partial \phi}$ and $\frac{1}{2\pi \sqrt{1-r^2}} \frac{\partial}{\partial \theta}$ is an orthonormal basis $T_{(\phi, \theta)}(\mathbb{S}^1, \mathbb{S}^1)$ and hence 
\begin{equation*}
\vec{H}_{\mathbb{R}^4}(F) = \frac{1}{4\pi^2 r^2} \frac{\partial^2F }{\partial \phi^2}  + \frac{1}{4\pi^2(1-r^2)} \frac{\partial^2 F}{\partial \theta^2},
\end{equation*}
which implies that
\begin{equation*}
|\vec{H}_{\mathbb{R}^4} (F) |^2 = \frac{1}{r^2} + \frac{1}{1-r^2}.
\end{equation*}
Also note that $\sqrt{\mathrm{det}(g)} = 4\pi^2 r \sqrt{1-r^2}$. The Willmore energy then reads 
\begin{equation*}
\mathcal{W}(\Sigma_b) = \frac{1}{4} \left(   \frac{1}{r^2} + \frac{1}{1-r^2} \right) 4\pi^2 r \sqrt{1-r^2} = \pi^2 \left( \frac{\sqrt{1-r^2}}{r} + \frac{r}{\sqrt{1-r^2}} \right) 
\end{equation*}
and the claim follows using that by definition of $r$ one has $r = \frac{1}{\sqrt{1+b^2}}$. 
\end{proof}

%%%%%%%%%%%%%%%%%%%%%%%%%%%%%%%%%%%%%%%%%%%%%%%%
%%%%%%%%%%%%%%%%%%%%%%%%%%%%%%%%%%%%%%%%%%%%%%
\appendix
%%%%%%%%%%%%%%%%%%%%%%%%%%%%%%%%%%%%%%%%%%%%%%%%%
%%%%%%%%%%%%%%%%%%%%%%%%%%%%%%%%%%%%%%%%%%%%%%%%%

%%%%%%%%%%%%%%%%%%%%%%%%%%%%%%%%%%%%%%%%%%%%%%%%%%%%%%%%%%%%%%%%%%%%%
\section{Consistency between extrinsic and intrinsic view}
\label{sec:AppendixGeometry}
%%%%%%%%%%%%%%%%%%%%%%%%%%%%%%%%%%%%%%%%%%%%%%%%%%%%%%%%%%%%%%%%%%%%%%%%
 
In literature there are multiple ways to define geometric quantities like curvature. This also leads to different notions of the Willmore energy and its gradient flow. Here we want to convince the reader that all those notions are consistent with the one we chose. For this we first have to do some computations in local coordinates. Let $M$ be a smooth $2$-D manifold, ${f:M \rightarrow \mathbb{R}^3}$ be an immersion and $\psi : M \rightarrow \mathbb{R}^2$ be a chart for $M$ with coordinates $(u^1,u^2)$. Given vector-field $X = x^i \frac{\partial}{\partial u^i} $ and $Y = y^j \frac{\partial}{\partial u^j}$ then
\begin{equation}\label{eq:Ainccoo}
A(X,Y) = x^i y^j \left( \frac{\partial^2 f}{\partial u^i \partial u^j} - \Gamma_{i,j}^k \frac{\partial f}{\partial u^k} \right),  
\end{equation}
where $\Gamma_{i, j}^k $ are the Christoffel symbols {defined using the metric $g_{ij}=\langle \frac{\partial}{\partial u^i}f, \frac{\partial}{\partial u^j}f\rangle$}. In particular, we see that the second fundamental form is symmetric.

If $ f: M \rightarrow \mathbb{R}^3$ is an {isometric immersion}  then for each local chart $(u^1,u^2)$ of $M$ one can define a unit normal field $\vec{N} = \frac{\partial_{u^1} f \times \partial_{u^2}f }{|\partial_{u^1} f \times \partial_{u^2}f|}$ for $(u^1, u^2)$ and rewrite
{\begin{align}\nonumber
A(X,Y) & = x^i y^j \left( \frac{\partial^2 f}{\partial u^i \partial u^j} - g^{kl}\langle \frac{\partial^2 f}{\partial u^i \partial u^j} , \frac{\partial f}{\partial u^l} \rangle_{\R^3}  \frac{\partial f}{\partial u^k} \right) %\label{eq:1.14} 
\\ & = x^i y^j  \langle \frac{\partial^2 f}{\partial u^i \partial u^j}, \vec{N} \rangle_{\R^3} \vec{N} \label{eq:1.15} .
\end{align}}
If $f:M \rightarrow f(M) \subset \mathbb{R}^3$ is now an isometric embedding and $f(M)$ is orientable, $\vec{N}$ is independent of the chosen chart and \eqref{eq:1.15} coincides with the usual definition of the second fundamental form. 

Let us now choose normal coordinates $(u^1,u^2)$ and fix $e_1 = \frac{\partial f}{\partial u^1}$ and $e_2 = \frac{\partial f}{\partial u^2}$. Then by \eqref{eq:1.15} we find 
$$A(e_i,e_j) = h_i^j \vec{N}, $$ 
where $h_i^j$ denote the usual coefficients of the Weingarten map. Then, the mean curvature (vector) and Gauss curvature are given by  
{\begin{align}\label{eq:Hvec} 
\vec{H} & = A(e_1,e_1) + A(e_2,e_2) = (h_1^1+h_2^2) \vec{N}= H \vec{N} ,\\ \nonumber
K & : =\langle A(e_1,e_1) , A(e_2,e_2) \rangle_{\R^3} -  \langle A(e_1,e_2) , A(e_2,e_1) \rangle_{\R^3} = h^1_1 h^2_2 -(h^1_2)^2,
\end{align}}
where $H$ denotes the \emph{scalar} mean curvature. For $Q(\Aring)H$, the {\lq cubic\rq}-term in the Willmore equation, one easily derives
$$ Q(\Aring{}) \vec{H} = \frac12 H (H^2-4K). $$
With similar computations,
{\begin{align*}
|A|^2 & = |H|^2 - 2K = \sum_{i,j^=1}^2 \langle A(e_i,e_j), A(e_i, e_j) \rangle_{\R^3} ,
\end{align*}}
and hence for each toroidal immersion {$f: \mathbb{S}^1 \times \mathbb{S}^1 \rightarrow \mathbb{R}^3$} one has by the Gauss-Bonnet Theorem 
\begin{equation}\label{eq:secwillmore}
\int_\Sigma |A|^2 \; \mathrm{d}\mu_f = 4 \mathcal{W}(f) .
\end{equation}
Similarly, again in the case of tori, an easy computation shows that $|\Aring{}|^2 = \frac{1}{2}H^2 - 2 K $  and 
\begin{equation*}
\int_\Sigma |\Aring{}|^2 \; \mathrm{d}\mu_f = 2   \W(f) .
\end{equation*}
Also, note that $|\Aring{}|^2 \leq |A|^2$. 

%%%%%%%%%%%%%%%%%%%%%%%%%%%%%%%%%%%%%%%%%%%%%%%%%%%%%%%%%%%%%%%%%%%%
%%%%%%%%%%%%%%%%%%%%%%%%%%%%%%%%%%%%%%%%%%%%%%%%%%%%%%%%%%%%%%%%%%%%
\section{Tensor calculus}\label{app:tensor}
%%%%%%%%%%%%%%%%%%%%%%%%%%%%%%%%%%%%%%%%%%%%%%%%%%%%%%%%%%%%%%%%%%%%%
%%%%%%%%%%%%%%%%%%%%%%%%%%%%%%%%%%%%%%%%%%%%%%%%%%%%%%%%%%%%%%%%%%%%%

Throughout the article, we use a nonstandard notation for some differential geometric concepts involving connections, derivatives and tensors. We discuss here that our notation is consistent with the one used in %It is important to check that the notation is consistent with 
the articles \cite{KS1}, \cite{KS2}, \cite{KS3}, since many results cited there are used. Here we shall briefly introduce these concepts and clarify their meaning. Let $M$ be a smooth $2$-dim. manifold and $f \in C^\infty(M;\mathbb{R}^{{n}})$ be an immersion. Moreover, let $\nabla$ be the Levi-Civita connection on $M$. For a vector field $X \in \mathcal{V}(M)$ we define the \emph{full derivative} $D_X: C^\infty(M;\mathbb{R}^{{n}}) \rightarrow C^\infty(M;\mathbb{R}^{{n}})$ via 
\begin{equation}\label{eq:DXDEF}
D_X G := \sum_{i = 1}^{{n}} X(G_i) \vec{e}_i, \quad \textrm{whenever} \quad G = \sum_{i = 1}^{{n}} G_i \vec{e}_i \in C^\infty(M;\mathbb{R}^3), 
\end{equation}
and ${\{ \vec{e}_1, \vec{e}_2 ,..., \vec{e}_n \}}$ is the canonical basis of $\mathbb{R}^{{n}}$. We say that $G \in C^\infty(M;\mathbb{R}^{{n}})$ is a \emph{normal vector field} if $G(p) \perp df_p(T_pM)$ for all $p \in M$. We define for short $N_pM := df_p(T_pM)^\perp$ and $NM := \bigsqcup_{p \in M } N_pM $ the \emph{normal bundle}. For such a normal vector field $G \in C^\infty(M,NM)$ we define the \emph{normal connection} of $G$ to be 
\begin{equation}\label{eq:normalconn}
\nabla^\perp_X G \big\vert_p := \pi_{N_pM} ( D_X G \big\vert_p ) = D_X G^\perp,
\end{equation}
where $\pi_U$ denotes the orthogonal projection on $U$. A normal vector field that will be used very frequently is $Y= A(Z,W)$ for some $Z,W \in \mathcal{V}(M)$. This is however not just a normal vector field but each of its components is also a $(2,0)$-tensor -- we may think of $p \rightarrow A_p(Z,W)$ as a $(2,0)$-tensor on $M$ with values in the normal bundle $NM$, i.e. a for each $p \in M$ it is a multilinear map from $T_pM^2$ to $N_pM$. 
If we do so, the standard concept of \emph{tensorial connections} (cf. \cite[Lemma 4.6]{Lee} is not applicable, since it is needed that the tensor takes values in $\mathbb{R}$. One can however overcome this by using two \emph{different connections}, namely $\nabla$ and $\nabla^\perp$. More precisely, for a $(k,0)$-tensor $F: p \mapsto (F_p : T_pM^k \rightarrow N_pM)$ on $M$ with values in the normal bundle $NM$ we can define a $(k+ 1,0)$-tensor $\nabla^\perp F$ via
\begin{equation}\label{eq:tnsderiv}
\nabla^\perp F(X_1,...,X_{k+1}) := \nabla^\perp_{X_1} F( X_2,...,X_{k+1} )- \sum_{j = 2}^{k+1} F( X_2,...,\nabla_{X_1}X_j,...,X_{k+1})  ,
\end{equation}
for $X_1,...,X_{k+1} \in \mathcal{V}(M)$.
It can easily be checked that $\nabla^\perp F$ is indeed a $(k+1)$-tensor, i.e. $\nabla^\perp F_p $ depends only on $X_1(p),...,X_{k+1}(p)$. Moreover, if $F$ is a $(0,0)$ tensor on $M$ with values in $NM$, i.e. $F \in C^\infty(M;NM)$ then the  notation of $\nabla^\perp F$ coincides with the previous definition in \eqref{eq:normalconn}. We remark that in \cite{KS1}, \cite{KS2}, \cite{KS3}, $\nabla^\perp$ and $\nabla$ are both denoted by $\nabla$. 
%In the course of the article 
The $L^\infty(M)$-norm of a $(k,0)$ tensor $F$ on $M$ with values in $NM$ is defined to be 
\begin{equation*}
||F||_{L^\infty(M)} := \sup_{p \in M} \sup_{\{ E_1,E_2 \}\; \textrm{orthonormal basis of} \; T_pM } \sum_{i_1,...,i_k = 1}^2 |F(E_{i_1},...,E_{i_k})|,
\end{equation*}
where $|\cdot|$ denotes the norm in $\mathbb{R}^{{n}}$.
 We will also use very frequently \cite[Eq. (2.7)]{KS1}, which we state here for the reader's convenience. Let $f  \in C^\infty(M; \mathbb{R}^{{n}})$ be an immersion with second fundamental form $A$ and normal bundle $NM$. Then for each $G \in C^\infty(M;NM)$ and $X \in \mathcal{V}(M)$ one has
 {\begin{equation}\label{eq:DXundnablaX}
 D_X G = \nabla^\perp_X G - \sum_{i = 1}^2 \langle G, A(X,E_i) \rangle_{\R^n} D_{E_i} f,
 \end{equation}}
 where $\{ E_1, E_2 \}$ is an arbitrary orthonormal basis of $T_pM$ with respect to $g_f := f^*g_{\mathbb{R}^{{n}}}$.% and $ ( \cdot, \cdot ) $ is the standard inner product of $\mathbb{R}^3$.
 We also remark that we can define a tensorial version of $D$, treated as a tensor on $M$ with values in $\mathbb{R}^{{n}}$. The transformation law we prescribe here is analogous to \eqref{eq:tnsderiv}, namely if $F$ is a $(k,0)$ Tensor on $M$ with values in $\mathbb{R}^{{n}}$ we define for $X_1,...,X_{k+1} \in \mathcal{V}(M)$ 
 \begin{equation*}
 DF (X_1,...,X_{k+1}) := D_{X_1} F(X_2,...,X_{k+1}) -  \sum_{j = 2}^{k+1} F(X_2,...,\nabla_{X_1} X_j, ..., X_{k+1}). 
 \end{equation*}
 As an important special case we obtain for $f \in C^\infty(M,\mathbb{R}^{{n}})$ 
 \begin{equation*}
 D^2 f(X,Y) = D_XD_Yf - D_{\nabla_XY } f.  
 \end{equation*}
If $f$ is additionally an immersion, this formula yields exactly the second fundamental form (cf. \eqref{eq:Def2FF}). Hence one could also write $A[f] =D^2f$.  

%%%%%%%%%%%%%%%%%%%%%%%%%%%%%%%%%%%%%%%%%%%%%%%%%%%%%%%%%%%%%%%%%%%%%%%%
%%%%%%%%%%%%%%%%%%%%%%%%%%%%%%%%%%%%%%%%%%%%%%%%%%%%%%%%%%%%%%%%%%%%%%%%
\section{On the smooth convergence of surfaces}\label{app:smoothconv}
%%%%%%%%%%%%%%%%%%%%%%%%%%%%%%%%%%%%%%%%%%%%%%%%%%%%%%%%%%%%%%%%%%%%%%%%
%%%%%%%%%%%%%%%%%%%%%%%%%%%%%%%%%%%%%%%%%%%%%%%%%%%%%%%%%%%%%%%%%%%%%%%%

Here we present some useful results concerning smooth convergence on compact subsets of $\mathbb{R}^n$, which we will simply call smooth convergence. 

 We remark that smooth convergence, see Definition \ref{def:smocon}, actually takes place in the equivalence class of surfaces that coincide up to reparametrization, more precisely

\begin{remark}\label{rem:equivclass}
Consider a sequence of immersions $(f_j)_{j \in \mathbb{N}}$, $f_j : \Sigma \rightarrow \mathbb{R}^n$, that converges to $\widehat{f}$ smoothly on compact subsets of $\mathbb{R}^n$ and a sequence of diffeomorphisms $(\Psi_j)_{j \in \mathbb{N}}$, $\Psi_j : \Sigma_j \rightarrow \Sigma$ with $\Sigma_j$ smooth manifold without boundary. Then it follows from the definition of smooth convergence, that $f_j \circ \Psi_j$ converges to $\widehat{f}$ smoothly on compact subsets of $\mathbb{R}^n$. Moreover if $\Psi: \widetilde{\Sigma} \rightarrow \widehat{\Sigma}$ is yet another diffeomorphism then $f_j $ also converges to $f \circ \Psi$ smoothly on compact subsets of $\mathbb{R}^n$. 
\end{remark}

\begin{remark}\label{rem:toppres}
In general, smooth convergence is not \emph{topology-preserving}, i.e. the topologies of $\widehat{\Sigma}$ and $\Sigma$ need not coincide, cf. \cite[Fig. 6]{Breuning}. The situation is better if $\Sigma$ is connected and $\widehat{\Sigma}$ has a compact component $C$. %From 
\cite[Lemma 4.3]{KS2} gives that %follows then that 
$\Sigma,\widehat{\Sigma}$ are diffeomorphic. By the previous remark they can then also chosen to be equal. 
\end{remark}

Next we examine how relevant geometric quantities behave with respect to smooth convergence, for instance the diameter.

\begin{lemma} \label{lem:appdiam}
Suppose that $(f_j)_{j=1}^{\infty} : \Sigma \rightarrow \mathbb{R}^{{n}}$ is a sequence that converges smoothly on compact subsets of $\mathbb{R}^n$ to $\widehat{f}: \widehat{\Sigma}  \rightarrow \mathbb{R}^{{n}}$. Then 
\begin{equation*}
\diam \widehat{f}(\widehat{\Sigma}) \leq   \liminf_{j \rightarrow \infty} \diam f_j(\Sigma) .
\end{equation*}  
\end{lemma}

\begin{proof}
Suppose that $(\widehat{f}(p_k))_{k = 1}^\infty , (\widehat{f}(q_k))_{k =1}^\infty \subset \widehat{f}(\widehat{\Sigma})$ are sequences such that 
{$$|\hat{f}(p_k) - \hat{f}(q_k)| \rightarrow \diam \widehat{f}(\widehat{\Sigma}) \, .$$ 
Then, by Definition \ref{def:smocon} for each $k \in \mathbb{N}$ there exists $j(k) \in \mathbb{N}$ such that $p_k,q_k \in \widehat{\Sigma}(j)$  for all $j \geq j(k) $.
Now \eqref{eq:graph} implies that for all $j \geq j(k)$}
\begin{align*}
|\widehat{f}(p_k)- \widehat{f}(q_k) | & \leq |f_j \circ \phi_j (p_k) - f_j \circ \phi_j(q_k)| + |u_j(p_k) - u_j(q_k)| \\ & 
\leq \diam f_j(\Sigma) + 2 ||u_j||_{L^\infty(\hat{\Sigma}(j))} .
\end{align*}
Letting first $j \rightarrow \infty$ and then $k \rightarrow \infty$ we obtain the claim.
\end{proof}

Now we study the lower semicontinuity {with respect to smooth convergence of the Willmore energy. As a first step we prove the following result.}
\begin{lemma}\label{lem:B3}
{Let $(f_j)_{j \in \mathbb{N}}$, $f_j: \Sigma \rightarrow \mathbb{R}^{{n}}$ be} a sequence of immersions that converges smoothly on compact subsets of $\mathbb{R}^n$ to an immersion $\widehat{f}: \widehat{\Sigma}  \rightarrow \mathbb{R}^{{n}}$. % in the sense of Definition \ref{def:smocon}. Suppose that 
Let $(U, \psi)$ be a chart for $\widehat{\Sigma}$ such that $U \subset \widehat{\Sigma}(J)$ for some $J \in \mathbb{N}$ and $\widehat{g}_{i\tau}\circ \psi^{-1} \in C^1(\overline{\psi(U)})$  $\widehat{\Gamma}_{i\tau}^\alpha \circ \psi^{-1} \in C^0(\overline{\psi(U)})$, for all $i,\tau,\alpha$, and $\mathrm{det}(\widehat{g}), \widehat{g}_{11}$ are bounded from below by some positive $\delta > 0$ {where $\widehat{g}_{i\tau}$ and $\widehat{\Gamma}_{i\tau}^\alpha$ denote the metric and Christoffel's symbols induced by $\widehat{f}$ on $\widehat{\Sigma}$. {Moreover we require that $||D^2\widehat{f}||_{L^\infty(U, g_{\widehat{f}})}$, $||A[\widehat{f}]||_{L^\infty(U,g_{\widehat{f}})}, ||DA[\widehat{f}]||_{L^\infty(U,g_{\widehat{f}})} < \infty$.} } Let $(\phi_j)_{j = 1}^\infty$, $\phi_j:\widehat{\Sigma} \to \Sigma$, be a sequence of diffeomorphisms as in Definition \ref{def:smocon}. Let $\widehat{g}(m)$ be the first fundamental form induced by $f_m \circ \phi_m$ on $U$ with respect to the chart $(U,\psi)$ and $H(m) := H_{f_m \circ \phi_m} $ be the mean curvature of $f_m \circ \phi_m$. 

Then, $\widehat{g}(m) \circ \psi^{-1}$ converges to $\widehat{g} \circ \psi^{-1}$ uniformly in $\psi(U)$ and $H(m) \circ \psi^{-1}$ converges to $H_{\widehat{f}} \circ \psi^{-1}$ uniformly in $\psi(U)$. 
\end{lemma}
\begin{proof}

For $m > J$ let $u_m$ be as in Definition \ref{def:smocon} such that on $\widehat{\Sigma}(m)$ one has 
\begin{equation}\label{eq:exp}
f_m \circ \phi_m + u_m = \widehat{f} \; \mbox{ and } \; ||(\widehat{\nabla}^\perp)^k u_m||_{L^\infty(\widehat{\Sigma}(m))} \rightarrow 0,  \; m \rightarrow \infty \, .
\end{equation}
Let $(y^1 , y^2)$ be the local coordinates induced by $(U, \psi)$, in particular  for all $h \in C^\infty(\Sigma;\mathbb{R}^d)$, {$d\in \mathbb{N}$,} we denote $\frac{\partial h}{\partial y^i} = \frac{\partial (h \circ \psi^{-1})}{\partial e_i} \circ \psi$.  Our first intermediate claim is that $\frac{\partial u_m}{\partial y^i}$ and $\frac{\partial^2 u_m}{\partial y^i \partial y^\tau}$ converge to zero uniformly in $U$ for all $i,\tau$. 

In the following we let $E_1,E_2 \in \mathcal{V}(U)$ be the smooth orthonormal frame on $(U, g_{\widehat{f}})$ which we obtain by applying the Gram-Schmidt procedure on $\left\lbrace \frac{\partial}{\partial y^1}, \frac{\partial}{\partial y^2} \right\rbrace$, i.e. %. One readily checks that 
$$E_1 = \frac{1}{\sqrt{\widehat{g}_{1,1}}} \frac{\partial}{\partial y^1} \mbox{ and } E_2 = \frac{1}{\sqrt{\widehat{g}_{1,1}}\sqrt{det(\widehat{g})}} \big( \widehat{g}_{11} \frac{\partial}{\partial y^2} - \widehat{g}_{12}  \frac{\partial}{\partial y^1} \big)\, . $$
Note that by \eqref{eq:DXundnablaX}  
\begin{equation*}
\frac{\partial u_m}{\partial y^i} = D_{\frac{\partial}{\partial y^i} } u_m = \widehat{\nabla}^\perp_{\frac{\partial}{\partial y^i} } u_m - \sum_{j = 1}^2 \langle  u_m , A[\widehat{f}]( \frac{\partial}{\partial y^i}, E_j) \rangle_{\R^3} D_{E_j} \widehat{f}  
\end{equation*}
and hence on $U$ we have 
\begin{equation}\label{eq:zehgestossenautsch}
\left\vert \frac{\partial u_m}{\partial y^i} \right\vert \leq ||\widehat{\nabla}^\perp u_m ||_{L^\infty(U)} + 2 ||A[\widehat{f}]||_{L^\infty(U)} |\widehat{g}_{i,i}|^\frac{1}{2} ||u_m||_{L^\infty(U)}.
\end{equation}
Estimating $||\widehat{\nabla}^\perp u_m||_{L^\infty(U)}  \leq ||\widehat{\nabla}^\perp u_m||_{L^\infty(\widehat{\Sigma}(m))} \rightarrow 0$, $|| u_m||_{L^\infty(U)}  \leq || u_m||_{L^\infty(\widehat{\Sigma}(m))} \rightarrow 0$ and $|\widehat{g}_{i,i}| \leq ||\widehat{g}_{i,i} \circ \psi^{-1}||_{L^\infty(\psi(U))} $  we infer that $\frac{\partial u_m}{\partial y^i }$ converges to  zero uniformly on $U$. Next we compute for all $i, \tau$ writing for short $A = A[\widehat{f}]$  
\begin{align*}
\frac{\partial^2 u_m}{\partial y^\tau \partial y^i} & = D_{\frac{\partial}{\partial y^\tau}} D_{\frac{\partial}{\partial y^i}} u_m =  D_{\frac{\partial}{\partial y^\tau}} \left( \widehat{\nabla}^\perp_{\frac{\partial}{\partial y^i}} u_m  - \sum_{j = 1}^2 \langle u_m , A ( \frac{\partial}{\partial y^i}, E_j) \rangle_{\R^{{n}}} D_{E_j} f\right) 
\\ & =  D_{\frac{\partial}{\partial y^\tau}}  \widehat{\nabla}^\perp_{\frac{\partial}{\partial y^i}} u_m - \sum_{j = 1}^2 D_{\frac{\partial}{\partial y^\tau}} \left[ \langle u_m , A ( \frac{\partial}{\partial y^i}, E_j) \rangle_{\R^{{n}}} D_{E_j} f\right]
\\ & = \widehat{\nabla}^\perp_{\frac{\partial}{\partial y^\tau}} \widehat{\nabla}^\perp_{\frac{\partial}{\partial y^i}} u_m - \sum_{l = 1}^2 \langle \widehat{\nabla}^\perp_{\frac{\partial}{\partial y^i}} u_m , A ( \frac{\partial}{\partial y^\tau} , E_l ) \rangle_{\R^{{n}}}  D_{E_l} f
\\ & \quad - \sum_{j = 1}^2 \langle \frac{\partial u_m}{\partial y^\tau}, A( \frac{\partial}{\partial y^i}, E_j ) \rangle_{\R^{{n}}} D_{E_j} \widehat{f} - \sum_{j =1}^2 \langle u_m , D_{\frac{\partial}{\partial y^\tau}} A( \frac{\partial}{\partial y^i}, E_j) \rangle_{\R^{{n}}} D_{E_j} f
\\ & \quad  - \sum_{ j = 1}^2 \langle u_m, A(\frac{\partial}{\partial y^i}, E_j) \rangle_{\R^{{n}}} D_{\frac{\partial}{\partial y^\tau}} D_{E_j} \widehat{f} 
\\ & = (\widehat{\nabla}^\perp)^2 u_m (\frac{\partial}{\partial y^\tau}, \frac{\partial}{\partial y^i} )  {+ \widehat{\nabla}^\perp u_m ( \widehat{\nabla}_{\frac{\partial}{\partial y^\tau}} \frac{\partial}{\partial y^i} ) }
\\ &  \quad -\sum_{l = 1}^2 \langle \widehat{\nabla}^\perp u_m ( \frac{\partial}{\partial y^i}) , A ( \frac{\partial}{\partial y^\tau} , E_l ) \rangle_{\R^{{n}}} D_{E_l} f - \sum_{j = 1}^2 \langle \frac{\partial u_m}{\partial y^\tau}, A( \frac{\partial}{\partial y^i}, E_j ) \rangle_{\R^{{n}}} D_{E_j} \widehat{f} 
\\ & \quad -  \sum_{ j = 1}^2 \langle u_m, DA( \frac{\partial}{\partial y^\tau} , \frac{\partial}{\partial y^i}, E_j) + A ( \widehat{\nabla}_{\frac{\partial}{\partial y^\tau} } \frac{\partial}{\partial y^i}, E_j)  + A( \frac{\partial}{\partial y^i}, \widehat{\nabla}_{\frac{\partial}{\partial y^\tau} }  E_j ) \rangle_{\R^{{n}}} D_{E_j} \widehat{f} 
\\ & \quad - \sum_{j = 1}^2 \langle u_m , A( \frac{\partial}{\partial y^i}, E_j) \rangle_{\R^{{n}}} \left[ D^2\widehat{f}( \frac{\partial}{\partial y^\tau}, E_j ) + D\widehat{f}( \widehat{\nabla}_{\frac{\partial}{\partial y^\tau}} E_j ) \right].
\end{align*}
All terms that appear here as arguments of tensors can be bounded in $L^\infty$-norm with quantities that we assumed to be bounded. Notice that a bound on %bounding 
$\widehat{\nabla}_{\frac{\partial}{\partial y^\tau} }  \frac{\partial}{\partial y^i}$  needs the fact that the Christoffel symbols lie in $C^0(\overline{\psi(U)})$. Bounding $\widehat{\nabla}_{\frac{\partial}{\partial y^\tau}} E_j$ in terms of the given quantities needs the explicit representation of $E_j$ that we discussed above. Here we also need that $\mathrm{det}(\widehat{g}) , \widehat{g}_{11}$ are bounded from below uniformly in $U$. We obtain with a straightforward computation that $\frac{\partial^2 u_m}{\partial y^\tau \partial y^i}$ converges to zero uniformly in $U$. 

We now show that $\widehat{g}(m)$ converges to $\widehat{g} $ uniformly on $U$ which implies the convergence claimed in the statement. First note that by \eqref{eq:exp} and \eqref{eq:zehgestossenautsch} 
$$\frac{\partial (f_m \circ \phi_m)}{\partial y^\tau} = \frac{\partial \widehat{f}}{\partial y^\tau} + \mathit{o}(1),$$
where $\frac{\partial \widehat{f}}{\partial y^\tau}$ are bounded by assumption. Hence, $\frac{\partial (f_m \circ \phi_m)}{\partial y^\tau}$ and $\hat{g}(m)$ are uniformly bounded. 
Now we can compute using \eqref{eq:exp}  
\begin{align*}
\widehat{g}_{i\tau} & = \langle \frac{\partial \widehat{f}}{\partial y^i}, \frac{\partial \widehat{f}}{\partial y^\tau} \rangle_{\R^{{n}}} \\
& = \widehat{g}_{i\tau}(m) + \langle \frac{\partial (f_m \circ \phi_m)}{\partial y^i}, \frac{\partial u_m}{\partial y^\tau} \rangle_{\R^{{n}}} + \langle \frac{\partial (f_m \circ \phi_m)}{\partial y^\tau}, \frac{\partial u_m}{\partial y^i} \rangle_{\R^{{n}}} + \langle \frac{\partial u_m}{\partial y^\tau}, \frac{\partial u_m}{\partial y^i} \rangle_{\R^n}.
\end{align*}
By the arguments above, the last three terms are uniformly convergent to zero and so convergence of the first fundamental form is shown.  Note in particular that also $\widehat{g}^{-1}(m)$ converges to $\widehat{g}^{-1}$ since we assumed that $\mathrm{det}(\widehat{g})$ is strictly bounded from below.\\
Observe now that  by \eqref{eq:Ainccoo} and \eqref{eq:Hvec}
\begin{equation*}
\vec{H}_{\widehat{f}} = \widehat{g}^{i\tau} \big( \frac{\partial^2 \widehat{f}}{\partial y^i \partial y^\tau} - \widehat{\Gamma}_{i\tau}^\alpha \frac{\partial \widehat{f}}{\partial y^\alpha}  \big)  
\end{equation*}
and 
\begin{equation*}
\vec{H}(m) = \widehat{g}^{i\tau}(m) \big( \frac{\partial^2 (f_m \circ \phi_m)}{\partial y^i \partial y^\tau} - \widehat{\Gamma}_{i\tau}^\alpha(m) \frac{\partial (f_m \circ \phi_m)}{\partial y^\alpha} \big),
\end{equation*}
where $\widehat{\Gamma}_{i\tau}^\alpha(m)$ denotes the Christoffel symbols of the immersion $f_m \circ \phi_m$ with respect to the chart $(U,\psi)$. 
We have already discussed the uniform convgence of all terms that $H(m)$ consists of except for the Christoffel symbols. The convergence of those however follows analogously to the convergence of $\widehat{g}(m)$ from the classical formula 
\begin{equation*}
\widehat{\Gamma}_{i\tau}^\alpha(m) =  g^{\alpha \beta}(m) \langle \frac{\partial^2 (f_m \circ \phi_m) }{\partial y^i \partial y^\tau}, \frac{\partial (f_m \circ \phi_m)}{\partial y^\beta }\rangle_{\R^{{n}}}. \qedhere  
\end{equation*}
\end{proof}

\begin{lemma}\label{lem:appsemi}
Suppose that $(f_j)_{j=1}^{\infty} : \Sigma \rightarrow \mathbb{R}^{{n}}$ is a sequence of immersions that converges smoothly on compact subsets of $\mathbb{R}^n$ to an immersion $\widehat{f}: \widehat{\Sigma}  \rightarrow \mathbb{R}^{{n}}$. Then 
\begin{equation*}
\mathcal{W}(\widehat{f}) \leq   \liminf_{j \rightarrow \infty} \mathcal{W}(f_j) .
\end{equation*}
Additionally, if $\widehat{\Sigma}$ is compact then $\mathcal{W}(\widehat{f})= \lim_{j \rightarrow \infty} \mathcal{W}(f_j)$. 
\end{lemma}
\begin{proof}
We start choosing a cover $\{(U_p,\psi_p)\}_{p \in \widehat{\Sigma}}$ of $\widehat{\Sigma}$ such that $U_p$ is an open neighborhood of $p$. Since each $p$ is contained in some $\Sigma(m_p)$ for some $m_p \in \mathbb{N}$ and $\Sigma(m_p)$ is open, we may assume that $U_p \subset \Sigma(m_p)$ by possibly shrinking $U_p$. Let 
$V_p$ be a neighborhood of of $p$ compactly contained in $U_p$. Then in each chart $(V_p, \psi_p)$, $\widehat{g}_{it}$ and $ \Gamma_{it}^\alpha $ are bounded and $\mathrm{det}(\widehat{g})$ is uniformly bounded from below by some $\delta = \delta(p) > 0$. By second countability there exist countably many points $\{ p_\nu \}_{\nu = 1}^\infty$ such that $\{(V_{p_\nu} , \psi_{p_\nu})\}_{\nu = 1}^\infty$ is a cover of $\widehat{\Sigma}$ and there exists a locally finite partition of unity $(\eta_\nu)_{\nu = 1}^\infty$ of smooth an compactly supported functions that satisfy $\mathrm{supp}(\eta_\nu) \subset V_{p_\nu} $. Now we infer by Lemma \ref{lem:B3} {(taking diffeomorphisms $\phi_m$ as in \eqref{eq:exp})} and Fatou's Lemma
\begin{align*}
\int_{\widehat{\Sigma}} H_{\widehat{f}}^2 d\mu_{\widehat{f}} & = \sum_{\nu = 1}^\infty \int_{ V_{p_\nu}}\eta_\nu H_{\widehat{f}}^2  d\mu_{\widehat{f}} \\
&= \sum_{\nu = 1}^\infty \int_{\psi_{p_{\nu}}(V_{p_\nu})} (\eta_{\nu} \circ \psi_{p_\nu}^{-1}) (H_{\widehat{f}} \circ \psi_{p_\nu}^{-1})^2 \sqrt{\det{\widehat{g}}} \circ \psi_{p_\nu}^{-1} \dx\\
& = \sum_{\nu = 1}^\infty \lim_{m \rightarrow \infty} \int_{\psi_{p_{\nu}}(V_{p_\nu})} (\eta_{\nu} \circ \psi_{p_\nu}^{-1}) (H_{f_m \circ \phi_m} \circ \psi_{p_\nu}^{-1})^2 \sqrt{\det{\widehat{g}(m)}} \circ \psi_{p_\nu}^{-1} \dx
\\ & \leq \liminf_{m \rightarrow \infty} \sum_{\nu = 1}^\infty \int_{\psi_{p_{\nu}}(V_{p_\nu})} (\eta_{\nu} \circ \psi_{p_\nu}^{-1}) (H_{f_m \circ \phi_m} \circ \psi_{p_\nu}^{-1})^2 \sqrt{\det{\widehat{g}(m)}} \circ \psi_{p_\nu}^{-1} \dx \\ & 
= \liminf_{m \rightarrow \infty} \sum_{ \nu = 1}^\infty \int_{V_{p_\nu}} \eta_\nu H_{f_m \circ \phi_m}^2 d\mu_{f_m \circ \phi_m} = \liminf_{m \rightarrow \infty} \int_{\widehat{\Sigma}} H_{f_m \circ \phi_m}^2 d\mu_{f_m \circ \phi_m} 
\end{align*} 
All in all we obtain $\mathcal{W}(\widehat{f}) \leq \liminf_{m\rightarrow \infty} \mathcal{W}(f_m \circ \phi_m) = \liminf_{m \rightarrow \infty} W(f_m)$ as the Willmore energy does not depend on the reparametrization. %It remains to show the last sentence of the statement. 
If $\widehat{\Sigma}$ is compact then the partition of unity can be chosen to be finite and the last claim follows then with the same techniques.
\end{proof}

\begin{lemma}\cite[Co. 1.4]{Breuning} \label{lem:breu}
Suppose that $f_j : \Sigma \rightarrow \mathbb{R}^n$ and $\widehat{f}: \widehat{\Sigma} \rightarrow \mathbb{R}^n$ are such that $f_j$ converges to $\widehat{f}$ smoothly on compact subsets of $\mathbb{R}^n$. Then the surface measures $f_j^*\mu_{g_j}$ converge in $C_0(\mathbb{R}^n)'$ to $ \hat{f}^*\mu_{\hat{f}}$. 
\end{lemma}

A second concept of convergence that is related to smooth convergence is the $C^l$-convergence which we also use throughout the article. 

\begin{definition}\label{def:Clconv}
 We say that a sequence of immersions $(f_j)_{j \in \mathbb{N}}$, $f_j : \Sigma \rightarrow \mathbb{R}^{{n}}$ defined on a two dimensional manifold ${\Sigma}$ without boundary converges to $\widehat{f}: {\Sigma} \rightarrow \mathbb{R}^{{n}}$ in $C^l(\Sigma)$, $l \in \mathbb{N}$ if there exist diffeomorphisms ${\phi}_j : {\Sigma} \rightarrow {\Sigma}$  for all $j \in \mathbb{N}$ and $u_j: {\Sigma} \rightarrow N \Sigma$ such that 
$f_j \circ \phi_j  + u_j = \widehat{f}$ on ${\Sigma}$ and $||(\widehat{\nabla}^\perp)^k u_j||_{L^\infty({\Sigma})} \rightarrow 0 $ as $j \rightarrow \infty$ for all $k \in \{ 0, ..., l\}$.
\end{definition}
 
\begin{remark}\label{ref:remcomp}
The two concepts of convergence we discussed are obviously related. Indeed, if $f_j :  \widetilde{\Sigma} \rightarrow \mathbb{R}^n$ is a sequence that converges smoothly on compact subsets to some $\widehat{f} : \widetilde{\Sigma} \rightarrow \mathbb{R}^n$ and $\widetilde{\Sigma}$ is compact, then $f_j$ converges to $\widehat{f}$ in $C^l$ for all $l \in \mathbb{N}$. We further say that a family $(f(t))_{t \in [0,\infty)}$ converges to $\widehat{f}$ in $C^l$ for all $l$ if for each sequence $t_j \rightarrow \infty$ one has $f(t_j) \rightarrow \widehat{f}$ as $j \rightarrow \infty$.
\end{remark}

We will now present an alternative characterization of $C^l$ convergence in which we do not need to require that $u_j$ are orthogonal. However we have to pay a price -- in this case one needs control the full derivative. {Even though we expect this result to be true even in higher codimension, we formulate it only in the case of $n = 3$ for the sake of simplicity. This will be sufficient for our purposes.} 

\begin{prop}\label{prop:propsmoconalt}
Let ${\Sigma}$ be a compact orientable two-dimensional manifold without boundary and $f_j : {\Sigma} \rightarrow \mathbb{R}^3$ be a sequence of immersions and $k \geq 2$. Then $f_j$ converges to a limit immersion $\widehat{f}: {\Sigma} \rightarrow \mathbb{R}^3$ in $C^k$ if and only if there exist $w_j \in C^k( {\Sigma}, \mathbb{R}^3)$ and $C^k$-smooth diffeomorphisms $\psi_j: {\Sigma} \rightarrow {\Sigma}$ such that for $j$ large enough
\begin{equation*}
f_j \circ \psi_j = \widehat{f} + w_j \quad \; \textrm{on} \; {\Sigma}
\end{equation*}
and  for all $k \in \mathbb{N}$ one has $||D^k w_j||_{L^\infty({\Sigma}, g_{\widehat{f}})} \rightarrow 0$ as $j \rightarrow \infty$. 
\end{prop}
\begin{proof}
First assume that $f_j: {\Sigma} \rightarrow \mathbb{R}^3$ converges to $\widehat{f}: {\Sigma} \rightarrow \mathbb{R}^3$ in $C^k({\Sigma})$ . Then, for $j$ large enough one can find $u_j \in C^k( {\Sigma}, N {\Sigma}) $ and $C^k$-diffeomorphisms $\phi_j: {\Sigma} \rightarrow {\Sigma}$ such that 
\begin{equation*}
f_j \circ \phi_j = \widehat{f} + u_j \quad \; \textrm{on} \; {\Sigma}
\end{equation*}
and for all $k \in \mathbb{N}$ one has $||(\widehat{\nabla}^\perp)^k u_j||_{L^\infty} \rightarrow 0$ as $j \rightarrow \infty$. Now we choose $\psi_j := \phi_j$ and $w_j := u_j$. It only remains to show that 
$||D^k w_j ||_{L^\infty} \rightarrow 0$ as $k \rightarrow \infty$. For $k = 1$ we observe that for each $X \in \mathcal{V}(M)$ one has {by \eqref{eq:DXundnablaX}
\begin{align}\label{eq:DXWJ}
D_X w_j & = \widehat{\nabla}^\perp_X w_j - \sum_{i = 1}^2 \langle w_j, A[\widehat{f}](X, E_i) \rangle_{\R^3} D_{E_i} f 
\\ & =  \widehat{\nabla}^\perp_X u_j - \sum_{i=1}^2   \langle u_j, A[\widehat{f}](X, E_i) \rangle_{\R^3} D_{E_i} f \, . \nonumber
\end{align}}
We obtain that 
\begin{equation*}
||Dw_j||_{L^\infty} \leq ||\widehat{\nabla}^\perp u_j||_{L^\infty} + C || u_j||_{L^\infty} ||A[\widehat{f}]||_{L^\infty} || D\widehat{f}||_{L^\infty}.
\end{equation*}
Since ${\Sigma}$ is compact, {%we obtain that 
$||A[\widehat{f}]||_{L^\infty}$ and $||D\widehat{f}||_{L^\infty}$ are finite  and thus %one obtains that 
$||Dw_j ||_{L^\infty} \rightarrow 0$ as $j \rightarrow \infty$. The estimates for $k\geq 2$ %rest of this claim ($k \geq 2$) 
follow} easily by using iterated versions of \eqref{eq:DXWJ}. 

For the converse, suppose we have diffeomorphisms $\psi_j$ and $w_j$ as in the statement.  We denote by $C^k({\Sigma};\mathbb{R})$ the set of all $C^k$-smooth real valued maps from ${\Sigma}$ of $\mathbb{R}$ equipped with the norm $||f||_{C^k({\Sigma};\mathbb{R})} := \sum_{l = 1}^k ||\widehat{\nabla}^l f||_{L^\infty}$, where $\widehat{\nabla}$ here denotes the tensorial connection with respect to the Levi-Civita connection on $({\Sigma}, g_{\widehat{f}})$, cf. \cite[Lemma 4.6]{Lee}. We also endow $C^k( {\Sigma}; \mathbb{R}^3)$ with the norm $||f||_{C^k( {\Sigma} ; \mathbb{R}^3 ) } = \sum_{l=1}^k ||D^l f||_{L^\infty}$. Moreover we define $\mathrm{Diffeo}^k( {\Sigma}, {\Sigma}) $ to be the set of all $C^k$ smooth diffeomorphisms of ${\Sigma}$. Note  that $\mathrm{Diffeo}^k( {\Sigma} , {\Sigma})$ is a smooth Banach manifold with the compact-open topology and {for all $\phi \in \mathrm{Diffeo}^k( {\Sigma})$ the tangent space $T_{\phi}  \mathrm{Diffeo}^k( {\Sigma} , {\Sigma})$ can be identified with $\mathcal{V}({\Sigma})$.} This fact follows from \cite{Wittmann} and \cite[Chap. 2 Thm 1.7]{Hirsch}.
Let now $N_{\widehat{f}}$ be a smooth unit normal field along $\widehat{f}$. (Here orientability of ${\Sigma}$ is needed). 
 We now define for all $k \in \mathbb{N}$ the map 
{\begin{equation}\label{eq:Anto}
F : \mathrm{Diffeo}^k({\Sigma}, {\Sigma}) \times C^k({\Sigma};\mathbb{R}) \rightarrow C^k( {\Sigma};\mathbb{R}^3), \quad  F(\eta, \beta) := (\widehat{f} + \beta N_{\widehat{f}} ) \circ \eta \, .
\end{equation}}
%to be given by $F(\eta, \beta) := (\widehat{f} + \beta N_{\widehat{f}} ) \circ \eta$.  
 It is easy to show that for all $X \in \mathcal{V}({\Sigma})$ and $\alpha \in C^k({\Sigma};\mathbb{R})$ one has 
 $d_{(id,0)} F (X,\alpha) = D_X \widehat{f} + \alpha N_{\widehat{f}}$. Having this formula, one checks that $d_{(id,0)} F : T_{(id,0) } (\mathrm{Diffeo}^k({\Sigma}, {\Sigma}) \times C^k({\Sigma};\mathbb{R})) \rightarrow T_{\widehat{f}}(C^k({\Sigma};\mathbb{R}^3)) \simeq C^k({\Sigma};\mathbb{R}^3)$ is an isomorphism. As a consequence one can find a small neighborhood $V$ of $(\mathrm{id},0)$ such that $F\vert_V$ is a diffeomorphism. We conclude that for all $k \in \mathbb{N}$ there exists $\epsilon > 0$ such that $||g- \widehat{f}||_{C^k({\Sigma};\mathbb{R}^3)}< \epsilon$ implies that there exists $\eta \in \mathrm{Diffeo}^k$ and $\beta \in C^k$ such that $g = (\widehat{f} + \beta N_{\widehat{f}})\circ \eta$. Next we look at $g = \widehat{f}+ w_j$. For $j$ large enough one has that there exists $\eta_j \in \mathrm{Diffeo}^k$ and $\beta_j \in C^k$ such that 
\begin{equation*}
\widehat{f} + w_j = (\widehat{f}+ \beta_j N_{\widehat{f}}) \circ \eta_j  
\end{equation*}
and thus we infer that 
\begin{equation*}
f_j \circ \phi_j = ( \widehat{f} + \beta_j N_{\widehat{f}}) \circ \eta_j .
\end{equation*}
We compose with $\eta_j^{-1}$ to obtain 
\begin{equation*}
f_j \circ \phi_j \circ \eta_j^{-1} = \widehat{f} + \beta_j N_{\widehat{f}} 
\end{equation*}
Defining $\psi_j := \phi_j \circ \eta_j^{-1}$ and $u_j := \beta_j N_{\widehat{f}}$ we obtain that $f_j \circ \psi_j = \widehat{f} + u_j$ and $u_j \in C^k({\Sigma}, N {\Sigma}) $. It remains to show that $||(\widehat{\nabla}^\perp)^l u_j|| \rightarrow 0$ for all $l = 1,...,k$. To do so we compute for any { $X \in \mathcal{V}({\Sigma})$}
\begin{equation*}
\widehat{\nabla}^\perp_X u_j = \widehat{\nabla}^\perp_X ( \beta_j N_f) = X(\beta_j) N_{\widehat{f}} + \beta_j \widehat{\nabla}^\perp_X  N_{\widehat{f}} 
\end{equation*}  
Note that $X(\beta_j) = \widehat{\nabla}_X \beta_j$ and thus 
\begin{equation*}
||\widehat{\nabla}^\perp u_j||_{L^\infty} \leq ||\beta_j||_{C^1(\Sigma,\mathbb{R})} ( 1 + ||\widehat{\nabla}^\perp N_{\widehat{f}}||_{L^\infty} ) .
\end{equation*}
Observe that $||\widehat{\nabla}^\perp N_{\widehat{f}}||_{L^\infty}$ is finite by compactness of ${\Sigma}$. 
Similarly one can show that 
\begin{equation}\label{eq:finalle}
||(\widehat{\nabla}^\perp)^j u_j ||_{L^\infty} \leq C(k,{\Sigma}, \widehat{f}) ||\beta_j||_{C^j({\Sigma},\mathbb{R}) }  \quad \forall j = 1,...,k.
\end{equation}
Note that $\widehat{f} + w_j \rightarrow \widehat{f}$ in $C^k({\Sigma};\mathbb{R}^3)$ and the fact that $F$, {defined in \eqref{eq:Anto},} is a local diffeomorphism implies that $(\eta_j, \beta_j) $ converges to $(\mathrm{id},0)$ in $\mathrm{Diffeo}^k({\Sigma}, {\Sigma}) \times C^k( {\Sigma}) $. Thus $\beta_j$ converges to $0$ in $C^k(\widehat{\Sigma})$.  This and \eqref{eq:finalle} verify Definition \ref{def:Clconv} for $l = k$. The claim is shown.  
\end{proof}

%\begin{remark}
%The prerequisite that ${\Sigma}$ is orientable is somewhat unnatural but needed for the construction of the normal field that we chose. Luckily in our case we always have ${\Sigma}= \mathbb{S}^1 \times \mathbb{S}^1$ which is orientable. If one wants to obtain the result for non-orientable manifolds ${\Sigma}$ one needs to localize very carefully, because the $\phi_j$ needs to be diffeomorphisms, which is a global statement.
%\end{remark}

Also $C^l$-convergence is not affected by reparametrizations and Remark \ref{rem:equivclass} can be formulated also for the $C^l$-convergence. This implies in particular that limits with respect to $C^l$-convergence are not unique. In the rest of this section we will however show that, in our setting, $C^l$-limits are unique up to reparametrizations. Let us first fix what we mean by classical $C^l$ convergence. 
\begin{definition}\label{def:classi}
 We say that a sequence of immersions $(h_j)_{j = 1}^\infty, h_j : \Sigma \rightarrow \mathbb{R}^{{n}}$ converges classically in $C^l$ to some immersion $h: \Sigma \rightarrow \mathbb{R}^{{n}}$ if $u_j := h- h_j : \Sigma \rightarrow \mathbb{R}^{{n}}$ satisfies $||D^k u_j||_{L^\infty(\Sigma)} \rightarrow 0 $ for all $k = 0,...,l$.
 \end{definition}

\begin{prop}\label{prop:fastunique}
Let $(f_j)_{j = 1}^\infty: \mathbb{S}^1 \times \mathbb{S}^1 \rightarrow \mathbb{R}^3$ be a sequence of smooth immersions and $l \geq 2$. {Let $f,h: \mathbb{S}^1 \times \mathbb{S}^1 \rightarrow \mathbb{R}^3$ be such that $f_j$ converges to $f$ in $C^l$ and and $f_j$ converges to $h$ classically in $C^l$. Then $f$ and $h$} coincide up to reparametrization, i.e. there exists a $C^l$-diffeomorphism $\phi :\mathbb{S}^1 \times \mathbb{S}^1 \rightarrow \mathbb{S}^1 \times \mathbb{S}^1$ such that $h = f \circ \phi$. 
\end{prop}
\begin{proof} 
Since $f_j$ converges to $f$ in $C^l$ there exist diffeomorphisms $\phi_j$ of $\mathbb{S}^1 \times \mathbb{S}^1$ and maps $u_j: \mathbb{S}^1 \times \mathbb{S}^1 \rightarrow \mathbb{R}^3$ such that
\begin{equation}\label{eq:B7}
f_j \circ \phi_j + u_j = f \quad \textrm{on} \; \mathbb{S}^1 \times \mathbb{S}^1,
\end{equation}
{and $||u_j||_{L^\infty}, ||D u_j||_{L^\infty}$ converge to zero.
 Moreover there exist $v_j$ such that 
 \begin{equation}\label{eq:B8}
 f_j  + v_j = h \quad \textrm{on} \; \mathbb{S}^1 \times \mathbb{S}^1,
 \end{equation}
and $||v_j||_{L^\infty}, ||D v_j||_{L^\infty}$ converge to zero.\\
{\textbf{Step 1:} $(\phi_j)_{j= 1}^\infty$ converges uniformly to some {$\phi \in C^0(\mathbb{S}^1 \times \mathbb{S}^1)$} that satisfies $h = f \circ \phi$.}\\
First note that functions on $\mathbb{S}^1 \times \mathbb{S}^1$ can be periodically extended on $\mathbb{R}^2$. Doing so and tacitly identifying all the functions we defined above with their unique periodic extensions we infer that \eqref{eq:B7} and \eqref{eq:B8} hold on the whole of $\mathbb{R}^2$. From both equations we infer that 
 \begin{equation}\label{eq:B8a}
 h \circ \phi_j  - v_j \circ \phi_j + u_j = f \quad \textrm{on} \; \mathbb{R}^2.
 \end{equation}
 Since we deal now with functions in $C^1(\mathbb{R}^2; \mathbb{R}^3)$ we can compute derivatives simply using the Jacobi matrix. By the chain rule
 \begin{equation}\label{eq:B9}
 (Dh(\phi_j)- Dv_j(\phi_j)) D\phi_j + Du_j = Df \quad \mbox{ in } \mathbb{R}^2.
 \end{equation}
We claim that $||D\phi_j||_{L^\infty(\mathbb{R}^{2\times 2})}$ is bounded. For this assume that a subsequence (which we do not relabel) satisfies $||D\phi_j||_{L^\infty} \rightarrow \infty$ and let $p_j \in \mathbb{S}^1 \times \mathbb{S}^1$ be such that $|D\phi_j(p_j)| = ||D\phi_j||_{L^\infty}$, where $|\cdot|$ is a suitable matrix norm. Evaluating \eqref{eq:B9} at $p_j$ and dividing by $||D\phi_j||_{L^\infty}$ we obtain
{ \begin{equation}\label{eq:B11}
 \big( Dh ( \phi_j(p_j)) -  Dv_j(\phi_j(p_j)) \big) \frac{D\phi_j(p_j)}{||D\phi_j||_{L^\infty}} + \frac{1}{||D\phi_j||_{L^\infty}}Du_j(p_j) = \frac{1}{||D\phi_j||_{L^\infty}} Df(p_j) \, .
 \end{equation} 
By} boundedness of $\phi_j: \mathbb{R}^2 \rightarrow \mathbb{S}^1 \times \mathbb{S}^1$ and the choice of $p_j$ one can choose a subsequence such that $(\phi_j(p_j))_{j = 1}^\infty$ converges to some $q \in \mathbb{S}^1 \times \mathbb{S}^1$  and $\frac{D\phi_j(p_j)}{||D\phi_j||_{L^\infty}}$ converges to some $B \in \mathbb{R}^{2\times 2}$ that satisfies $|B| = 1$. Note that by the requirements on $u_j$, $v_j$ and the fact that the first fundamental forms of $f,h$ w.r.t the local coordinates $(u,v)$ are bounded one has
$||Du_j||_{L^\infty(\mathbb{R}^2,\mathbb{R}^{2\times 3})},||Dv_j||_{L^\infty(\mathbb{R}^2,\mathbb{R}^{2\times 3})} \rightarrow 0$ as $j \rightarrow \infty$. Passing to the limit in \eqref{eq:B11} we obtain
  \begin{equation*}
  Dh(q) B = 0 \, . 
\end{equation*}
This is a contradiction to $h$ being an immersion and $|B| = 1$. Hence $||D \phi_j||_{L^\infty(\mathbb{R}^2,\mathbb{R}^{2,2})}$ is bounded. Note also that {$\phi_j: \mathbb{R}^2 \rightarrow \mathbb{S}^1 \times \mathbb{S}^1$} is uniformly bounded as it takes values only in {$\mathbb{S}^1 \times \mathbb{S}^1$}. By Arzela-Ascoli's theorem there exists a subsequence (which we do not relabel) and $\phi \in C^0(\mathbb{S}^1 \times \mathbb{S}^1)$ such that $\phi_j \rightarrow \phi$ on $\mathbb{S}^1 \times \mathbb{S}^1$.   
We can now go back to \eqref{eq:B8a} and pass to the limit there to obtain 
{\begin{equation}\label{eq:B13}
h \circ \phi = f \quad \textrm{on} \; \mathbb{S}^1 \times \mathbb{S}^1.
\end{equation}}

  {\textbf{Step 2:}} $\phi$ is a local $C^l$ diffeomorphism, i.e. $\phi$ is $C^l$ smooth and for all $p \in \mathbb{S}^1 \times \mathbb{S}^1$ there exists an open neighborhood $U$ containing $p$ such that $\phi_{\mid_U}$ is a diffeomorphism onto its image.\\
To this end fix $p \in \mathbb{S}^1 \times \mathbb{S}^1$ and recall that, being $h$ an immersion, there exists an open neighborhood $W$ of $\phi(p)$ such that $h_{\mid_W}$ is a diffeomorphism onto its image $V:= h(W)$. We denote by $\widetilde{h} :V \rightarrow W$ the inverse of $h_{\mid_W}$.  By \eqref{eq:B13} we obtain 
\begin{equation}\label{eq:B14}
\phi = \widetilde{h} \circ f \quad  \textrm{on} \; f^{-1}(V).
\end{equation}
Notice that since $\phi(p) \in W$ it follows that $f(p) =h(\phi(p))\in V$ and hence $p \in f^{-1}(V)$ so that $f^{-1}(V)$ is an open neighborhood of $p$. Now there exists another open neighborhood $G$ of $p$  such that $f_{\mid_G}$ is a $C^l$-diffeomorphism onto its image. Defining $U = G \cap f^{-1}(V)$ we obtain that $\phi_{\mid_U}$ is a $C^l$-diffeomorphism as a composition of two diffeomorphisms. Note in particular that $D\phi(p)$ is invertible at each point {$p \in \mathbb{S}^1 \times \mathbb{S}^1$}. This implies in particular, as $\mathbb{S}^1 \times \mathbb{S}^1$ is connected and $\phi \in C^l$ that 
$ \mathrm{sgn}(\mathrm{det}(D\phi))$ is constant. 

{\textbf{Step 3:} $\mathrm{deg}(\phi) = \pm 1$.}\\ 
Recall that the \emph{mapping degree} of $\phi$ is given by 
\begin{equation}\label{eq:degree}
\mathrm{deg}(\phi) := \sum_{x \in \phi^{-1}(\{y\}) } \mathrm{sgn} \left(  \mathrm{det}(D\phi(x)) \right) 
\end{equation}
for any choice of $y \in \mathbb{S}^1 \times \mathbb{S}^1$. See \cite[Chap. 3]{Outerelo} or \cite[Chap. 3,§3]{Pollack} for the well-definedness of $\mathrm{deg}$, e.g. the independence of the definition of the chosen $y$ and finiteness of the sum in the definition. We make use of the degree-integration formula (cf \cite[p.188]{Pollack} to compute $\mathrm{deg}(\phi)$. Since {$\phi:\mathbb{S}^1 \times \mathbb{S}^1 \to \mathbb{S}^1 \times \mathbb{S}^1 $} is sufficiently smooth one has for all differential forms $\omega$ on $\mathbb{S}^1 \times \mathbb{S}^1$ that 
\begin{equation*}
\int_{\mathbb{S}^1 \times \mathbb{S}^1} \phi^* \omega = \mathrm{deg}(\phi) \int_{\mathbb{S}^1 \times \mathbb{S}^1} \omega,
\end{equation*}
where $\phi^* \omega$ is defined as in \cite[p.166]{Pollack}. Let $\eta \in C_0^\infty(\mathbb{R}^3)$ be arbitrary. Take $\omega_\eta(u,v) := \eta(h(u,v)) \sqrt{\det{Dh^T Dh}} \;  du \wedge dv$. Then 
\begin{equation}\label{eq:B17}
\int_{\mathbb{S}^1 \times \mathbb{S}^1} \omega_\eta = \int_0^1 \int_0^{2\pi} \eta(h(u,v)) \sqrt{\mathrm{det}(Dh^TDh)} \; \mathrm{d}u \; \mathrm{d}v  = \int \eta \; \mathrm{d}h^*\mu_h,
\end{equation}
since $Dh^TDh$ is the first fundamental form of $(\mathbb{S}^1 \times \mathbb{S}^1, g_h)$.
 Note that by Lemma \ref{lem:breu} $f^* \mu_f$ coincides with $h^* \mu_h$ as  both  measures are $C_0(\mathbb{R}^n)'$-limits of $f_j^* \mu_{f_j}$. Hence by \eqref{eq:B17}
 \begin{equation}\label{eq:B18}
  \int \eta \; \mathrm{d}f^* \mu_f = \int \eta \; \mathrm{d}h^* \mu_h  = \int_{\mathbb{S}^1 \times \mathbb{S}^1 }  \omega_\eta 
 \end{equation} 
Using now that $f = h \circ \phi$ we can also compute $ \int \eta \; \mathrm{d}f^* \mu_f$ in another way. Since $s := \mathrm{sgn} \; \mathrm{det} D\phi$ is constant, by definition of $\phi^* \omega_\eta$  we obtain 
\begin{align*}
 \int \eta \; \mathrm{d}f^*\mu_f &  = \int_0^1 \int_0^{1} \eta(f(u,v)) \sqrt{\mathrm{det}(Df^TDf)}  \; \mathrm{d}u \; \mathrm{d}v  \\ &  = \int_0^1 \int_0^{1} \eta(h(\phi(u,v))) \sqrt{\mathrm{det}(Dh^TDh)} |\mathrm{det}(D\phi)|  \; \mathrm{d}u \; \mathrm{d}v
 \\ & = s \int_0^1 \int_0^{1}  \eta(h(\phi(u,v))) \sqrt{\mathrm{det}(Dh^TDh)} \mathrm{det}(D\phi)  \; \mathrm{d}u \; \mathrm{d}v
 \\ &  = s \int_{\mathbb{S}^1\times \mathbb{S}^1 } \phi^*\omega_\eta
= s \cdot \mathrm{deg}(\phi) \int_{{\mathbb{S}^1 \times \mathbb{S}^1}}  \omega_\eta .  
\end{align*}
This and \eqref{eq:B18} yields that $\mathrm{deg}(\phi) = \frac{1}{s} = \pm 1$. 

{\textbf{Conclusion}}\\  
The fact that $\mathrm{deg}(\phi) = \pm 1$, $\mathrm{sgn}(\mathrm{det}(D\phi))$ is constant together with \eqref{eq:degree} imply that $\phi^{-1}(\{y\})$ must be a singleton for any choice of $y \in \mathbb{S}^1 \times \mathbb{S}^1$. This proves the injectivity of $\phi$. Surjectivity follows directly from \cite[Chap. 3, Remark 1.5(2)]{Outerelo}. We finally end up with a surjective and injective local diffeomorphism. By this inverse function theorem, this is also a global diffeomorphism.}
\end{proof}

\begin{cor}\label{cor:unique}
If $(f_j)_{j = 1}^\infty: \mathbb{S}^1 \times \mathbb{S}^1 \rightarrow \mathbb{R}^3$ converges in $C^l$ to some $f : \mathbb{S}^1 \times \mathbb{S}^1 \rightarrow \mathbb{R}^3$ and also to some $h :  \mathbb{S}^1 \times \mathbb{S}^1 \rightarrow \mathbb{R}^3$. Then there exists a $C^l$ diffeomorphism $\phi : \mathbb{S}^1 \times \mathbb{S}^1 \rightarrow \mathbb{S}^1 \times \mathbb{S}^1$ such that $f = h \circ \phi$. 
\end{cor}
\begin{proof}
If $f_j$ converges to $h$ in $C^l$ then by Proposition \ref{prop:propsmoconalt}  there exists a sequence of diffeomorphisms $(\psi_j)_{j = 1}^\infty$ of $\mathbb{S}^1 \times \mathbb{S}^1$ such that $f_j \circ \psi_j$ converges to $h$ classically in $C^l$. Since (non-classical) $C^l$ convergence is not affected by reparametrizations, we infer that {also $f_j \circ \psi_j$ converges to $f$ in $C^l$. By Proposition \ref{prop:fastunique} applied to $f_j \circ \psi_j$ we infer that $f = h \circ \phi$ for a $C^l$-diffeomorphism $\phi$ of $\mathbb{S}^1 \times \mathbb{S}^1$.} 
\end{proof}

%%%%%%%%%%%%%%%%%%%%%%%%%%%%%%%%%%%%
%%%%%%%%%%%%%%%%%%%%%%%%%%%%%%%%%%%%
\section{On the Willmore flow}\label{Apptech}
%%%%%%%%%%%%%%%%%%%%%%%%%%%%%%%%%%%%
%%%%%%%%%%%%%%%%%%%%%%%%%%%%%%%%%%%%

Here we mention some previous results on the Willmore flow, which we will use. Since we need the precise formulations and constants we state them here for the readers convenience. We start with a short time existence and uniqueness result. We remark that this result is not the only short time existence result in the literature (cf. eg. \cite{simonett2001}), but it is the most useful for the formulation we use. 
\begin{theorem}\cite[Thm 1.2]{KS1}\label{thm:C1}
Suppose that $f_0 : \Sigma \rightarrow \mathbb{R}^n$ is a smooth immersion. Then there exist constants $\varepsilon_0> 0, c_0 < \infty$ that depend only on $n$ such that for all $\rho > 0 $ that satisfy 
\begin{equation*}
\sup_{x \in \mathbb{R}^n} \int_{ f_0^{-1}( B_\rho(x) ) } |A[f_0]|^2 d\mu_{f_0} \leq \varepsilon_0
\end{equation*}
then there exists a unique maximal smooth Willmore flow $(f(t))_{t\in [0,T)}$ starting at $f_0$ that satisfies $T \geq c_0 \rho^4$. {Moreover,} for all $m \geq 0$ there exists $C = C(n,m,f_0)$ such that 
\begin{equation}\label{eq:gradbound}
||(\nabla^\perp)^m A[f(t)]||_{L^\infty(\Sigma)} \leq C  \quad  \forall t \in [0, c_0 \rho^4]. 
\end{equation}
\end{theorem} 
Note that \eqref{eq:gradbound} is not in the statement of \cite[Thm 1.2]{KS1} but in its proof, see \cite[Eq. (4.27)]{KS1}. In fact the bound of the derivatives of the curvature are crucial in the proof of the short time existence theorem. In addition to bounds on the curvature one also needs a bound on the metric. Let us emphasize that this bound is (in finite time) implied by the curvature bounds as part of a more general result, see \cite[Lemma 14.2]{Hamilton}. Once short time existence is shown one can look at long time existence. The most important blow up criterion obtained so far is the one { discussed in Theorem \ref{thm:KS} below. It says that if $T < \infty$ then the curvature has to concentrate.}%we mentioned in Theorem \ref{thm:KS} -- if $T < \infty$ then the curvature has to concentrate. 

One can ask what happens to other quantities once the curvature degenerates. By Simon's monotonicity formula, the `density' will not degenerate. Indeed, in \cite[Eq. (1.3)]{Simon}, a local bound for the surface measure is shown. A useful implication stated in \cite[Lemma 4.1]{KS2} is that there exists $c >0$ such that for all proper immersions $f: \Sigma \rightarrow \mathbb{R}^n$ {($\Sigma$ compact and without boundary)} one has
\begin{equation}\label{eq:C3}
\frac{\mu_f( f^{-1}(B_\rho(x_0)))}{\rho^2} \leq c  \mathcal{W}(f) \quad \forall \rho > 0, 
\end{equation}
where we further assume that $\Sigma$ is a torus so that its  Euler characteristic vanishes.

Up to this point, no examples of evolutions where the curvature degenerates are known, even though there exists one candidate {for this phenomenon,} %where it is likely, 
cf. \cite{Mayer01anumerical}.

Close to local minimizers curvature concentration cannot occur and one deduces convergence with the aid of a Lojasiewicz-Simon gradient inequality.  

\begin{theorem}\cite[Lemma 4.1]{CFS09}\label{thm:CFS}
Let $f_{W}:\Sigma \to \mathbb{R}^n$ be a Willmore immersion of a compact manifold $\Sigma$ without boundary, and let $k \in \mathbb{N}$, $\delta>0$. Then there exists $\varepsilon= \varepsilon(f_{W})>0$ such that the following is true: \\suppose that $(f(t))_{t\in [0,T)}$ is a Willmore flow of $\Sigma$ satisfying 
$$ \| f_0-f_{W}\|_{W^{2,2}\cap C^1} < \varepsilon, $$
and
\begin{equation}\label{eq:nowillmoreconjecture}
 \mathcal{W}(f(t)) \geq \mathcal{W}(f_{W}) \mbox{ whenever }\| f(t)\circ \Phi(t)-f_{W}\|_{C^k} \leq \delta \, ,
 \end{equation}
for some appropriate diffeomorphisms $\Phi(t): \Sigma \to \Sigma$.\\
Then this Willmore flow exists globally, that is, $T=\infty$, and converges, after reparametrization by appropriate diffeomorphisms $\tilde{\Phi}(t): \Sigma \to \Sigma$, smoothly to a Willmore immersion $f_{\infty}$. That is, 
$$ f(t)\circ \tilde{\Phi}(t) \to f_{\infty} \mbox{ as }t \to \infty. $$
Moreover, $\mathcal{W}(f_{\infty})= \mathcal{W}(f_W)$ and $\| f_0-f_{W}\|_{C^k} < \delta$. 
\end{theorem}

\begin{remark}\label{rem:appepslos}
Notice that $\varepsilon$ in the statement does not change if instead of $f_{W}$ one considers the translated Willmore surface $f_{W}+\bar{x}$ for $\bar{x} \in \mathbb{R}^{{n}}$. Indeed, if $f_0$ satisfies 
$$ \| f_0 - (f_{W}+\bar{x})\|_{W^{2,2}\cap C^1} < \varepsilon= \varepsilon(f_W),$$
then clearly $f_0-\bar{x}$ satisfies the assumptions on the initial datum stated in \ref{thm:CFS} so that the corresponding Willmore flow $\tilde{f}(t)$ converges. Due to the uniqueness of the solution for the Willmore flow, $ \tilde{f}(t)=f(t)-\bar{x}$ with $f(t)$ the solution of the Willmore flow which starts in $f_0$. Hence, also $f(t)$ converges.
\end{remark}

\begin{remark}\label{rem:convsecFF}
We also remark that in case that the Willmore flow converges in $C^k$ for all $k$ one obtains uniform bounds on all derivatives of the second fundamental form, i.e. for all $m \in \mathbb{N}_0$ there exists $C = C(m, f_0)$ such that 
\begin{equation*}
||(\nabla^\perp)^m A [f(t) ]||_{L^\infty} \leq C  \quad \forall t \in [0,\infty). 
\end{equation*}
\end{remark}

Not every evolution of the Willmore flow is convergent. What one can however always obtain is a \emph{Willmore concentration limit} of appropriate parabolic rescalings. Below we will introduce the Willmore concentration limit rigorously since we need to examine it for the proof of Theorem \ref{thm:THM22NEW}.  

\begin{theorem}[{Willmore Concentration Limit, \cite[Sec. 4]{KS2}}] \label{thm:KS}
Let $\Sigma$ be a compact two-dimensional manifold {without boundary} 
and let $f \colon [0,T)\times \Sigma\to \mathbb{R}^{{n}}$ be immersions evolving by the Willmore flow with initial datum $f_0$. 
{Let $\varepsilon_0>0$ and $c_0$ be defined as in Theorem \ref{thm:C1}.} 

Then for each sequence $(t_j)_{j = 1}^\infty \nearrow T$ there exist $(x_j)_{j = 1}^\infty \subset \mathbb{R}^{{n}}$, $(r_j)_{j = 1}^\infty \subset (0,\infty) $ {(defined as in \eqref{eq:radii})} and $c_0 > 0 $ such that 
\begin{equation}\label{eq:tjrj}
 t_j + c_0r_j^4 < T \quad \text{for all } j \in \mathbb{N}
\end{equation}
 and
\begin{equation}\label{eq:deftildef}
\tilde{f}_j := \frac{1}{r_j} \left( f(t_j + c_0 r_j^4,\cdot ) - x_j  \right) \colon \Sigma \rightarrow \mathbb{R}^{{n}}
\end{equation}
converges smoothly on compact subsets of $\mathbb{R}^n$ to a {proper} Willmore immersion $\widehat{f}: \widehat{\Sigma} \rightarrow \mathbb{R}^{{n}}$, where $\widehat{\Sigma} \neq \emptyset$ is a smooth two-dimensional manifold without boundary. Moreover 
\begin{equation}   \label{eq:concentration}
\liminf_{j \rightarrow \infty } \int_{B_j} |A(t_j + c_0 r_j^4)|^2 \diff \mu_{g(t_j + c_0 r_j^4)} > 0,
\end{equation}
where $B_j=(f(t_j + c_0 r_j^4))^{-1}(\overline{B_{r_j}(x_j)})$.
\end{theorem}

Now we are finally ready to prove Theorem \ref{thm:THM22NEW}. 
\begin{proof}[Proof of Theorem \ref{thm:THM22NEW}] {The first part of the statement follows from \eqref{eq:tjrj}.} From Theorem \ref{thm:KS} it follows that there exists a sequence $(x_j)_{j \in \mathbb{N}} {\subset \mathbb{R}^n} $ and a proper Willmore immersion $\widehat{f}: \widehat{\Sigma} \rightarrow \mathbb{R}^{{n}}$ such that 
\begin{equation}\label{eq:433}
\tilde{f}_{j,c_0} - \frac{x_j}{r_j}  \rightarrow \widehat{f},
\end{equation}
smoothly as $j \rightarrow \infty$. Now we examine the asymptotics of $(r_j)_{j \in \mathbb{N}}$ \\
If there exists a subsequence of the radii $r_j$ that tends to zero or infinity. By \cite[Thm 1.1]{CFS09}, $\widehat{\Sigma}$ is not compact. In particular $\mathrm{diam}(\widehat{f}(\widehat{\Sigma}))= \infty$ since otherwise $\widehat{f}(\widehat{\Sigma})$ lies in a compact set of $\mathbb{R}^{{n}}$ which is a contradiction to the properness of {$\widehat{f}$}. By lower semincontinuity of the diameter, cf. Lemma \ref{lem:appdiam}, we infer 
\begin{equation*}
\infty = \mathrm{diam}(\widehat{f}(\widehat{\Sigma})) \leq \liminf_{j \rightarrow \infty} \mathrm{diam} \left( \tilde{f}_{j,c_0} - \frac{x_j}{r_j} \right) =  \liminf_{j \rightarrow \infty} \mathrm{diam} ( \tilde{f}_{j,c_0} ) 
\end{equation*}
Hence we have shown that (2) occurs. 

Suppose on contrary that $(r_j)_{j \in \mathbb{N}}$ has no subsequence that tends to zero or infinity. Then there exists $\delta > 0 $ such that $\delta < r_j < \frac{1}{\delta}$ for all $j \in \mathbb{N}$ and Case (1) occurs. {Necessarily from \eqref{eq:tjrj} we see that $T=\infty$.} 

It remains to show that a bound on the diameter ensures {full convergence to a Willmore immersion}. Suppose therefore that $\mathrm{diam}(\tilde{f}_{j,c_0}) \leq M$ for all $j \in \mathbb{N}$. Note that then  - once again by lower semicontinuity, cf. Lemma \ref{lem:appdiam},
\begin{equation*}
\mathrm{diam}(\widehat{f}(\widehat{\Sigma})) \leq \liminf_{j \rightarrow \infty} \mathrm{diam} \left( f_{j,c_0} - \frac{x_j}{r_j} \right) = \liminf_{j \rightarrow \infty} \mathrm{diam} \left( f_{j,c_0}\right) \leq M. 
\end{equation*}
Since $\widehat{f}$ is proper this ensures that {$\widehat{\Sigma}$}  %$\widehat{f}$ 
is compact. By \cite[Lemma 4.3]{KS2} we infer that $\widehat{\Sigma} = \mathbb{S}^1 \times \mathbb{S}^1$ and the convergence in \eqref{eq:433} is actually convergence in $C^k$ for all $k \in \mathbb{N}$. Now we define 
\begin{equation*}
\widetilde{f}_j: [0,c_0] \times \mathbb{S}^1 \times \mathbb{S}^1 \rightarrow \mathbb{R}^{{n}}, \quad \widetilde{f}_j(s) := \frac{f(t_j + s r_j^4)}{r_j}.
\end{equation*}
Note that by scaling properties of the Willmore gradient $\widetilde{f}_j$ %: [0,c_0] \times \mathbb{S}^1 \times \mathbb{S}^1 \rightarrow \mathbb{R}^3$ 
solves the Willmore flow equation. By \eqref{eq:433} we can now fix $j_0 \in \mathbb{N}$ and a smooth diffeomorphism $\Phi :  \mathbb{S}^1 \times \mathbb{S}^1 \rightarrow \mathbb{S}^1 \times \mathbb{S}^1$ such that 
\begin{equation}\label{eq:4.7}
\big\Vert \tilde{f}_{j_0,c_0} \circ \Phi - \frac{x_{j_0}}{r_{j_0}} - \widehat{f} \big\Vert_{C^2} < \epsilon = \epsilon(\widehat{f}),
\end{equation}
where $\epsilon(\widehat{f})$ is chosen as in Theorem \ref{thm:CFS}. By Remark \ref{rem:appepslos} we also have $\epsilon(\widehat{f}) = \epsilon \big( \widehat{f} + \frac{x_{j_0}}{r_{j_0}} \big)$. We infer by Theorem \ref{thm:CFS} that the Willmore flow starting at $\widetilde{f}_{j_0, c_0} \circ \Phi$ exists globally and converges { (up to reparametrization)} to a Willmore immersion $f_\infty : \mathbb{S}^1 \times \mathbb{S}^1 \rightarrow \mathbb{R}^{{n}}$. By geometric uniqueness of Willmore evolutions we infer that $\widetilde{f}_{j_0} \circ \Phi$, {first defined on $[0,c_0]$,} % : [0,c_0] \times \mathbb{S}^1 \times \mathbb{S}^1$ 
extends to a global evolution, i.e. defined on $[0,\infty)$, and converges (up to reparametrization) to $f_\infty$. Again by geometric uniqueness we infer that $\widetilde{f}_{j_0}$ extends to a global evolution converging (up to reparametrization) to $f_\infty \circ \Phi^{-1}$. 
 Using scaling properties of the Willmore flow we infer that $f$ extends to a global evolution by Willmore flow that converges to  $r_{j_0} f_\infty$, which is again a Willmore immersion. 
 
 To show the last sentence of the claim we first observe that a uniform bound on the diameter implies that case (2) may not occur, in particular $r_j \in (\delta, \frac{1}{\delta})$ for some $\delta > 0 $. Then the fact that $t_j + c_0 r_j^4 < T$ for all $j$ and $t_j \rightarrow T$ implies that $ T = \infty$. Convergence follows then according to case (1) with the diameter bound. 
\end{proof}

With this theorem we have proved that boundedness of $\mathrm{diam}( \tilde{f}_{j,c_0} )$ implies convergence. The fact that $\tilde{f}_{j,c_0}$ need information about $f(t_j + c_0 r_j^4)$ and not just about $f(t_j)$ adds a technical difficulty -- the time shift might cause that geometric quantities degenerate. Luckily, the diameter is not so much affected by (bounded) time shifts, as we shall see in the following

\begin{lemma}[Evolution of Diameter and Area] \label{lem:diamflow}
Suppose that $f :  [0,T) \times \Sigma \rightarrow \mathbb{R}^{{n}}$  is a maximal evolution by Willmore flow. Then there exist constants $C_1 = C_1(\mathcal{W}(f(0))), C_2 = C_2( \mathcal{W}(f(0)))$ depending monotonically on $\mathcal{W}(f(0))$ such that 
\begin{equation}\label{eq:totalarea}
\mu_{g_{f(t)}} ( \Sigma)  \leq \mu_{g_{f(0)}} ( \Sigma) + C_1(\mathcal{W}(f(0))) t^\frac{1}{2} 
\end{equation}
and 
\begin{equation*}
\mathrm{diam}(f(t)(\Sigma)) \leq C_2(\mathcal{W}(f(0))) ( \mathrm{diam}(f(0)(\Sigma)) + t^\frac{1}{4} ) .
\end{equation*}
\end{lemma} 
\begin{proof}
First we remark that, {since the Willmore flow is a gradient flow,} for all  $s \geq 0$
\begin{equation}\label{eq:aiu}
\int_0^s  \int_\Sigma |\partial_t f(t)|^2 d\mu_{g_{f(t_j)}} =  \mathcal{W}(f(0))- \mathcal{W}(f(s)) {\leq  \mathcal{W}(f(0))}.  
\end{equation}
By \cite[Eq.(2.16)]{KS1} we have
 \begin{align*}
 \left\vert \frac{d}{dt} \mu_{g_{f(t)} } ( \Sigma) \right\vert  & =  \left\vert \int_\Sigma \langle \vec{H}[f(t)], \partial_t f(t) \rangle \; \mathrm{d}\mu_{g_{f(t)}} \right\vert  \\ & \leq \left( \int_\Sigma  |\vec{H}[f(t)]|^2 \; \mathrm{d}\mu_{g_{f(t)}}  \right)^\frac{1}{2}\left( \int_\Sigma  |\partial_t f(t)|^2  \; \mathrm{d}\mu_{g_{f(t)}}  \right)^\frac{1}{2}
 \\ & \leq 2 \sqrt{\mathcal{W}(f(t))}\left( \int_\Sigma  |\partial_t f(t)|^2  \; \mathrm{d}\mu_{g_{f(t)}}  \right)^\frac{1}{2}
 \end{align*}
 Integrating {with respect to $t$ and since $t \mapsto \mathcal{W}(f(t))$ is decreasing} we obtain 
 \begin{align*}
| \mu_{g_{f(s)}}( \Sigma) - \mu_{g_{f(0)}} ( \Sigma) | & \leq 2 \sqrt{\mathcal{W}(f(0))} \int_0^s \left( \int_\Sigma  |\partial_t f(t)|^2  \; \mathrm{d}\mu_{g_{f(t)}}  \right)^\frac{1}{2} \; \mathrm{d}t 
\\ & \leq  2 \sqrt{\mathcal{W}(f(0))} s^\frac{1}{2} \left( \int_0^s \int_\Sigma  |\partial_t f(t)|^2  \; \mathrm{d}\mu_{g_{f(t)}} \; \mathrm{d}t \right)^\frac{1}{2}
\\ & \leq 2 \mathcal{W}(f(0)) s^\frac{1}{2},
 \end{align*}
{using \eqref{eq:aiu} in the last step.} The estimate in \eqref{eq:totalarea} %for the total area 
follows if we choose {$C_1(W) = 2 {\mathcal{W}(f(0))}$}. Next we use a generalization of \cite[Lemma 1.1]{Simon} (cf. the following Lemma) for immersed surfaces to obtain that there exists $C_S  > 0$ such that $\mathrm{diam}(f(\Sigma))^2 \leq C_S \mathcal{W}(f) \mu_{g_f}( \Sigma)$. Using this, {\eqref{eq:totalarea}} and Lemma \ref{lem:simon} we obtain 
%\todo{Possibly need to get rid of the square-root in the computations below: would change also the constant at the end and make it better.}
 \begin{align*}
 \mathrm{diam} (f(t)(\Sigma))^2 & \leq C_S \mathcal{W}(f(t)) \mu_{g_{f(t)}}(\Sigma)  \leq C_S \mathcal{W}(f(0)) ( \mu_{g_{f(0)}}( \Sigma ) + 2 {\mathcal{W}(f(0))} t^\frac{1}{2} ) 
 \\ & \leq C_S \mathcal{W}(f(0)) ( \mathcal{W}(f(0)) \mathrm{diam}(f(0))^2 + 2 {W(f(0))} t^\frac{1}{2} ) .
 \\ & \leq C_S \mathcal{W}(f(0))^2 ( \mathrm{diam}(f(0))^2 + 2t^\frac{1}{2} ) 
 \\ & \leq 2C_S \mathcal{W}(f(0))^2 ( \mathrm{diam}(f(0)) + t^\frac{1}{4} )^2.
 \end{align*}
 The choice of $C_2(W) := 2 C_S W^2 $ does the job. 
\end{proof}

{In this proof we have used the following Lemma, which generalizes \cite[Lemma 1.1]{Simon}.
\begin{lemma}[{cf. \cite[Lemma 1.1]{Simon}}] \label{lem:simon}
 There exists $C_S= C_S(n) > 0$ such that for all immersions $f: \Sigma \rightarrow \mathbb{R}^n$  of a compact connected $2D$-manifold without boundary $\Sigma$ one has 
 \begin{equation*}
     \frac{\mu_{g_f}(\Sigma)}{\mathcal{W}(f)} \leq \mathrm{diam}(f(\Sigma))^2 \leq C_S(n) {\mu_{g_f}(\Sigma) \mathcal{W}(f)} .
 \end{equation*}
 \end{lemma}
 \begin{proof}
 Let $\Sigma$ be as in the statement.
 By \cite[Lemma 1.1]{Simon} we infer that for all $n \in \mathbb{N}$ there exists $c(n) > 0$ such that for all embeddings $f: \Sigma \rightarrow \mathbb{R}^n$ one has
     \begin{equation}\label{eq:simonabsch}
     \frac{\mu_{g_f}(\Sigma)}{\mathcal{W}(f)} \leq \mathrm{diam}(f(\Sigma))^2 \leq c(n) {\mu_{g_f}(\Sigma) \mathcal{W}(f)} .
 \end{equation}
 We need to generalize this result to immersions.
 Let $N \in \mathbb{N}$ be such that each smooth two-dimensional manifold can be smoothly embedded into $\mathbb{R}^{N}$. Such constant $N$ exists due to Nash's embedding theorem (or alternatively one can derive $N= 4$ explicitly using a handle decomposition). We will show that the desired estimate is satisfied with the constant $C_S(n) := c(n+N)$. To this end let $f : \Sigma \rightarrow \mathbb{R}^n$ be an immersion and $\iota : \Sigma \rightarrow \mathbb{R}^N$ be an embedding. For fixed $\epsilon > 0$ define $f_\epsilon: \Sigma \rightarrow \mathbb{R}^{n+N}$ via $f_\epsilon(p) := (f(p), \epsilon \iota(p))^T$. It is easy to check that $f_\epsilon$ is an embedding.  We infer by \eqref{eq:simonabsch} that 
 \begin{equation}\label{eq:approxsimon}
     \frac{\mu_{g_{f_\epsilon}}(\Sigma)}{\mathcal{W}(f_\epsilon)} \leq \mathrm{diam}(f_\epsilon(\Sigma))^2 \leq c(n+N) {\mu_{g_{f_\epsilon}}(\Sigma) \mathcal{W}(f_\epsilon)}. 
 \end{equation}
 Next we pass to the limit as $\epsilon \rightarrow 0$. 
 First we examine the diameter. Note that for all $x,y \in \Sigma$ one has 
 \begin{equation*}
     |f_\epsilon(x) - f_\epsilon(y)|^2 = |f(x) - f(y)|^2 + \epsilon^2|\iota(x) - \iota(y)|^2.
 \end{equation*}
 From this one easily infers 
 \begin{equation*}
     \mathrm{diam}(f(\Sigma))^2 \leq \mathrm{diam}(f_\epsilon(\Sigma))^2  \leq \mathrm{diam}(f(\Sigma))^2 + \epsilon^2 \mathrm{diam}( \iota (\Sigma)).
 \end{equation*}
 Since $\Sigma$ is compact we find that $\mathrm{diam}(\iota(\Sigma))< \infty$. Hence
 \begin{equation*}
     \lim_{\epsilon \rightarrow 0} \mathrm{diam}(f_\epsilon(\Sigma)) = \mathrm{diam}(f(\Sigma))
 \end{equation*}
 One readily checks that $f_\epsilon \rightarrow (f,0)$ in $C^k$ for all $k$. From Lemma \ref{lem:appsemi} one infers then that 
 $
     \lim_{\epsilon \rightarrow 0} \mathcal{W}(f_\epsilon) = \mathcal{W}((f,0)) = \mathcal{W}(f).
$
 That $\mathcal{W}((f,0)) = \mathcal{W}(f)$ can easily be checked since $A[(f,0))](X,Y) = D^2(f,0)(X,Y) = (D^2f(X,Y),0)$, where the last identity is due to the fact that $D$ is defined componentwise, cf. \eqref{eq:DXDEF}. Using methods similar to Lemma \ref{lem:appsemi} one can also check 
 $
     \lim_{\epsilon \rightarrow 0 } \mu_{g_{f_\epsilon}}(\Sigma ) = \mu_{g_{(f,0)}}(\Sigma) = \mu_{g_f}(\Sigma).
$
 This being shown, the claim follows from \eqref{eq:approxsimon} letting $\epsilon \rightarrow 0$.
 \end{proof}
}

\bibliography{Lib}

\end{document}